\documentclass[12pt,a4paper]{amsart}
\usepackage[marginratio=1:1]{geometry}
\usepackage[T1]{fontenc}
\usepackage[utf8]{inputenc}
\usepackage[american]{babel}
\usepackage[draft=false]{hyperref}
\hypersetup{
	colorlinks,
	linkcolor={red!50!black},
	citecolor={blue!50!black},
	urlcolor={blue!80!black}
}
\usepackage{enumitem}
\usepackage{graphicx}
\usepackage{amsfonts}
\usepackage{amssymb}
\usepackage{tikz-cd}
\usepackage{float}
\makeatletter
\providecommand*{\twoheadrightarrowfill@}{%
	\arrowfill@\relbar\relbar\twoheadrightarrow
}
\providecommand*{\twoheadleftarrowfill@}{%
	\arrowfill@\twoheadleftarrow\relbar\relbar
}
\providecommand*{\xtwoheadrightarrow}[2][]{%
	\ext@arrow 0579\twoheadrightarrowfill@{#1}{#2}%
}
\providecommand*{\xtwoheadleftarrow}[2][]{%
	\ext@arrow 5097\twoheadleftarrowfill@{#1}{#2}%
}
\makeatother
\newcommand{\C}{{\mathfrak C}}
\newcommand{\M}{{\mathcal M}}

\newcommand{\R}{{\mathbb{R}}}

\newcommand{\fcolon}{\colon}
\DeclareMathOperator{\tp}{{tp}}
\DeclareMathOperator{\Th}{{Th}}
\DeclareMathOperator{\gal}{{Gal}}
\DeclareMathOperator{\cl}{{cl}}

\DeclareMathOperator{\id}{{id}}
\DeclareMathOperator{\aut}{{Aut}}
\DeclareMathOperator{\autf}{{Autf}}

\DeclareMathOperator{\ext}{{ext}}

\DeclareMathOperator{\Sl}{{SL}}

\DeclareMathOperator{\acl}{{acl}}

\DeclareMathOperator{\fix}{{fix}}
\DeclareMathOperator{\im}{{Im}}
\DeclareMathOperator{\pr}{{proj}}
\DeclareMathOperator{\stab}{{Stab}}
\DeclareMathOperator{\alt}{{Alt}}
\DeclareMathOperator{\qftp}{{qftp}}
\DeclareMathOperator{\val}{{val}}
\DeclareMathOperator{\sym}{{Sym}}

\newtheorem{thm}{Theorem}[section]

\newtheorem{ques}[thm]{Question}
\newtheorem{problem}[thm]{Problem}
\newtheorem{lem}[thm]{Lemma}
\newtheorem{fct}[thm]{Fact}
\newtheorem{cor}[thm]{Corollary}
\newtheorem{prop}[thm]{Proposition}
\theoremstyle{remark}
\newtheorem{rem}[thm]{Remark}
\theoremstyle{definition}
\newtheorem{dfn}[thm]{Definition}

\newtheorem{clm*}{Claim}
\newtheorem{ex}[thm]{Example}
\newcounter{claimcounter}[thm]

\makeatletter
\newtheorem*{rep@theorem}{\rep@title}
\newcommand{\newreptheorem}[2]{%
\newenvironment{rep#1}[1]{%
 \def\rep@title{#2 \ref{##1}}%
 \begin{rep@theorem}}%
 {\end{rep@theorem}}}
\makeatother
\newreptheorem{theorem}{Theorem}
\title[Boundedness and absoluteness of some dynamical invariants]{Boundedness and absoluteness of some dynamical invariants in model theory}
\author{Krzysztof Krupi\'nski}
\email[K.\ Krupi\'{n}ski]{kkrup@math.uni.wroc.pl}
\address[K.\ Krupi\'nski]{
	Instytut Matematyczny, Uniwersytet Wroc\l awski\\
	pl. Grunwaldzki 2/4\\
	50-384 Wroc\l aw, Poland}
\thanks{The first and second authors are supported by the Narodowe Centrum Nauki grant no. 2015/19/B/ST1/01151; the first author is also supported by the Narodowe Centrum Nauki grant no. 2016/22/E/ST1/00450}
\author{Ludomir Newelski}
\email[L.\ Newelski]{newelski@math.uni.wroc.pl}
\address[L.\ Newelski]{
	Instytut Matematyczny, Uniwersytet Wroc\l awski\\
	pl. Grunwaldzki 2/4\\
	50-384 Wroc\l aw, Poland}
\thanks{}
\author{Pierre Simon}
\email[P.\ Simon]{simon@math.berkeley.edu }
\address[P.\ Simon]{University of California, Berkeley\\
Mathematics Department, Evans Hall\\
Berkeley, CA, 94720-3840, USA}
\thanks{The third author was partially supported by ValCoMo (ANR-13-BS01-0006).}
\keywords{Group of automorphisms, Ellis group, minimal flow, boundedness, absoluteness}
\subjclass[2010]{03C45, 54H20, 37B05}
\date{}
\begin{document}
	
	\begin{abstract}
Let $\C$ be a monster model of an {\em arbitrary} theory $T$, let $\bar \alpha$ be any (possibly infinite) tuple of bounded length of elements of $\C$, and let $\bar c$ be an enumeration of all elements of $\C$ (so a tuple of unbounded length). By $S_{\bar \alpha}(\C)$ we denote the compact space of all complete types over $\C$ extending $\tp(\bar \alpha/\emptyset)$, and $S_{\bar c}(\C)$ is defined analogously. Then $S_{\bar \alpha}(\C)$ and $S_{\bar c}(\C)$ are naturally $\aut(\C)$-flows (even $\aut(\C)$-ambits). We show that the Ellis groups of both these flows are of {\em bounded} size (i.e. smaller than the degree of saturation of $\C$), providing an explicit bound on this size. Next, we prove that these Ellis groups do not depend (as groups equipped with the so-called $\tau$-topology) on the choice of the monster model $\C$; thus, we say that these Ellis groups are {\em absolute}. We also study minimal left ideals (equivalently subflows) of the Ellis semigroups of the flows $S_{\bar \alpha}(\C)$ and $S_{\bar c}(\C)$. We give an example of a NIP theory in which the minimal left ideals are of unbounded size. Then we show that in each of these two cases, boundedness of a minimal left ideal (equivalently, of all the minimal left ideals) is an absolute property (i.e. it does not depend on the choice of $\C$) and that whenever such an ideal is bounded, then in some sense its isomorphism type is also absolute.

Under the assumption that $T$ has NIP, we give characterizations (in various terms) of when a minimal left ideal of the Ellis semigroup of $S_{\bar c}(\C)$ is bounded. Then we adapt the proof of \cite[Theorem 5.7]{ChSi} to show that whenever such an ideal is bounded, a certain natural epimorphism (described in \cite{KrPiRz}) from the Ellis group of the flow $S_{\bar c}(\C)$ to the Kim-Pillay Galois group $\gal_{KP}(T)$  is an isomorphism (in particular, $T$ is G-compact). 
We also obtain some variants of these results, formulate some questions, and explain differences (providing a few counter-examples) which occur when the flow $S_{\bar c}(\C)$ is replaced by $S_{\bar \alpha}(\C)$.
	\end{abstract}

	\maketitle
	
	\section{Introduction}
Some methods and ideas of topological dynamics were introduced to
model theory by the second author in \cite{Ne1,Ne2,Ne3}. For a group $G$ definable in a first order structure $M$ the fundamental object in this approach is the $G$-flow $S_G(M)$ (the space of complete types over $M$ concentrated on $G$) or rather $S_{G,\ext}(M)$ (the space of all external types over $M$ concentrated on $G$). Then the Ellis semigroup $E(S_{G,\ext}(M))$ of this flow, minimal left ideals of this semigroup, and the Ellis group play the key role in this approach. The motivation for these considerations is the fact that various notions and ideas from topological dynamics lead to new interesting objects and phenomena in model theory. In particular, this allows us to extend some aspects of the theory of stable groups to a much more general context. Another context to apply ideas of topological dynamics to model theory is via looking at the action of the group $\aut(\C)$ of automorphisms of a monster model $\C$ of a given (arbitrary) theory $T$ on various spaces of types over $\C$. Some aspects of this approach appeared already in \cite{NePe}. More recently, this approach was developed and applied in \cite{KrPiRz} to prove very general theorems about spaces of arbitrary strong types (i.e. spaces of classes of bounded invariant equivalence relations refining type), answering in particular some questions from earlier papers. So this is an example where the topological dynamics approach to model theory not only leads to a development of a new theory, but can also be used to solve existing problems in model theory which were not accessible by previously known methods.

From the model-theoretic perspective, an important kind of question
(raised and considered by the second author and later also by
Chernikov and the third author in \cite{ChSi}) is to what extent various dynamical invariants have a model-theoretic nature, e.g. in the case of a group $G$ definable in $M$, what are the relationships between the minimal left ideals of the Ellis semigroups, or between the Ellis groups, of the flow $S_{G,\ext}(M)$ when the ground model $M$ varies: are some of these objects invariant under changing $M$, do there exist epimorphisms from objects computed for a bigger model onto the corresponding objects computed for a smaller model? This turned out to be difficult to settle in general and there are only some partial results in \cite{Ne2,Ne3}; a quite comprehensive understanding of these issues appears in \cite{ChSi} in the context of definably amenable groups in NIP theories, e.g. in this context the Ellis group does not depend on the choice of the ground model $M$ and is isomorphic to the so-called definable Bohr compactification of $G$. 

In this paper, we study these kind of questions for the second context mentioned above, namely for the group $\aut(\C)$ acting on some spaces of complete types over $\C$, and, in contrast with the case of a definable group $G$, here we answer most of these questions in full generality. Some of these questions were formulated at the end of Section 2 of \cite{KrPiRz}.

In the rest of this introduction, we will state precisely in which questions we are interested in and what our main results are. For an exposition of basic notions and facts on topological dynamics which are essential in this paper the reader is referred to \cite[Subsection 1.1]{KrPiRz}. 

Let $\C$ be a monster model of a theory $T$, and let $\bar c$ be an enumeration of all elements of $\C$. Recall that a monster model is a $\kappa$-saturated and strongly $\kappa$-homogeneous model for a ``big enough'' strong limit cardinal $\kappa$. A cardinal is said to be bounded (with respect to $\C$) if it is less than $\kappa$. 

Then $S_{\bar c}(\C):=\{ \tp(\bar a/\C): \bar a \equiv \bar c\}$ is an $\aut(\C)$-ambit. By $EL$ we denote the Ellis semigroup of this flow. Let ${\mathcal M}$ be a minimal left ideal of $EL$ and $u \in {\mathcal M}$ an idempotent; so $u{\mathcal M}$ is the Ellis group of the flow $S_{\bar c}(\C)$. This group played a fundamental role in \cite{KrPiRz}, where a continuous epimorphism $f \colon u{\mathcal M} \to \gal_{L}(T)$ (where $\gal_L(T)$ is the Lascar Galois group of $T$) was found and used to understand the descriptive set-theoretic complexity of $\gal_L(T)$ and of spaces of arbitrary strong types.
The group $u{\mathcal M}$ can be equipped with the so-called $\tau$-topology, and then $u{\mathcal M}/H(u{\mathcal M})$ is a compact (Hausdorff) topological group, where $H(u{\mathcal M})$ is the intersection of the $\tau$-closures of the $\tau$-neighborhoods of $u$. In fact, an essential point in \cite{KrPiRz} was that $f$ factors through $H(u{\mathcal M})$.

We will say that an object or a property defined in terms of $\C$ is {\em absolute} if it does not depend (up to isomorphism, if it makes sense) on the choice of $\C$. One can think that absoluteness means that an object or a property is a model-theoretic invariant of the theory in question.
In Section \ref{section 3}, we study Ellis groups.

\begin{ques}\label{question: main question 1}
i) Does the Ellis group $u{\mathcal M}$ of the flow $S_{\bar c}(\C)$ have bounded size?\\
ii) Is this Ellis group independent of the choice of $\C$ as a group equipped with the $\tau$-topology?\\
iii) Is the compact topological group $u{\mathcal M}/H(u{\mathcal M})$ independent (as a topological group) of the choice of $\C$? 
\end{ques}

The positive answer to (ii) clearly implies that the answer to (iii) is also positive.
One can formulate the same questions also with $\bar c$ replaced by a
tuple $\bar \alpha$ of elements of $\C$, of bounded length (we call
such a tuple $\alpha$ {\em short}). We will prove that the answers to all these questions are positive (with no assumptions on the theory $T$). In fact, we will deduce it from a more general theorem. Namely, consider any product $S$  of 
sorts $S_i$, $i \in I$ (possibly with unboundedly many factors, with repetitions allowed), and let $\bar x = (x_i)_{i \in I}$ be the corresponding tuple of variables (i.e. $x_i$ is from the sort $S_i$). By $S_S(\C)$ or $S_{\bar x}(\C)$ we denote the space of all complete types over $\C$ of tuples from $S$. Then $S_S(\C)$ is an $\aut(\C)$-flow. 
In this paper, by a  $\emptyset$-type-definable subset $X$ of $S$ we will mean a partial $\emptyset$-type in the variables $\bar x$, and when $S$ has only boundedly many factors, $X$ will be freely identified with the actual subset $X(\C)$ of $S$ computed in $\C$. The space $S_X(\C)$ of complete types over $\C$ concentrated on $X$ is an $\aut(\C)$-subflow of the flow $S_S(\C)$. Of course, it makes sense to consider $S_X(\C')$ for an arbitrary monster model $\C'$.    One of our main results is

\begin{thm}\label{theorem: main theorem 1}
Let $X$ be any $\emptyset$-type-definable subset of $S$. Then the Ellis group of the $\aut(\C)$-flow $S_X(\C)$ is of bounded size and does not depend (as a group equipped with the $\tau$-topology) on the choice of the monster model $\C$.
\end{thm}

\noindent
Moreover, we will see that 
$\beth_5(|T|)$ is an absolute bound on the size of this Ellis group. (Where $\beth_0(\kappa)=\kappa$ and $\beth_{n+1}(\kappa)=2^{\beth_n(\kappa)}$.)
From this theorem, we will easily deduce that the answer to Question \ref{question: main question 1}, and also to its counterpart for $\bar c$ replaced by a short tuple $\bar \alpha$, is positive, and, moreover, that 
$\beth_5(|T|)$ is an absolute bound on the size of the Ellis group in both these cases.

Note that there is no counterpart of this behavior in the case of a
group $G$ definable in $M$, namely it may happen that the sizes of the
Ellis groups of the flows $S_{G,\ext}(M)$ are getting arbitrarily
large when $M$ varies. For example, it is the case whenever the
component ${G^*}^{00}$ does not exist. Indeed, in this case the groups
$G^*/{G^*_M}^{00}$ are getting arbitrarily large when $M$ grows, and
by \cite{Ne1} they are homomorphic images of the Ellis groups of $S_{G,\ext}(M)$.

The fundamental tool in the proof Theorem \ref{theorem: main theorem 1} is the notion of content of a sequence of types introduced in Section \ref{section2.5}. This gives a model-theoretic characterization of when a type is in the orbit closure of another one. It also allows us to define an independence relation which makes sense in any theory. The properties of this relation will be studied in a future work.

In Section \ref{section 4}, we focus on minimal left ideals in Ellis semigroups. As above, let ${\mathcal M}$ be a minimal left ideal of the Ellis semigroup $EL$ of the flow $S_{\bar c}(\C)$. Recall that all minimal left ideals of $EL$ are isomorphic as $\aut(\C)$-flows (so they are of the same size), but they need not be isomorphic as semigroups. 

\begin{ques}\label{question: main question 2}
i) Is the size of ${\mathcal M}$ bounded?\\
ii) Is the property that ${\mathcal M}$ is of bounded size absolute?\\
iii) If ${\mathcal M}$ is of bounded size and the answer to (ii) is positive, what are the relationships between the minimal left ideals when $\C$ varies?
\end{ques}

As above, one has the same question for $\bar c$ replaced by any short tuple $\bar \alpha$. We give an easy example showing that the answer to (i) is negative (see Example \ref{example: unbounded minimal left ideal}). Then we answer positively  (ii) and give an answer to (iii) for both $\bar c$ and $\bar \alpha$. As above, we will deduce these things from the following, more general result.

\begin{thm}\label{theorem: main theorem 2}
Let $X$ be any $\emptyset$-type-definable subset of $S$.\\
i) The property that a minimal left ideal of the Ellis semigroup of the flow $S_X(\C)$ is of bounded size is absolute.\\
ii) Assume that a minimal left ideal of the Ellis semigroup of $S_X(\C)$ is of bounded size. Let $\C_1$ and $\C_2$ be two monster models. Then for every minimal left ideal $\mathcal{M}_{1}$ of the Ellis semigroup of the flow $S_X(\C_1)$ there exists a minimal left ideal ${\mathcal M}_{2}$ of the Ellis semigroup of the flow $S_X(\C_{2})$ which is isomorphic to ${\mathcal M}_{1}$ as a semigroup.
\end{thm}

\noindent
Moreover, we show that if a minimal left ideal of the Ellis semigroup of the flow $S_X(\C)$ is of bounded size, then this size is bounded by $\beth_3(|T|)$.

The above theorem together with Example \ref{example: unbounded minimal left ideal} leads to the following problem which we study mostly in Section \ref{section 6}. (In between,  in Section \ref{section 5}, we easily describe the minimal left ideals, the Ellis group, etc. in the case of stable theories.) 

\begin{problem}\label{problem: characterizations of boundedness of minimal left ideals}
i) Characterize when a minimal left ideal of the Ellis semigroup of the flow $S_X(\C)$ is of bounded size.\\
ii) Do the same for the flow $S_{\bar c}(\C)$.\\
iii) Do the same for the flow $S_{\bar \alpha}(\C)$, where $\bar \alpha$ is a short tuple from $\C$.
\end{problem}

We will easily see that that a solution of (i) would yield solutions of (ii) and (iii). As to (i), we observe in Section \ref{section 4} that such an ideal is of bounded size if and only if there is an element in the Ellis semigroup of $S_X(\C)$ which maps all types from $S_X(\C)$ to Lascar invariant types. Then, in Section \ref{section 6}, under the additional assumption that $T$ has NIP, we give several characterizations (in various terms) of when a minimal left ideal of the Ellis semigroup of the flow $S_{\bar c}(\C)$ is of bounded size (i.e. we solve Problem \ref{problem: characterizations of boundedness of minimal left ideals}(ii) for NIP theories). This is done in Theorem \ref{theorem: main characterization theorem under NIP}. A part of this theorem is the following.

\begin{prop}\label{proposition: characterization of boundedness in NIP theories}
Assume that $T$ has NIP. Then a minimal left ideal of the Ellis semigroup of the flow $S_{\bar c}(\C)$ is of bounded size if and only if $\emptyset$ is an extension base (i.e. every type over $\emptyset$ does not fork over $\emptyset$).
\end{prop}

\noindent
We also get some characterizations of boundedness of the minimal left ideals of the Ellis semigroup of  $S_X(\C)$, which lead us to natural questions (see Questions \ref{question: bounded = forking equals dividing} and \ref{question: characterization for bar alpha}).

In the second subsection of Section \ref{section 6}, we get a better bound on the size of the Ellis group of the flow $S_X(\C)$ under the NIP assumption, namely 
$\beth_3(|T|)$. In the case of the Ellis groups of the flows $S_{\bar c}(\C)$ and $S_{\bar \alpha}(\C)$, we get a yet smaller bound, namely ${2^{2^{|T|}}}$.

In the last subsection of Section \ref{section 6}, we recall from
\cite{KrPiRz} a natural epimorphism $F$ from the Ellis group of the
flow $S_{\bar c}(\C)$ to the Kim-Pillay Galois group $\gal_{KP}(T)$. We adapt the proof of Theorem  5.7 from \cite{ChSi} to show:

\begin{thm}\label{theorem theorem from ChSi}
Assume that $T$ has NIP. If a minimal  left ideal of the Ellis
semigroup of the flow $S_{\bar c}(\C)$ is of bounded size, then the
epimorphism $F$ mentioned above is an isomorphism.
\end{thm}

Since the map $F$ is the composition of some natural homomoprhism  from the Ellis group of the flow $S_{\bar c} (\C)$ to $\gal_L(T)$ with the natural map $\gal_L(T)\to \gal_{KP}(T)$, as an immediate corollary we get that under the assumtpions of
Theorem \ref{theorem theorem from ChSi}, $T$ is G-compact. Alternatively, G-compactness follows from Proposition \ref{proposition: characterization of boundedness in NIP theories} and \cite[Corollary 2.10]{HrPi}.
On the other hand, taking any non G-compact theory, although by Theorem \ref{theorem: main theorem 1} the Ellis group of the flow $S_{\bar c}(\C)$ is always bounded, we get that $F$ need not be an isomorphism (even in the NIP context, as there exist non G-compact NIP theories). This shows that the assumption that a minimal left ideal is of bounded size is essential in the last theorem. Using some non-trivial results from \cite{KrPiRz}, we can conclude more, namely that the natural epimorphism from the Ellis group onto $\gal_{L}(T)$ need not be an isomorphism, even in the NIP context.

We also describe an epimorphism from the Ellis group of the flow $S_{\bar \alpha}(\C)$ to the new Kim-Pillay Galois group introduced in \cite{DoKiLe} and denoted by $\gal_{KP}^{\fix,1}(p)$, where $p =\tp(\bar \alpha / \emptyset)$. This is a natural counterpart of the epimorphism considered in Theorem \ref{theorem theorem from ChSi}. However, we give an example showing that the obvious counterpart of Theorem \ref{theorem theorem from ChSi} does not hold for this new epimorphism, namely assuming that $T$ has NIP and  a minimal left ideal of the Ellis semigroup of the flow $S_{\bar \alpha}(\C)$ is of bounded size, the new epimorphism need not be an isomorphism.

\section{Some preliminaries}\label{section 1}

For the relevant notions and facts from topological dynamics the reader is referred to \cite[Subsection 1.1]{KrPiRz}. (We do not include it here, because this would be the exact copy.)
Let us only say that the Ellis semigroup of a flow $(G,X)$ will be
denoted by $EL(X)$, a minimal left ideal in $EL(X)$ will be usually
denoted by ${\mathcal M}$ (sometimes with some indexes) and
idempotents in ${\mathcal M}$ will be denoted by $u$ or $v$. We should
emphasize that all minimal left ideals (equivalently, minimal
subflows) of $EL(X)$ are isomorphic as $G$-flows \cite[Proposition
  I.2.5]{Gl} but not necessarily
as semigroups. The fact that they are isomorphic as $G$-flows implies that they are of the same size, and so, in the whole paper, the statement ``a minimal left ideal of a given Ellis semigroup is of bounded size'' is equivalent to the statement ``the minimal left ideals of the given Ellis semigroup are of bounded size''.

In this paper, we will often consider nets indexed by some formulas $\varphi(\bar x, \bar b)$. Formally, this means that on the set of formulas treated as elements of the Lidenbaum algebra (i.e. equivalent formals are identified) we take the natural directed order: $\varphi \leq \psi$ if and only if $\psi \vdash \varphi$. But throughout this paper, we will just talk about formulas, not always mentioning that we are working in the Lidenbaum algebra. The same remark concerns nets indexed by finite sequences of formulas (where the order is formally the product of orders on the Lindenbaum algebras in every coordinate). In fact, we could work just with formulas and with nets indexed by preorders, but we find more elegant to use orders.

Now, we give a few details on Galois groups in model theory. For more
detailed expositions the reader is referred to \cite[Subsection
  1.3]{KrPiRz} or \cite[Subsection 4.1]{KrPi_recent}. If the reader is
interested in yet more details and proofs, he or she may consult
fundamental papers around this topic, e.g. \cite{LaPi, Zi} or \cite{CLPZ}.
 
Let $\C$ be a monster model of a theory $T$.

\begin{dfn}$\,$
		\begin{enumerate}[label=\roman{*}),nosep]
			\item
			{\em The group of Lascar strong automorphisms}, denoted by $\autf_L(\C)$, is the subgroup of $\aut(\C)$ generated by all automorphisms fixing a small submodel of $\C$ pointwise, i.e.\ $\autf_L(\C)=\langle \sigma : \sigma \in \aut(\C/M)\;\, \mbox{for a small}\;\, M\prec \C\rangle$.
			\item
			{\em The Lascar Galois group of $T$}, denoted by $\gal_L(T)$, is the quotient group $\aut(\C)/\autf_L(\C)$ (which makes sense, as $\autf_L(\C)$ is a normal subgroup of $\aut(\C)$).
		\end{enumerate}
	\end{dfn}
The orbit equivalence relation of $\autf_L(\C)$ acting on any given product $S$ of boundedly many sorts of $\C$ is usually denoted by $E_L$. It turns out that this is the finest bounded (i.e. with boundedly many classes), invariant equivalence relation on $S$; and the same is true after the restriction to the set of realizations of any type in $S_S(\emptyset)$. The classes of $E_L$ are called {\em Lascar strong types}. 

 Now, we recall the logic topology on $\gal_L(T)$. 
	Let $\nu\fcolon \aut(\C) \to \gal_L(T)$ be the quotient map. Choose a small model $M$, and let $\bar m$ be its enumeration. By $S_{\bar m}(M)$ we denote $\{ \tp(\bar n/M): \bar n \equiv \bar m\}$. Let $\nu_1\fcolon \aut(\C) \to S_{\bar m}(M)$ be defined by $\nu_1(\sigma)=\tp(\sigma(\bar m)/M)$, and $\nu_2\fcolon S_{\bar m}(M) \to \gal_L(T)$ by $\nu_2(\tp(\sigma(\bar m)/M))=\sigma /\autf_L(\C)$. Then $\nu_2$ is a well-defined surjection, and $\nu=\nu_2 \circ \nu_1$. Thus, $\gal_L(T)$ becomes the quotient of the space $S_{\bar m}(M)$ by the relation of lying in the same fiber of $\nu_2$, and so we can define a topology on $\gal_L(T)$ as the quotient topology. In this way, $\gal_L(T)$ becomes a quasi-compact (so not necessarily Hausdorff) topological group. This topology does not depend on the choice of the model $M$. 

Next, define $\gal_0(T)$ as the closure of the identity in $\gal_L(T)$. We put $\autf_{KP}(T):=\nu^{-1}[\gal_0(T)]$, and finally $\gal_{KP}(T): = \aut(\C)/\autf_{KP}(\C)$. Then  $\gal_{KP}(T) \cong \gal_L(T)/\gal_0(T)$ becomes a compact (Hausdorff) topological group (with the quotient topology). We also have the obvious continuous epimorphism $h \colon \gal_L(T) \to \gal_{KP}(T)$. We say that $T$ is {\em G-compact} if $h$ is an isomorphism; equivalently, if $\gal_0(T)$ is trivial.

Taking $M$ of cardinality $|T|$, since $\gal_L(T)$ is the image of $S_{\bar m}(M)$ by $\nu_2$, we get that $|\gal_L(T)| \leq 2^{|T|}$.

Finally, we recall some results from \cite{KrPiRz}. As usual, $\bar c$ is an enumeration of $\C$. Let $EL=EL(S_{\bar c}(\C))$, ${\mathcal M}$ be a minimal left ideal in $EL$, and $u \in {\mathcal M}$ an idempotent. Let $\C' \succ \C$ be a monster model with respect to $\C$. 

We define $\hat{f}\colon EL \to \gal_L(T)$ by
	\[
		\hat{f}(\eta)=\sigma'/ \autf_L(\C'),
	\]
	where $\sigma' \in \aut(\C')$ is such that $\sigma'(\bar c)
        \models \eta(\tp(\bar c/\C))$. It turns out that this is a
        well-defined semigroup epimorphism. Its restriction $f$ to
        $u{\mathcal M}$ is a group epimorphism from $u{\mathcal M}$ to
        $\gal_L(T)$ \cite[Proposition 2.4 and Corollary 2.6]{KrPiRz}. 

\begin{fct}[Theorem 2.7 from \cite{KrPiRz}]\label{fact: main fact 1 from KrPiRz}
		Equip $u\M$ with the $\tau$-topology and $u\M/H(u\M)$ with the induced quotient topology. Then:
		\begin{enumerate}[nosep]
			\item $f$ is continuous.
			\item $H(u\M) \leq \ker(f)$.
			\item The formula $p/H(u\M) \mapsto f(p)$ yields a well-defined continuous epimorphism $\bar f$ from $u\M/H(u\M)$ to $\gal_L(T)$.
		\end{enumerate}
		In particular, we get the following sequence of continuous epimorphisms:
\begin{equation}
			u\M\twoheadrightarrow u\M/H(u\M)\xtwoheadrightarrow{\bar f}{}\gal_L(T)\xtwoheadrightarrow{h}{}\gal_{KP}(T).
		\end{equation}
		
	\end{fct}

\begin{fct}[Theorem 2.9 from \cite{KrPiRz}]\label{fact: main fact 2 from KrPiRz}
For $Y:=\ker(\bar f)$ let $\cl_\tau(Y)$ be the closure of $Y$ inside $u\M/H(u\M)$. Then $\bar f[\cl_\tau(Y)]=\gal_0(T)$, so $\bar f$ restricted to $\cl_\tau(Y)$ induces an isomorphism between $\cl_\tau(Y)/Y$ and $\gal_0(T)$.
\end{fct}

These facts imply that $\gal_L(T)$ can be presented as the quotient of a compact Hausdorff group by some subgroup, and $\gal_0(T)$ is such a quotient but by a dense subgroup.

\begin{cor}\label{corollary: f iso implies G-compactness}
If $\bar f$ is an isomorphism, then $\gal_0(T)$ is trivial (i.e. $T$ is G-compact). In particular, if $f$ is an isomorphism, then $T$ is G-compact.
\end{cor}

\begin{proof}
If $\bar f$ is an isomorphism, then $Y$ is a singleton. But the $\tau$-topology on $u{\mathcal M}$ is $T_1$, so $Y$ is $\tau$-closed, i.e. $\cl_\tau(Y)=Y$ has only one element. Hence, by Fact \ref{fact: main fact 2 from KrPiRz}, $\gal_0(T)=\bar f[\cl_\tau(Y)]$ is trivial.
\end{proof}

\section{A few reductions}\label{section 2}

We explain here some basic issues which show that Theorems \ref{theorem: main theorem 1} and \ref{theorem: main theorem 2} yield answers to Questions \ref{question: main question 1} and \ref{question: main question 2}, and that it is enough to prove these theorems assuming that the number of factors in the product $S$ is bounded.

Let $\C$ be a monster model of an arbitrary theory $T$.
Recall that $\bar c$ is an enumeration of $\C$, and $\bar \alpha$ is a short tuple of elements of $\C$.
Whenever we talk about realizations of types over $\C$, we choose them from a bigger monster model $\C' \succ \C$.

If $p(\bar y)$ is a type and $\bar x\subseteq \bar y$ is a subsequence of variables, we let $p|_{\bar x}$ denote the restriction of $p$ to the variables $\bar x$.

The following remark is easy to check.
%
%

\begin{rem}\label{remark: very easy 2}
Let $\bar d$ be a tuple of all elements of $\C$ (with repetitions) such that $\bar \alpha$ is a subsequence of $\bar d$. Let $r \colon S_{\bar d}(\C) \to S_{\bar \alpha}(\C)$ be the restriction to the appropriate coordinates. Then $r$ is an epimorphism of $\aut(\C)$-ambits which induces an epimorphism $\hat{r} \colon EL(S_{\bar d}(\C)) \to EL(S_{\bar \alpha}(\C))$ of semigroups. In particular, each minimal left ideal in $EL(S_{\bar d}(\C))$ maps via $\hat{r}$ onto a minimal left ideal in $EL(S_{\bar \alpha}(\C))$, and similarly for the Ellis groups.
\end{rem}

\begin{cor}\label{corollary: the size for alpha bounded by the size for c}
The size of a minimal left ideal in $EL(S_{\bar c}(\C))$ is greater than or equal to the size of a minimal left ideal in $EL(S_{\bar \alpha}(\C))$, and the size of the Ellis group of the flow $S_{\bar c}(\C)$ is greater than or equal to the size of the Ellis group of the flow $S_{\bar \alpha}(\C)$. 
\end{cor}
%

Let $S,S'$ be two products of sorts, with possibly an unbounded number of factors and repetitions allowed, with associated variables $\bar x$ and $\bar y$ respectively. Let $X$ [resp. $Y$] be a $\emptyset$-type-definable subset of $S$ [resp. $S'$]. Say that $X$ and $Y$ have the same {\em finitary content} if for every finite $\bar x'\subseteq \bar x$ and $p\in S_X(\C)$ there is a $\bar y'\subseteq \bar y$ of the same size (and associated with the same sorts) and $q\in S_Y(\C)$ such that $p|_{\bar x'} = q|_{\bar y'}$ and conversely, switching the roles of $X$ and $Y$ (formally, this equality denotes equality after the identification of the corresponding variables, but we will usually ignore this). The above notion of finitary content has nothing to do with the notion of content of a type which will be introduced in the next section.

\begin{prop}\label{proposition: same content isomorphic ellis groups}
	Let $S,S'$, $X$ and $Y$ be as above and assume that $X$ and $Y$ have the same finitary content. Then $EL(S_X(\C)) \cong EL(S_Y(\C))$ as semigroups and as $\aut(\C)$-flows. In particular, the corresponding minimal left ideals of these Ellis semigroups are isomorphic, and the Ellis groups of the flows $S_X(\C)$ and $S_Y(\C)$ are isomorphic as groups equipped with the $\tau$-topology.
\end{prop}

\begin{proof}
	Take $\eta\in EL(S_X(\C))$, $p(\bar x),q(\bar x)\in S_X(\C)$. Let $\bar x_0, \bar x_1\subseteq \bar x$ be two finite tuples of the same size. Assume that $p|_{\bar x_0}=q|_{\bar x_1}$. Then $\eta(p)|_{\bar x_0}=\eta(q)|_{\bar x_1}$. This allows us to define $f \colon EL(S_{X}(\C)) \to EL(S_Y(\C))$ by putting $f(\eta)=\eta'$, where for every $q(\bar y)\in S_Y(\C)$ and finite $\bar y'\subseteq \bar y$, $\eta'(q)|_{\bar y'}=\eta(p)|_{\bar x'}$, where $p\in S_X(\C)$ and $\bar x'\subseteq \bar x$ are such that $p|_{\bar x'}=q|_{\bar y'}$.
	
	Now, $f$ is a morphism of semigroups and of
        $\aut(\C)$-flows. Furthermore, the map $g:EL(S_{Y}(\C)) \to
        EL(S_X(\C))$ defined in the same way as $f$ switching the
        roles of $X$ and $Y$, is an inverse of $f$. Therefore, $f$ is
        an isomorphism of semigroups and $\aut(\C)$-flows. Also, $f$
        respects the $\tau$-topology. This follows immediately from 
the definition of the $\tau$-topology (see  
\cite[Definitions 1.3 and 1.4]{KrPiRz}) and the fact that $f$ maps $\id$ to 
$\id$.
\end{proof}

\begin{prop}\label{proposition: invariants for c isomorphic to invariants for S}
Let $S$ be the product of all the sorts of the language such that each sort is repeated $\aleph_0$ times. Then $EL(S_{\bar c}(\C)) \cong EL(S_S(\C))$ as semigroups and as $\aut(\C)$-flows. In particular, the corresponding minimal left ideals of these Ellis semigroups are isomorphic, and the Ellis groups of the flows $S_{\bar c}(\C)$ and $S_{S}(\C)$ are isomorphic as groups equipped with the $\tau$-topology.
\end{prop}

\begin{proof}This follows at once from Proposition \ref{proposition: 
same content isomorphic ellis groups}, because $EL(S_{\bar c}(\C)) \cong 
EL(S_{\bar d}(\C))$, and  $\tp(\bar d/\emptyset)$ and $S$ have the same 
finitary content, where $\bar d$ is the tuple of all elements $\C$ such 
that each element is repeated $\aleph_0$ times.
\end{proof}

Proposition \ref{proposition: same content isomorphic ellis groups} also implies that $EL(S_{\bar c}(\C))$ is isomorphic with $EL(S_{\bar \beta}(\C))$ for a suitably chosen short tuple $\bar \beta$. Namely, we have 

\begin{prop}\label{proposition: short tuple instead of the monster model}
Let $\bar m$ be an enumeration of an $\aleph_0$-saturated model (of bounded size). Then $EL(S_{\bar c}(\C)) \cong EL(S_{\bar m}(\C))$ as semigroups and as $\aut(\C)$-flows, and we have all the further conclusions as in Proposition \ref{proposition: invariants for c isomorphic to invariants for S}.
\end{prop}

To complete the picture, we finish with the following proposition which shows that in most of our results without loss of generality one can assume that the product of sorts in question has only boundedly many factors. 
Actually, the next proposition is ``almost'' a generalization of Proposition \ref{proposition: invariants for c isomorphic to invariants for S}. Note, however, that the bound on the number of factors in the obtained product is smaller in Proposition \ref{proposition: invariants for c isomorphic to invariants for S}, and an important information in Proposition \ref{proposition: invariants for c isomorphic to invariants for S} is that the obtained product of sorts does not depend on the choice of $\C$ and its enumeration $\bar c$ (although the product of sorts to which $\bar c$ belongs is getting arbitrarily long  when $\C$ is getting bigger). 

\begin{prop}\label{proposition: unboundedly many sorts}
Let $S$ be a product of some sorts of the language with repetitions allowed so that the number of factors may be unbounded, and let $X$ be a $\emptyset$-type-definable subset of this product. Then there exists a product $S'$ of some sorts with a bounded number (at most $2^{|T|}$) of factors and a $\emptyset$-type-definable subset $Y$ of $S'$ such that $EL(S_X(\C)) \cong EL(S_Y(\C))$ as semigroups and as $\aut(\C)$-flows. In particular, the corresponding minimal left ideals of these Ellis semigroups are isomorphic, and the Ellis groups of the flows $S_{X}(\C)$ and $S_{Y}(\C)$ are isomorphic as groups equipped with the $\tau$-topology. Moreover, $S'$ and $Y$ can be chosen independently of the choice of $\C$ (for the given $S$ and $X$).
\end{prop}

\begin{proof}	
Let $S=\prod_{i<\lambda} S_i$ and let $\bar x = (x_i)_{i< \lambda}$ be the variables of types in $S_S(\C)$. For any two finite subsequences $\bar y$ and $\bar z$ of $\bar x$ of the same length and associated with the same sorts, any two types $q(\bar z)$ and $r(\bar y)$ will be identified if and only if they are equal after the identification of $\bar z$ and $\bar y$.

For any subsequence  $\bar y = (\bar x_i)_{i \in I}$ of $\bar x$, let $\pi_{\bar y}$ denote the projection from $X$ to $\prod_{i \in I} S_i$. Then $\pi_{\bar y}[X]$ is a $\emptyset$-type-definable subset of $\prod_{i \in I} S_i$. Moreover, the restriction map $r_{\bar y}: S_X(\C) \to S_{\prod_{i \in I} S_i}(\C)$ (i.e. $r_{\bar y}(p) := p|_{\bar y}$) maps $S_X(\C)$ onto $S_{\pi_{\bar y}[X]}(\C)$.

Now, there exists $I \subseteq \lambda$ (independent of the choice of $\C$) of cardinality at most $2^{|T|}$ such that for any finite sequence of sorts $\bar P$ the collection of all sets $\pi_{\bar z}[X]$, where $\bar z$ is a finite subsequence of $\bar x$ of variables associated with the sorts $\bar P$, coincides with the collection of such sets with $\bar z$ ranging over all finite subsequences (of variables associated with the sorts $\bar P$) of the tuple $\bar y := (\bar x)_{i \in I}$. Let $Y=\pi_{\bar y}[X]$ be the projection of $X$ to the product of sorts $S':=\prod_{i \in I} S_i$.  Then $Y$ is a $\emptyset$-type-definable subset of $S'$. By construction, $Y$ has the same finitary content as $X$, hence by Proposition \ref{proposition: same content isomorphic ellis groups}, $EL(S_X(\C))$ and $EL(S_Y(\C))$ are isomorphic as semigroups and $\aut(\C)$-flows.
\end{proof}

\section{Content of types and subflows of $S_X(\C)$}\label{section2.5}

In this section, we define the content of a type (or a tuple of types) and use it to understand subflows of the space of types.

As usual, we work in a monster model $\C$ of an arbitrary theory $T$.

\begin{dfn}\label{definition: content}
Let $A\subseteq B$.
\begin{enumerate}
\item[i)] 	Let $p(\bar x)\in S(B)$. Define the \emph{content} of $p$ over $A$ as
	\[c_A(p)=\{(\varphi(\bar x,\bar y),q(\bar y)) : \varphi(\bar x,\bar b)\in p(\bar x) \text{ for some }\bar b\text{ such that }q(\bar y)=\tp(\bar b/A)\},\]
where $\varphi(\bar x,\bar y)$ are formulas with parameters from $A$.
	When $A=\emptyset$, we write simply $c(p)$ and call it the {\em content} of $p$.
\item[ii)] If the pair $(\varphi(\bar x,\bar y),q(\bar y))$ is in $c_A(p)$, then we say that it is {\em represented} in $p(\bar x)$.
\item[iii)]	Likewise, for a sequence of types $p_1(\bar x),\ldots,p_n(\bar x)\in S(B)$, $q(\bar y)\in S(A)$ and a sequence $\varphi_1(\bar x,\bar y),\ldots,\varphi_n(\bar x, \bar y)$ of formulas with parameters from $A$, we say that $(\varphi_1,\ldots,\varphi_n,q)$ is {\em represented} in $(p_1,\ldots,p_n)$ if $\varphi_1(\bar x,\bar b)\in p_1,\ldots,\varphi_n(\bar x,\bar b)\in p_n$ for some $\bar b\models q$. We define the contents $c_A(p_1,\ldots,p_n)$ and $c(p_1,\ldots,p_n)$ accordingly.
\end{enumerate}
\end{dfn}

The relation $c(p)\subseteq c(q)$ is similar to the fundamental order as defined by Lascar and Poizat (\cite{Poizat_book}) as an alternative approach to forking in stable theories. More precisely, the relation of inclusion of content is a refinement of the usual fundamental order. With this analogy in mind, we define analogue notions of heirs and coheirs.
\begin{dfn}
	Let $M\subseteq A$ and $p(\bar x)\in S(A)$. 
\begin{enumerate}
\item[i)] We say that $p(\bar x)$ is a {\em strong heir} over $M$ if for every finite $\bar m\subset M$ and $\varphi(\bar x,\bar a)\in p(\bar x)$, $\bar a$ finite, there is $\bar a'\subset M$ with $\varphi(\bar x,\bar a')\in p(\bar x)$ such that $\tp(\bar a/\bar m)=\tp(\bar a'/\bar m)$.
\item[ii)] We say that $p(\bar x)\in S(A)$ is a {\em strong coheir} over $M$ if for every finite $\bar m\subset M$ and every $\varphi(\bar x',\bar a)\in p(\bar x)$, where $\bar x'\subseteq \bar x$ is a finite subsequence of variables, there is a $\bar b\subset M$ realizing $\varphi(\bar x',\bar a)$, with $\tp(\bar b/\bar m)=p|_{\bar m}(\bar x')$.
\end{enumerate}
\end{dfn}
In this definition $\varphi(\bar{x},\bar{y})$ denotes a formula over
$\emptyset$, but equivalently we may assume that $\varphi(\bar{x},\bar{y})$
is over $M$. 

Note that those notions are dual:
\[\tp(\bar a/M\bar b)\text{ is a strong heir over }M \iff \tp(\bar b/M\bar a)\text{ is a strong coheir over }M.\]

\begin{lem}
	Let $M\subseteq A$, where $M$ is $\aleph_0$-saturated. Assume that $p(\bar x)\in S(M)$. Then $p(\bar x)$ has an extension $p'(\bar x)\in S(A)$ which is a strong coheir [resp. strong heir] over $M$.
\end{lem}

\begin{proof}
	Choose an ultrafilter $U$ on $M^{|\bar x|}$ containing all sets of the form $p|_{\bar m}(\bar x')(M)$, where $\bar m\subset M$ and $\bar x'\subseteq \bar x$ are finite. Define $p'(\bar x)$ as the set of formulas $\varphi(\bar x)$ with parameters from $A$ such that $\varphi(\C)\cap M^{|\bar x|} \in U$. Then $p'$ extends $p$ and is a strong coheir over $M$.
	
	To construct a strong heir of $p$, we dualize the
        argument. Let $\bar b\models p(\bar x)$ and let $\bar a$
        enumerate $A$. By the previous paragraph, let $\bar
        a'\equiv_M\bar a$ be such that $\tp(\bar a'/M\bar b)$ is a strong coheir over $M$. Then $\tp(\bar b/M\bar a')$ is a strong heir over $M$. Take $f\in \aut(\C/M)$ mapping $\bar a'$ to $\bar a$. Then $r(\bar x):=\tp(f(\bar b)/M\bar a)$ is a strong heir over $M$ extending $p(\bar x)$.
\end{proof}

Let $S$ be a product of some sorts of the language, possibly unboundedly many with repetitions allowed, and  $X$ a $\emptyset$-type-definable subset of $S$. Let $l_S$ be the number of factors in $S$.

Directly from the definition of the content of a tuple of types we get

\begin{rem}\label{remark: bound of the number of configurations}
For every natural number $n$ there are only boundedly many possibilities for the content $c({p_1,\dots,p_n})$ of types $p_1,\dots,p_n \in S_S(\C)$. More precisely, the number of possible contents is bounded by $2^{l_S+ 2^{|T|}}$.
\end{rem}


The relation with the Ellis semigroup is given by the following

\begin{prop}\label{proposition: content determines orbit}
Let  $(q_1,\dots,q_n)$ and $(p_1,\dots,p_n)$ be tuples of types from $S_X(\C)$. 
Then $c({q_1,\dots,q_n}) \subseteq c({p_1,\dots,p_n})$ if and only if there is $\eta \in EL(S_X(\C))$ such that $\eta(p_i)=q_i$ for all $i=1,\dots,n$.
\end{prop}

\begin{proof}
 ($\rightarrow$) Consider any $\varphi_1(\bar x,\bar b) \in q_1,\dots, \varphi_n(\bar x,\bar b) \in q_n$. By assumption, there is a tuple $\bar b' \equiv_\emptyset \bar b$ such that $\varphi_i(\bar x, \bar b') \in p_i$ for all $i=1,\dots,n$. Take $\sigma_{\varphi_1(\bar x, \bar b),\dots,\varphi_n(\bar x, \bar b)} \in \aut(\C)$ mapping $\bar b'$ to $\bar b$. 
Choose a subnet $(\sigma_j)$ of the net  $(\sigma_{\varphi_1(\bar x, \bar b),\dots,\varphi_n(\bar x, \bar b)})$ which converges to some $\eta \in EL$. Then $\eta(p_i)=q_i$ for all $i$.\\[1mm]
 ($\leftarrow$) Consider any $(\varphi_1(\bar x, \bar y), \dots, \varphi_n(\bar x, \bar y),q(\bar y))\in c(q_1,\ldots,q_n)$. Then there is $\bar b \in q(\C)$ such that $\varphi_i(\bar x, \bar b) \in q_i$ for all $i=1,\dots,n$. By the fact that $\eta$ is approximated by automorphisms of $\C$, we get $\sigma \in \aut(\C)$ such that $\varphi_i(\bar x, \bar b) \in \sigma(p_i)$, and so $\varphi_i(\bar x, \sigma^{-1}(\bar b)) \in p_i$, holds for all $i=1,\dots,n$. This shows that $(\varphi_1(\bar x, \bar y), \dots, \varphi_n(\bar x, \bar y),q(\bar y))\in c(p_1,\ldots,p_n)$.
\end{proof}

This allows us to give a description of all point-transitive subflows of $S_X(\C)$. Given a subflow $Y\subseteq S_X(\C)$ and $q\in Y$ with dense orbit, let $c_Y=c(q)$. By the previous proposition, this does not depend on the choice of $q$. The mapping $Y\mapsto c_Y$ is injective and preserves inclusion. In particular, we deduce that there are boundedly many subflows of $S_X(\C)$ (at most $\beth_2(l_S+ 2^{|T|})$). Likewise, there are boundedly many subflows of $S_X(\C)^n$ (the $n$-th Cartesian power of $S_X(\C)$).

Analogous to the way non-forking can be defined in stable theories as extensions maximal in the fundamental order, we can use the content to define a notion of free extension.

\begin{dfn}
	Let $p(\bar x)\in S(A)$. Say that an extension $p'(\bar x)\in S(\C)$ of $p$ is free if $c_A(p')$ is minimal among $c_A(r)$ for $r(\bar x)\in S_p(\C):=\{ q(\bar x) \in S(\C): p(\bar x) \subseteq q(\bar x)\}$.
	
\end{dfn}

\begin{lem}
	Let $p(\bar x)\in S(A)$. Then $p$ has a free extension $p'(\bar x)\in S(\C)$.
\end{lem}

\begin{proof}
	Expand the language by constants for the elements of $A$. Let $Y$ be a minimal subflow of $S_p(\C)$ and take $p'\in Y$ (so $p'$ is an almost periodic type of the $\aut(\C)$-flow $S_p(\C)$). Then, by the discussion above, $c_Y$ is minimal among $c_Z$, $Z\subseteq S_p(\C)$ a point-transitive subflow, and hence $p'$ is a free extension of $p$.
\end{proof}

This provides us with a notion of freeness that is well-defined and satisfies existence in any theory. 
One can check (using definability of types) that in the case of a
stable theory it coincides with non-forking (we will not use this in
this paper).
Its properties will be investigated in a future work.

\section{Boundedness and absoluteness of the Ellis group}\label{section 3}

This section is devoted to the proof of Theorem \ref{theorem: main theorem 1}: in the first part, we will prove boundedness of the Ellis group and we will give an explicit bound on its size; in the second part, we will prove absoluteness. The key tool in both parts is the notion of content of a sequence of types introduced in the previous section. Boundedness will follow easily. The proof of absoluteness is more technical.

Let $\C$ be a monster model of an arbitrary theory $T$, and $\kappa$
the degree of saturation of $\C$. Let $S$ be a product of some number
sorts (posibly unbounded, with repetitions allowed), and let $X$ be a $\emptyset$-type-definable subset of $S$. 
In this section, by $EL$ we will denote the Ellis semigroup $EL(S_X(\C))$. Let $l_S$ be the length of $S$ (i.e. the number of factors in the product $S$).

\medskip
We first prove boundedness of the Ellis group. By Remark \ref{remark: bound of the number of configurations}, we can find a subset $P \subseteq \bigcup_{n \in \omega} S_X(\C)^n$ of size $\leq 2^{l_S+ 2^{|T|}}$) such that for every $n$,
\[\{ c(p_1,\dots,p_n): (p_1,\dots,p_n) \in P\} = \{ c(p_1,\dots,p_n): (p_1,\dots,p_n) \in S_X(\C)^n\}.\]
Define $R$ to be the closure of the collection $P_{\pr}$ of all types $p\in S_X(\C)$ for which there is $(p_1,\dots,p_n) \in P$ such that $p=p_i$ for some $i$. Note that $|R| \leq \beth_2({|P_{\pr}|}) \leq \beth_3(l_S+ 2^{|T|})$.

\begin{lem}\label{lemma: Im(eta) contained in R} 
Assume $S\subseteq S_X(\C)$ is closed and for every finite tuple $\bar
p=(p_1,\dots,p_n)$ of types form $S_X(\C)$ there is a finite tuple
$\bar q=(q_1,\dots,q_n)$ of types from $S$ with $c(\bar q)\subseteq
c(\bar p)$. Then there exists $\eta\in EL$ with $\im(\eta)\subseteq
S$. In particular, there is $\eta \in EL$ with $\im(\eta) \subseteq R$. 
\end{lem}

\begin{proof}
By Proposition \ref{proposition: content determines orbit}, for every
finite tuple $\bar p =(p_1,\dots,p_n)$ of types from $S_X(\C)$ there
exists $\eta_{\bar p} \in EL$ such that $\eta_{\bar p}(p_i) \in S$ for
all $i=1,\dots,n$. The collection of all finite tuples of types from
$S_X(\C)$ forms a directed set (where $\bar p \leq \bar q$ if and only
if $\bar p$ is a subtuple of $\bar q$). So the elements $\eta_{\bar
  p}$ (where $\bar p$ ranges over all finite tuples of types from
$S_X(\C)$) form a net which has a subnet convergent to some $\eta \in
EL$. Then $\im (\eta) \subseteq S$, because $S$ is closed. The last
part of the lemma follows by the choice of $R$.
\end{proof}

From now on, let ${\mathcal M}$ be a minimal left ideal of $EL$. So
$\mathcal M$ is partitioned into groups of the form $u\mathcal M$,
where $u\in\mathcal M$ ranges over idempotents of $\mathcal M$. The
next lemma clarifies the nature of this partition.

\begin{lem}\label{lemma: hM=uM}
Assume $u,u'\in\mathcal M$ are idempotents and
$I=\im(u),\ I'=\im(u')$.
\begin{itemize}
\item[(1)] $u{\mathcal M} = \{ h\in{\mathcal M}: \im(h)=I\}$.
\item[(2)] If $u\neq u'$, then $I\not\subseteq I'$.
\item[(3)] For every $h\in u{\mathcal M}$, $h|_I$ is a permutation of
  $I$. Let $\sym(I)$ denote the group of permutations of $I$. The
  function $F\colon u{\mathcal M}\to\sym(I)$ mapping $h$ to $h|_I$ is a
  group monomorphism.
\item[(4)] For every $h\in{\mathcal M}$ there is a unique idempotent
  $u\in\mathcal M$ wuch that $h{\mathcal M}=u{\mathcal M}$. In
  particular, $h\mathcal M$ is the Ellis group of the flow $S_X(\C)$.
\end{itemize}
\end{lem} 

\begin{proof}
First notice that for $h\in u{\mathcal M}$ we have $h=uh$ and $u=hh'$,
where $h'\in u\mathcal M$ is the group inverse of $H$. Hence,
$I=\im(u)\subseteq \im(h)\subseteq I$ and we get $\subseteq$ in (1).

For (2), suppose for a contradiction that $I\subseteq I'$. Then $u'u=u$
belongs to the group $u'\mathcal M$, a contradiction.

For $\supseteq$ in (1) notice that if $h\in{\mathcal M}\setminus
u\mathcal M$, then by (2) and $\subseteq$ in (1), $\im(h)\neq I$.

(3) Let $h\in u\mathcal M$. Then there is $h'\in u\mathcal M$ with
$u=hh'=h'h$. Hence, 
$$\id_I=u|_I=h|_I\circ h'|_I=h'|_I\circ h|_I.$$
Thus, $h_I\in\sym(I)$. Obviously, $F$ is a group homomorphism. It is
injective, since $u$ is the only idempotent in $u\mathcal M$.

(4) is immediate, since $h{\mathcal M}=hu{\mathcal M}=u\mathcal M$,
  where $u\in \mathcal M$ is the unique idempotent with $h\in
  u\mathcal M$.
\end{proof}



\begin{lem}\label{lemma: u in place of eta}
For every $\eta \in EL$ there is an idempotent $u \in {\mathcal M}$ such that $\im (u) \subseteq \im (\eta)$.
\end{lem}

\begin{proof}
Take any $h' \in {\mathcal M}$. Then $h:= \eta h'$ belongs to ${\mathcal M}$ and satisfies $\im (h) \subseteq \im (\eta)$. By Lemma \ref{lemma: hM=uM}(4), there is an idempotent $u \in h{\mathcal M}$. Then $u \in {\mathcal M}$ and $\im(u) \subseteq \im (\eta)$.
\end{proof}

We consider the set of functions $R^R$ as a semigroup, with
composition of functions. By Lemmas \ref{lemma: Im(eta) contained in
  R} and \ref{lemma: u in place of eta}, we can find an idempotent $u\in\mathcal M$
with $\im(u)\subseteq R$. Let $I=\im(u)$. By Lemma \ref{lemma: hM=uM}(3),
we get the following corollary, which yields the first part of Theorem \ref{theorem: main theorem 1}, namely to the boundedness of the Ellis group. 



\begin{cor}\label{corollary: first part of main theorem 1}
The function $F:u{\mathcal M}\to \sym(I)$ given by $F(h)=h|_I$ is a
group monomorphism. In particular, the size of the Ellis group of the
flow $S_X(\C)$ is bounded by $|\sym(I)|\leq |R|^{|R|}$, which in turn
  is bounded by $\beth_4(l_S+ 2^{|T|})$. In the case where $l_S \leq 2^{|T|}$, this bound equals $\beth_5(|T|)$.
\end{cor}
\begin{proof}
This first part follows directly from Lemma \ref{lemma: hM=uM}(3)
The precise bound on the size of the Ellis group follows from the bound on $|R|$ computed before Lemma \ref{lemma: Im(eta) contained in R}. 
\end{proof}

By Proposition  
\ref{proposition: unboundedly many sorts} and Corollary \ref{corollary: first part of main theorem 1}, we get the following corollary which in particular contains the first part of Theorem \ref{theorem: main theorem 1} and answers Question \ref{question: main question 1}(i).

\begin{cor}\label{corollary: final corollary concerning the size of the Ellis group}
Let $S$ be a product of an arbitrary (possibly unbounded) number of sorts (with repetitions allowed) and $X$ be a $\emptyset$-type-definable subset of $S$. Let ${\bar c}$ be an enumeration of $\C$ and $\bar \alpha$ a short tuple in $\C$. Then the sizes of the Ellis groups of the flows $S_X(\C)$, $S_{\bar c}(\C)$ and $S_{\bar \alpha}(\C)$ are all bounded by $\beth_5(|T|)$. 
\end{cor}

The following problem is left for the future.

\begin{problem}
Find the optimal (i.e. smallest) upper bound on the size of the Ellis groups of the flows of the form $S_X(\C)$.
\end{problem}

In Subsection \ref{subsection 6.2}, we will see that under the assumption that $T$ has NIP, the set $R$ in Corollary \ref{corollary: first part of main theorem 1} can be replaced by the set of types invariant over a fixed small model, which gives us a smaller bound on the size of the Ellis group, namely $\beth_3(|T|)$. On the other hand, we know that the optimal bound is at least $2^{|T|}$, as by the material recalled in the final part of Section \ref{section 1}, the size of the Ellis group of $S_{\bar c}(\C)$ is at least $|\gal_L(T)|$ which can be equal to $2^{|T|}$ (e.g. for the theory of countably many independent equivalence relations each of which has two classes).

\medskip

Now, we turn to the second, more complicated, part of Theorem \ref{theorem: main theorem 1}, namely the absoluteness of the isomorphism type of the Ellis group of $S_X(\C)$. So from now on, fix any monster models $\C_1 \succ \C_2$ of the theory $T$. Let $EL_1=EL(S_X(\C_1))$ and $EL_2=EL(S_X(\C_2))$, and let ${\mathcal M}_1$ and ${\mathcal M}_2$ be minimal left ideals of $EL_1$ and $EL_2$, respectively. Let $\pi_{12}\colon S_X(\C_1) \to S_X(\C_2)$ be the restriction map.

The general idea is to find 
idempotents $u \in {\mathcal M}_1$ and $v \in {\mathcal M}_2$ such that $\im(u)$ is of bounded size and $\pi_{12}$ restricted to $\im (u)$ is a homeomorphism onto $\im (v)$. Having this, we will easily show that the restriction maps $F_1 \colon u{\mathcal M_1} \to \sym(\im (u))$ and $F_2 \colon v{\mathcal M_2} \to \sym(\im(v))$ are group monomorphisms, and that $\pi_{12}$ induces an isomorphism between $\im(F_1)$ and $\im(F_2)$. This implies that $u{\mathcal M}_1 \cong v{\mathcal M}_2$. In Subsection \ref{subsection 6.2}, we will see that under NIP such idempotents $u$ and $v$ may be chosen so that their images lie in the set of types invariant over a given model $M \prec \C$. In general, we have to find an abstract substitute of the set of invariant types, which will be a more carefully chosen set $R$ as above. This is done in Corollary \ref{corollary: corollary of technical lemmas} which requires proving some technical lemmas concerning content. If the reader wishes first to see the main steps of the proof of the main result, he or she can go directly to the statement of Corollary \ref{corollary: corollary of technical lemmas} and continue reading from that point on.

For types $p_1,\dots,p_n \in S_S(M)$, in order to emphasize the model in which we are working, the content $c(p_1,\dots,p_n)$ of the tuple $(p_1,\dots,p_n)$ will be denoted by $c^M({p_1,\dots,p_n})$.

\begin{lem}\label{lemma: technical lemma 1}
Let $M \prec N$ be any models of $T$ and assume that $M$ is $\aleph_0$-saturated. Then each type $p \in S_S(M)$ can be extended to a type $p^N \in S_S(N)$ so that for every $p_1,\dots,p_n \in S_S(M)$ one has $c^M({p_1,\dots,p_n}) = c^N(p_1^N,\dots, p_n^N)$.

\end{lem}
\begin{proof}
	Let $\bar p=(p_i)_{i\in I}$ be an enumeration of the types in $S_S(M)$ and let $\bar a\subset \C$ realize $\bar p$. Choose $\bar a'$ so that $\tp(\bar a'/N)$ extends $\tp(\bar a/M)$ and is a strong heir over $M$. Then the types $p_i^N := \tp(\bar a_i'/N)$ have the required property.
\end{proof}

From now on, in this section, we assume that $S$ is a product of
boundedly many sorts, unless stated otherwise.

\begin{lem}\label{lemma: technical lemma 2}
Let $M$ be any small (i.e. of cardinality less than $\kappa$), $\aleph_0$-saturated model of $T$, and let $Y$ be any subset of $S_S(M)$. Then there exists a small, 
$\aleph_0$-saturated 
model $N \succ M$ and an extension $p^N \in S_S(N)$ of each $p \in Y$ such that: 
\begin{enumerate}
\item for every $p_1,\dots,p_n \in Y$ one has $c^M({p_1,\dots,p_n}) = c^N(p_1^N,\dots, p_n^N)$, and
\item for every $\aleph_0$-saturated $N' \succ N$ and for every
  $p^{N'} \in S_S(N')$ extending $p^N$ in such a way that $c^M(p_1,\dots,p_n) = c^{N'}(p_1^{N'},\dots, p_n^{N'})$ for all $p_1,\dots,p_n \in Y$, the restriction map $\pi \colon \cl(\{ p^{N'}: p \in Y\}) \to  \cl(\{ p^{N}: p \in Y\})$ is a homeomorphism.
\end{enumerate}
\end{lem}

\begin{proof}
Suppose this is not the case. Then
we can construct sequences $(N_\alpha)_{\alpha< \kappa}$ and $(p_{\alpha})_{\alpha < \kappa}$ for each $p \in Y$, where $N_\alpha \succ M$ 
is $\aleph_0$-saturated 
and $p \subseteq p_{\alpha} \in S_S(N_\alpha)$, such that:
\begin{enumerate}
\item[(i)] $|N_\alpha| \leq |\alpha| +2^{|T|}+|M|$,
\item[(ii)] $\alpha < \beta$ implies $N_\alpha \prec N_\beta$,
\item[(iii)] $\alpha < \beta$ implies $p_{\alpha} \subseteq p_{\beta}$,
\item[(iv)] for every $\alpha$ one has  $c^{N_\alpha}({{p_1}_\alpha,\dots,{p_n}_\alpha}) = c^M({p_1,\dots, p_n})$ for all $p_1,\dots,p_n \in Y$,
\item[(v)] for every $\alpha$ the restriction map $\pi_\alpha \colon \cl(\{ p_{\alpha+1}: p \in Y\}) \to  \cl(\{ p_\alpha: p \in Y\})$ is not injective.
\end{enumerate}
Then $|\cl(\{ p_\alpha : p \in Y\})| \geq |\alpha|$. Taking $\alpha$
with $|\alpha| > \beth_3({l_S+ |T|+|M|})$, we get a contradiction to the fact that  $|\cl(\{ p_\alpha : p \in Y\})| \leq \beth_2({|S_S(M)|}) \leq \beth_3({l_S+ |T|+|M|})$.
\end{proof}
 
The next remark follows from  Lemma  \ref{lemma: technical lemma 1}.

\begin{rem}\label{remark: completion of technical lemma 2}
If an ($\aleph_0$-saturated) model $N$ satisfies the conclusion of Lemma \ref{lemma: technical lemma 2}, then so does every $\aleph_0$-saturated elementary extension $N'$ of $N$; this is witnessed by arbitrarily chosen $p^{N'} \in S_S(N')$ for $p \in Y$ so that $p^{N} \subseteq p^{N'}$ and $c^M({p_1,\dots,p_n}) = c^{N'}({{p_1}^{N'},\dots, {p_n}^{N'}})$ for all $p, p_1,\dots,p_n \in Y$.
It follows in particular that any $|N|$-saturated model extending $M$ satisfies the conclusion of the lemma.
\end{rem}

\begin{lem}\label{lemma: technical lemma 3}
Let $N$ be any small, $\aleph_0$-saturated model of $T$, and let $Z$ be any subset of $S_S(N)$. Then there exists a small, $\aleph_0$-saturated model $N' \succ N$ and an extension $p^{N'} \in S_S(N')$ of each $p \in Z$ such that:
\begin{enumerate}
\item for every $p_1,\dots,p_n \in Z$ one has $c^N({p_1,\dots,p_n}) = c^{N'}(p_1^{N'},\dots, p_n^{N'})$, and
\item for every $\aleph_0$-saturated $N'' \succ N'$ and for every $p^{N''} \in S_S(N'')$ extending $p^{N'}$ (where $p \in Z$) in such a way that $c^N({p_1,\dots,p_n}) = c^{N''}(p_1^{N''},\dots, p_n^{N''})$ for all $p_1,\dots,p_n \in Z$, for every $q_1,\dots,q_n \in  \cl(\{ p^{N''}: p \in Z\})$ one has  $c^{N''}({q_1,\dots,q_n}) = c^{N'}({{q_1}|_{N'},\dots, {q_n}|_{N'}})$.
\end{enumerate}
\end{lem}

\begin{proof}
Suppose this lemma is false.  Then
we can construct sequences $(N_\alpha)_{\alpha< \kappa}$ and $(p_{\alpha})_{\alpha < \kappa}$ for each $p \in Z$, where $N_\alpha \succ N$ 
is $\aleph_0$-saturated and $p \subseteq p_{\alpha} \in S_S(N_\alpha)$, such that:
\begin{enumerate}
\item[(i)] $|N_\alpha| \leq |\alpha| +2^{|T|}+|N|$,
\item[(ii)] $\alpha < \beta$ implies $N_\alpha \prec N_\beta$,
\item[(iii)] $\alpha < \beta$ implies $p_{\alpha} \subseteq p_{\beta}$,
\item[(iv)] for every $\alpha$ one has $c^{N_\alpha}({{p_1}_\alpha,\dots,{p_n}_\alpha}) = c^N({p_1,\dots, p_n})$ for all $p_1,\dots,p_n \in Z$,
\item[(v)] for every $\alpha$ there is $n$ and there are types $q_1,\dots,q_n \in \cl(\{ p_{\alpha+1}: p \in Z\})$ such that  $ c^{N_\alpha}({{q_1}|_{N_\alpha},\dots, {q_n}|_{N_\alpha}}) \subsetneq c^{N_{\alpha+1}}({q_1,\dots, q_n})$.
\end{enumerate}
For each $p \in Z$, let $p' = \bigcup_\alpha p_\alpha \in S_S(\bigcup_\alpha N_\alpha)$. Since $|\cl(\{ p' : p \in Z\})| \leq \beth_2(|Z|) \leq \beth_3(l_S+ |T|+|N|)< \kappa$ and the restriction map from $S_S(\bigcup_\alpha N_\alpha)$ to $S_S(N_\alpha)$ maps $\cl(\{ p' : p \in Z\})$ onto $\cl(\{ p_\alpha : p \in Z\})$, we can find a natural number $n$ and types $q_1,\dots,q_n \in \cl(\{ p' : p \in Z\})$ such that $ c^{N_\alpha}({{q_1}|_{N_\alpha},\dots, {q_n}|_{N_\alpha}}) \subsetneq c^{N_{\alpha+1}}({q_1|_{N_{\alpha+1}},\dots, q_n|_{N_{\alpha+1}}})$ 
for more than $2^{l_S+ 2^{|T|}}$ ordinals $\alpha$. But this contradicts Remark \ref{remark: bound of the number of configurations}. 
\end{proof}

The next remark follows from Lemma \ref{lemma: technical lemma 1}.

\begin{rem}\label{remark: completion of technical lemma 3}
If an ($\aleph_0$-saturated) model $N'$ satisfies the conclusion of Lemma \ref{lemma: technical lemma 3}, then so does every $\aleph_0$-saturated elementary extension of $N'$. 
\end{rem}

By Lemmas \ref{lemma: technical lemma 2}, \ref{lemma: technical lemma 3} and Remark \ref{remark: completion of technical lemma 2}, we get

\begin{cor}\label{corollary: technical corollary}
Let $M$ be any small, $\aleph_0$-saturated model of $T$, and let $Y$ be any subset of $S_S(M)$. Then there exists a small, $\aleph_0$-saturated
model $N' \succ M$ and an extension $p^{N'} \in S_S(N')$ of each $p \in Y$ such that: 
\begin{enumerate}
\item for every $p_1,\dots,p_n \in Y$ one has $c^M({p_1,\dots,p_n}) = c^{N'}({p_1^{N'},\dots, p_n^{N'}})$,
\item for every $\aleph_0$-saturated $N'' \succ N'$ and for every $p^{N''} \in S_S(N'')$ extending $p^{N'}$ in such a way that $c^M({p_1,\dots,p_n}) = c^{N''}(p_1^{N''},\dots, p_n^{N''})$ for all $p_1,\dots,p_n \in Y$, the restriction map $\pi \colon \cl(\{ p^{N''}: p \in Y)\}) \to  \cl(\{ p^{N'}: p \in Y\})$ is a homeomorphism,
\item for every $\aleph_0$-saturated $N'' \succ N'$ and for every $p^{N''} \in S_S(N'')$ extending $p^{N'}$ in such a way that $c^M({p_1,\dots,p_n}) = c^{N''}(p_1^{N''},\dots, p_n^{N''})$ for all $p_1,\dots,p_n \in Y$, for every $q_1,\dots,q_n \in  \cl(\{ p^{N''}: p \in Y\})$ one has  $c^{N''}({q_1,\dots,q_n}) = c^{N'}({{q_1}|_{N'},\dots, {q_n}|_{N'}})$.
\end{enumerate}
\end{cor}

\begin{proof}
First, we apply Lemma \ref{lemma: technical lemma 2} to get an $\aleph_0$-saturated model $N \succ M$ and types $p^N \in S_S(N)$ (for all $p \in Y$) satisfying the conclusion of Lemma \ref{lemma: technical lemma 2}. Next, we apply Lemma \ref{lemma: technical lemma 3} to this model $N$ and the set $Z:=\{p^N : p \in Y\}$, and we obtain an $\aleph_0$-saturated model $N'\succ N$ and types $p^{N'} \in S_S(N')$ (for all $p \in Y$) satisfying the conclusion of Lemma \ref{lemma: technical lemma 3}; in particular, $p^{N} \subseteq p^{N'}$ and $c^M(p_1,\dots,p_n) = c^{N'}(p_1^{N'},\dots, p_n^{N'})$ for all $p, p_1,\dots,p_n \in Y$. By Remark \ref{remark: completion of technical lemma 2}, $N'$ also satisfies the conclusion of Lemma \ref{lemma: technical lemma 2} which is witnessed by the types $p^{N'}$ for $p \in Y$. Therefore, the model $N'$ and the types $p^{N'}$ (for $p \in Y$) satisfy the conclusion of Corollary \ref{corollary: technical corollary}.
\end{proof}

As above, Lemma \ref{lemma: technical lemma 1} yields

\begin{rem}\label{remark: completion of technical corollary}
If an ($\aleph_0$-saturated) model $N'$ satisfies the conclusion of Corollary \ref{corollary: technical corollary}, then so does every $\aleph_0$-saturated elementary extension of $N'$. 
\end{rem}

The above technical lemmas and remarks are needed only to prove the following corollary. Recall that $X$ is a $\emptyset$-type-definable subset of $S$.

\begin{cor}\label{corollary: corollary of technical lemmas}
There exist closed, bounded (more precisely, of cardinality less than the degree of saturation of $\C_2$) subsets $R_{\C_1} \subseteq S_X(\C_1)$ and $R_{\C_2} \subseteq S_X(\C_2)$ with the following properties.
\begin{enumerate}
\item The restriction map $\pi_{12} \colon S_X(\C_1) \to S_X(\C_2)$ induces a homeomorphism from $R_{\C_1}$ onto $R_{\C_2}$, which will also be denoted by $\pi_{12}$.
\item For every $n$, $\{c^{\C_2}({p_1,\dots,p_n}) : p_1,\dots, p_n \in S_X(\C_2)\} = \{c^{\C_2}({p_1,\dots,p_n}) : p_1,\dots, p_n \in R_{\C_2}\}$.
\item  For every $n$, $c^{\C_1}({p_1,\dots,p_n}) = c^{\C_2}({\pi_{12}(p_1),\dots,\pi_{12}(p_n)})$ for all $p_1,\dots,p_n \in R_{\C_1}$.
\item For every $n$ and for every $p_1,\dots, p_n \in S_X(\C_1)$ there exist $q_1,\dots,q_n \in R_{\C_1}$ such that $c^{\C_1}({q_1,\dots,q_n}) \subseteq c^{\C_1}({p_1,\dots,p_n})$.
\end{enumerate}
\end{cor}

\begin{proof}
As at the beginning of Section \ref{section 3}, by Remark \ref{remark: bound of the number of configurations},
we can find a subset $P \subseteq \bigcup_{n \in \omega} S_X(\C_2)^n$ of bounded size (in fact, of size $\leq 2^{l_S+ 2^{|T|}}$) such that for every $n$
$$\{ c^{\C_2}(p_1,\dots,p_n): (p_1,\dots,p_n) \in P\} = \{ c^{\C_2}(p_1,\dots,p_n): (p_1,\dots,p_n) \in S_X(\C_2)^n\}.$$
Let $P_{\pr}$ be the collection of all types $p\in S_X(\C_2)$ for which there is $(p_1,\dots,p_n) \in P$ such that $p=p_i$ for some $i$.

Since $P_{\pr}$ is of bounded size, there is a small, $\aleph_0$-saturated $M \prec \C_2$ such that for every $n$ and for every $p_1,\dots,p_n \in P_{\pr}$, $c^{\C_2}({p_1,\dots,p_n}) = c^{M}({{p_1}|_M,\dots,{p_n}|_M})$. Put 
$$Y := \{ p|_M : p \in P_{\pr}\}\subseteq S_X(M).$$

Now, take a small, $\aleph_0$-saturated model $N'\succ M$ provided by Corollary \ref{corollary: technical corollary}. We can assume that $N' \prec \C_2$.
By Remark \ref{remark: completion of technical corollary}, the model $\C_2$ in place of $N'$ also satisfies the properties in Corollary \ref{corollary: technical corollary} (here $\C_2$ is small with respect to $\C_1$). So we get types $p^{\C_2} \in S_X(\C_2)$ (for all $p \in Y$) satisfying properties (1)-(3) from Corollary \ref{corollary: technical corollary} and such that $p \subseteq p^{\C_2}$ for all $p \in Y$. In particular:
\begin{enumerate}
\item[(i)] for every $p_1,\dots,p_n \in Y$ one has $c^M(p_1,\dots,p_n) = c^{\C_2}(p_1^{\C_2},\dots, p_n^{\C_2})$,
\item[(ii)] for every $p^{\C_1} \in S_S(\C_1)$ extending $p^{\C_2}$ in such a way that $c^M({p_1,\dots,p_n}) = c^{\C_1}(p_1^{\C_1},\dots, p_n^{\C_1})$ for all $p_1,\dots,p_n \in Y$, the restriction map $\pi \colon \cl(\{ p^{\C_1}: p \in Y)\}) \to  \cl(\{ p^{\C_2}: p \in Y\})$ is a homeomorphism,
\item[(iii)] for every $p^{\C_1} \in S_S(\C_1)$ extending $p^{\C_2}$ in such a way that $c^M({p_1,\dots,p_n}) = c^{\C_1}(p_1^{\C_1},\dots, p_n^{\C_1})$ for all $p_1,\dots,p_n \in Y$, for every $q_1,\dots,q_n \in  \cl(\{ p^{\C_1}: p \in Y\})$ one has  $c^{\C_1}({q_1,\dots,q_n}) = c^{\C_2}({q_1}|_{\C_2},\dots, {q_n}|_{\C_2})$.
\end{enumerate}

By (i) and Lemma \ref{lemma: technical lemma 1} applied to the models $\C_2 \prec \C_1$, for each $p \in Y$ we can find $p^{\C_1} \in S_{X}(\C_1)$ extending $p^{\C_2}$ and such that $c^M({p_1,\dots,p_n}) = c^{\C_1}(p_1^{\C_1},\dots,p_n^{\C_1})$ for all $p_1,\dots,p_n \in Y$.
Define 
$$R_{\C_1}:= \cl (\{ p^{\C_1} : p \in Y\}) \;\, \mbox{and}\,\; R_{\C_2}:= \cl (\{ p^{\C_2} : p \in Y\}).$$
We will check now that these sets have the desired properties.

Properties (1) and (3) follow from (ii) and (iii), respectively. Property (2) follows from (i) and the choice of $P$ and $M$. Property (4) follows from (1)-(3), but we will check it. Consider any $p_1,\dots,p_n \in S_X(\C_1)$. Clearly, $c^{\C_1}({p_1,\dots,p_n}) \supseteq c^{\C_2}({{p_1}|_{\C_2},\dots, {p_n}|_{\C_2}})$. By (2), we can find $r_1,\dots,r_n \in R_{\C_2}$ such that $c^{\C_2}({{p_1}|_{\C_2},\dots, {p_n}|_{\C_2}}) = c^{\C_2}({r_1,\dots,r_n})$. By (1), define $q_1 :=\pi_{12}^{-1}(r_1), \dots, q_n :=\pi_{12}^{-1}(r_n)$. Then, by (3), $c^{\C_2}({r_1,\dots,r_n}) = c^{\C_1}({q_1,\dots,q_n})$. We conclude that $c^{\C_1}({q_1,\dots,q_n}) \subseteq c^{\C_1}({p_1,\dots,p_n})$.
\end{proof}

In the rest of this section we take the notation from Corollary \ref{corollary: corollary of technical lemmas}. Property (1) will be used many times without mentioning. From now on, $\pi_{12}$ denotes the homeomorphism from $R_{\C_1}$ to $R_{\C_2}$.

\begin{lem}\label{lemma: for every f there is g}
For every $f \in EL_1$ such that $f[R_{\C_1}] \subseteq R_{\C_1}$ there exists $g \in EL_2$ such that $g[R_{\C_2}] \subseteq R_{\C_2}$ and $\pi_{12} \circ f|_{R_{\C_1}} = g \circ \pi_{12}|_{R_{\C_1}}$.
\end{lem}

\begin{proof}
Consider any pairwise distinct types $p_1,\dots, p_n \in R_{\C_2}$ and any formulas
\begin{equation*}
\varphi_1(\bar x, \bar a) \in \pi_{12} \circ f \circ \pi_{12}^{-1} (p_1), \dots, \varphi_n(\bar x, \bar a) \in   \pi_{12} \circ f \circ \pi_{12}^{-1} (p_n).
\end{equation*}
Let $\bar p =(p_1,\dots,p_n)$, $\bar \varphi = (\varphi_1(\bar x, \bar y),\dots, \varphi_n(\bar x, \bar y))$, and $\bar \varphi_{\bar a} = (\varphi_1(\bar x, \bar a),\dots, \varphi_n(\bar x, \bar a))$.

We claim that it is enough to find $\sigma_{\bar p, \bar \varphi_{\bar a}} \in \aut(\C_2)$ such that 
$$\varphi_1(\bar x, \bar a) \in \sigma_{\bar p, \bar \varphi_{\bar a}}(p_1), \dots, \varphi_n(\bar x, \bar a) \in \sigma_{\bar p, \bar \varphi_{\bar a}}(p_n).$$ 
Indeed, taking the partial directed order on all such pairs $(\bar p, \bar \varphi_{\bar a})$ (where formulas are treated as elements of the Lindenbaum algebra) given by 
$$\begin{array}{ll}
(\bar p, \bar \varphi_{\bar a}) \leq (\bar q, \bar \psi_{\bar b})  \iff & \\
  \bar p\;\,  \mbox{is a subsequence of}\;\,  \bar q \;\,  \mbox{and}\;\, (\forall i,j)(p_i=q_j \Rightarrow\; \models \psi_j(\bar x, \bar b) \rightarrow \varphi_i(\bar x, \bar a)), & 
\end{array}
$$
the net $(\sigma_{\bar p, \bar \varphi_{\bar a}})$ has a subnet convergent to some $g \in EL_2$ which satisfies the conclusion.

Now, we will explain how to find $\sigma_{\bar p, \bar \varphi_{\bar a}}$. By the choice of the $\varphi_i$'s, there exists $\tau_{\bar p, \bar \varphi_{\bar a}} \in \aut(\C_1)$ such that
$$\varphi_1(\bar x, \tau_{\bar p, \bar \varphi_{\bar a}}(\bar a)) \in \pi_{12}^{-1}(p_1), \dots, \varphi_n(\bar x, \tau_{\bar p, \bar \varphi_{\bar a}}(\bar a)) \in  \pi_{12}^{-1}(p_n).$$
This implies that for $q:=\tp(\bar a/\emptyset)$ one has $(\varphi_1,\ldots,\varphi_n,q)\in c^{\C_1}(\pi_{12}^{-1}(\bar p))$. 
Hence, by Property (3) in Corollary \ref{corollary: corollary of technical lemmas}, $(\varphi_1,\ldots,\varphi_n,q)\in c^{\C_2}(\bar p)$. This means that there exists $\sigma_{\bar p, \bar \varphi_{\bar a}}' \in \aut(\C_2)$ such that
$$\varphi_1(\bar x, \sigma_{\bar p, \bar \varphi_{\bar a}}'(\bar a)) \in p_1, \dots, \varphi_n(\bar x, \sigma_{\bar p, \bar \varphi_{\bar a}}'(\bar a)) \in p_n.$$
So, $\sigma_{\bar p, \bar \varphi_{\bar a}}:= \sigma_{\bar p, \bar \varphi_{\bar a}}'^{-1}$ is as required.
\end{proof}

\begin{lem}\label{lemma: for every g there is f}
For every $g \in EL_2$ such that $g[R_{\C_2}] \subseteq R_{\C_2}$ there exists $f \in EL_1$ such that $f[R_{\C_1}] \subseteq R_{\C_1}$ and $\pi_{12} \circ f|_{R_{\C_1}} = g \circ \pi_{12}|_{R_{\C_1}}$.
\end{lem}

\begin{proof}
We proceed  as in Lemma \ref{lemma: for every f there is g}. Consider
any pairwise distinct types $p_1,\dots,p_n\in R_{\C_1}$ and any
formulas 
$$\varphi_1(\bar x,\bar a)\in\pi_{12}^{-1}\circ g\circ
\pi_{12}(p_1),\dots,\varphi_n(\bar x,\bar a)\in\pi_{12}^{-1}\circ
g\circ\pi_{12}(p_n).$$
As in Lemma \ref{lemma: for every f there is g}, it is enough to find
$\sigma=\sigma_{\bar p,{\bar \varphi}_{\bar a}}\in\aut(\C_1)$ such
that $\varphi_i(\bar x,\bar a)\in\sigma(p_i)$ for $i=1,\dots,n$, where
$\bar p,\ {\bar \varphi}_{\bar a}$ are as in the proof there.

By property (3) from Corollary \ref{corollary: corollary of technical
  lemmas}, $c^{\C_2}(g\circ\pi_{12}(\bar
p))=c^{\C_1}(\pi_{12}^{-1}\circ g\circ \pi_{12}(\bar p))$, hence we
find ${\bar a}'\subset\C_2$ with ${\bar a}'\equiv_{\emptyset}\bar a$
and $\varphi_i(\bar x,{\bar a}')\in g\circ\pi_{12}(p_i)$ for
$i=1,\dots,n$. Then there is $\tau\in \aut(\C_2)$ with
$\varphi_i(\bar x,{\bar a}')\in\tau(\pi_{12}(p_i))$ for $i=1,\dots,n$.
Let ${\bar a}''=\tau^{-1}({\bar a}')$. Hence, ${\bar
  a}''\equiv_{\emptyset}{\bar a}'\equiv_{\emptyset}\bar a$ and
$\varphi_i(\bar x,{\bar a}'')\in\pi_{12}(p_i)\subseteq p_i$. Therefore,
there is $\sigma\in \aut(\C_1 )$ with $\sigma({\bar a}'')=\bar
a$. Clearly, $\sigma$ satisfies our demands.
\end{proof}

Recall that ${\mathcal M}_1$ and ${\mathcal M}_2$ are minimal left ideals of $EL_1$ and $EL_2$, respectively. 
By Properties (4) and (2) in Corollary \ref{corollary: corollary of technical lemmas} and Lemmas \ref{lemma: Im(eta) contained in R} and \ref{lemma: u in place of eta}, we get idempotents $u \in {\mathcal M}_1$ and $v' \in {\mathcal M}_2$ such that $\im (u) \subseteq R_{\C_1}$ and $\im (v') \subseteq R_{\C_2}$. Note that for any such idempotents $\im (u) = \im (u|_{R_{\C_1}})$ and $\im (v')=\im (v'|_{R_{\C_2}})$.

\begin{lem}\label{lemma: existence of v such that Im(u)=Im(v)}
There exists an idempotent $v \in {\mathcal M}_2$ such that $\pi_{12}[\im(u)]=\im(v)$.
\end{lem}

\begin{proof}
By Lemma \ref{lemma: for every f there is g}, there is $g \in EL_2$ such that  $\pi_{12} \circ u|_{R_{\C_1}} = g \circ \pi_{12}|_{R_{\C_1}}$. Then $\pi_{12}[\im(u)]=\im(g|_{R_{\C_2}})$. Let $h=gv' \in {\mathcal M}_2$. By Lemma \ref{lemma: hM=uM}(4), there is an idempotent $v \in h{\mathcal M}_2$ such that $h{\mathcal M}_2 = v{\mathcal M}_2$. We will show that  $\pi_{12}[\im(u)]=\im(v)$.

We clearly have
\begin{equation}\label{equation 3}
\im(v)=\im(h)\subseteq\im(g|_{R_{\C_2}})=\pi_{12}[\im(u)]\subseteq
R_{\C_2},
\end{equation}
hence also
\begin{equation}\label{equation 4}
\pi_{12}^{-1}[\im(v)]\subseteq\im(u).
\end{equation}
By Lemma \ref{lemma: for every g there is f}, there is $f\in EL_1$ such
that $f|_{R_{\C_1}}=\pi_{12}^{-1}\circ v\circ
\pi_{12}|_{R_{\C_1}}$. Then
$\im(f|_{R_{\C_1}})=\pi_{12}^{-1}[\im(v)]$. Choose an idempotent
$u'\in{\mathcal M}_1$ with $fu\in u'{\mathcal M}_1$. Then, using
(\ref{equation 4}) and Lemma \ref{lemma: hM=uM}, we have
\begin{equation}\label{equation 5}
\im(u')=\im(fu)\subseteq\im(f|_{R_{\C_1}})=\pi_{12}^{-1}[\im(v)]\subseteq\im(u).
 \end{equation}
By Lemma \ref{lemma: hM=uM}, $u=u'$ and
$\im(u)=\im(u')=\pi_{12}^{-1}[\im(v)]$. Hence, $\im(v)=\pi_{12}[\im(u)]$. 
\end{proof}



Take $v$ from the above lemma. Let $I_1:=\im(u)\subseteq R_{\C_1}$ and
$I_2:=\im(v) \subseteq R_{\C_2}$. Then, for every $f \in u{\mathcal
  M}_1$ and $g \in v{\mathcal M}_2$ we have $\im(f)=I_1$ and
$\im(g)=I_2$. By Lemma \ref{lemma: hM=uM}(3), we get group
monomorphisms $F_1\colon u{\mathcal M}_1\to\sym(I_1)$ and $F_2 \colon v{\mathcal
  M}_2\to\sym(I_2)$. By the choice of $v$ in Lemma \ref{lemma: existence of v such that Im(u)=Im(v)}, $\pi_{12}$ induces an isomorphism of
topological groups  $\pi_{12}'\colon \sym(I_1)\to\sym(I_2)$.


\begin{lem}\label{lemma: the image of Im(F1) equals Im(F2)}
$\pi_{12}'[\im(F_1)]=\im(F_2)$.
\end{lem}

\begin{proof}
($\subseteq$) Take  $f \in u{\mathcal M}_1$. By Lemma \ref{lemma: for every f there is g}, we can find $g \in EL_2$ such that $\pi_{12} \circ f|_{R_{\C_1}} = g \circ \pi_{12}|_{R_{\C_1}}$. Hence,  $\pi_{12} \circ f|_{I_1} = g \circ \pi_{12}|_{I_1}$. By the choice of $v$ in Lemma \ref{lemma: existence of v such that Im(u)=Im(v)} and the fact that $v|_{I_2}=\id_{I_2}$, this implies  that $\pi_{12} \circ f|_{I_1} = vgv \circ \pi_{12}|_{I_1}$. Note that $vgv \in v{\mathcal M}_2$, and we get
$$\pi_{12}'(F_1(f)) = \pi_{12}'(f|_{I_1})=\pi_{12} \circ f \circ \pi_{12}^{-1}|_{I_2}=vgv|_{I_2}=F_2(vgv).$$
($\supseteq$) Take $g \in v{\mathcal M}_2$.  By Lemma \ref{lemma: for every g there is f}, we can find $f \in EL_1$ such that $f[R_{\C_1}] \subseteq R_{\C_1}$ and $\pi_{12} \circ f|_{R_{\C_1}} = g \circ \pi_{12}|_{R_{\C_1}}$. Hence,  $\pi_{12} \circ f|_{I_1} = g \circ \pi_{12}|_{I_1}$. By the choice of $v$ in Lemma \ref{lemma: existence of v such that Im(u)=Im(v)}, the fact that $u|_{I_1}=\id_{I_1}$, and the fact that $\pi_{12} \colon R_{\C_1} \to R_{\C_2}$ is a bijection, this implies that $\pi_{12} \circ ufu|_{I_1} = g \circ \pi_{12}|_{I_1}$. Note that $ufu \in u{\mathcal M}_1$, and we get  
$$\pi_{12}'(F_1(ufu)) = \pi_{12}'(ufu|_{I_1})=\pi_{12} \circ ufu \circ \pi_{12}^{-1}|_{I_2}=g|_{I_2}=F_2(g).\qedhere$$
\end{proof}

As a conclusion, we get 
absoluteness of the isomorphism type of the Ellis group.

\begin{cor}\label{corollary: absoluteness of the Ellis group}
The map $F_2^{-1} \circ \pi_{12}' \circ F_1 \colon u{\mathcal M}_1 \to v{\mathcal M}_2$ is a group isomorphism. Thus the Ellis group of the flow $S_X(\C)$ does not depend (up to isomorphism) on the choice of the monster model $\C$.
\end{cor}

By Proposition \ref{proposition: unboundedly many sorts} this implies the second part of Theorem \ref{theorem: main theorem 1}.

\begin{cor}\label{corollary: absoluteness of the Ellis group for unboundedly many sorts}
Here, let $S$ be a product of an arbitrary (possibly unbounded) number of sorts and $X$ be a $\emptyset$-type-definable subset of $S$. Then the Ellis group of the flow $S_X(\C)$ does not depend (up to isomorphism) on the choice of the monster model $\C$.
\end{cor}

By virtue of Corollaries \ref{corollary: first part of main theorem 1} and \ref{corollary: absoluteness of the Ellis group} together with Proposition \ref{proposition: unboundedly many sorts}, in order to complete the proof of Theorem \ref{theorem: main theorem 1}, it remains to show 

\begin{prop}\label{proposition: topological absoluteness} 
$F_2^{-1} \circ \pi_{12}' \circ F_1 \colon u{\mathcal M}_1 \to v{\mathcal M}_2$ is a homeomorphism, where $u{\mathcal M}_1$ and $v{\mathcal M}_2$ are equipped with the $\tau$-topology.
\end{prop}

The proof will consist of two lemmas.
Let $F_1' \colon {\mathcal M}_1 \to S_X(\C_1)^{I_1}$ be the
restriction map (so it is an extension of $F_1$ to a bigger
domain). As in the case of  $F_1$, we easily get that $F_1'$ is
injective: if $F_1'(f)=F_1'(g)$ for some $f,g \in {\mathcal M}_1$
(which implies that $f=fu$ and $g=gu$), then $f|_{I_1} = g|_{I_1}$,
so, since $\im (u) = I_1$, we get $f=fu=gu=g$.

$\aut(\C_1)$ acts on $EL_1$ and on ${\mathcal M}_1$; this induces an
action of $\aut(\C_1)$ also on $\im(F_1')\subseteq
S_X(\C_1)^{I_1}$: for $\sigma\in\aut(\C_1)$ and
$d\in{\mathcal M}_1$ define $\sigma(F_1'(d))$ as $F_1'(\sigma
d)=(\sigma d)|_{I_1}$.

\begin{lem}\label{lemma: description of tau-closures}
Let $D \subseteq u{\mathcal M}_1$. Then $F_1[\cl_\tau (D)]$ is the set of all $f \in \im (F_1)$ for which there exist nets $(\sigma_i)_i$ in $\aut(\C_1)$ and $(d_i)_i$ in $D$ such that $\lim \sigma_i'= \id_{I_1}$ (where $\sigma_i'$ is the element of $S_X(\C_1)^{I_1}$ induced by $\sigma_i$) and $\lim \sigma_i(F_1(d_i)) = f$. The same statement also holds for $\C_2$, $v{\mathcal M}_2$, $I_2$ and $F_2$ in place of $\C_1$, $u{\mathcal M}_1$, $I_1$, and $F_1$, respectively.
\end{lem}

\begin{proof}
($\subseteq$) Take any $g \in \cl_\tau (D)$ and $f=F_1(g)$. By the definition of the $\tau$-closure (see \cite[Subsection 1.1]{KrPiRz}), there exist nets $(\sigma_i)_i$ in $\aut(\C_1)$ and $(d_i)_i$ in $D$ such that $\lim \sigma_i = u$ and $\lim \sigma_i(d_i)=g$. Since $u|_{I_1} = \id_{I_1}$ and $F_1'$ is continuous with respect to the 
product topologies on the domain and on the target, we see that $\lim \sigma_i' = \id_{I_1}$ and $\lim\sigma_i(F_1(d_i))=\lim\sigma_i(F_1'(d_i))=\lim F_1'(\sigma_i(d_i)) = F_1'(g)=f$.\\[1mm]
($\supseteq$) Consider any $f \in \im(F_1)$  for which there exist nets $(\sigma_i)_i$ in $\aut(\C_1)$ and $(d_i)_i$ in $D$ such that $\lim \sigma_i'= \id_{I_1}$ and $\lim \sigma_i(F_1(d_i)) = f$.

There is a subnet $(\tau_j)_j$ of $(\sigma_i)_i$ such that $h:=\lim \tau_j$ exists in $EL_1$ and for the corresponding subnet $(e_j)_j$ of $(d_i)_i$ the limit $g:=\lim_j \tau_j(e_j)$ also exists (and belongs to ${\mathcal M}_1$). 

Now, we will use the circle operation $\circ$ from the definition of the $\tau$-closure (see \cite[Definition 1.3]{KrPiRz}). To distinguish it from composition of functions (for which the symbol $\circ$ is reserved in this paper), the circle operation will be denoted by $\bullet$.

We have that $g \in h \bullet D \subseteq h \bullet (u \bullet D) \subseteq (hu) \bullet D$. But $h|_{I_1} = \lim \tau_j' = \lim \sigma_i' = \id_{I_1}$ and $\im (u)=I_1$, so $hu=u$. Hence, $g \in u \bullet D$. We also see that $f=\lim_i \sigma_i(F_1(d_i)) = \lim_j \tau_j(F_1(e_j)) = \lim_j F_1'(\tau_j(e_j)) = F_1'(g)$. But $f \in \im(F_1) = F_1'[u{\mathcal M}_1]$ and $F_1'$ is injective. Therefore, $g \in (u \bullet D) \cap u{\mathcal M}_1 =: \cl_\tau(D)$. Hence, $f=F_1(g) \in F_1[\cl_\tau(D)]$.  
\end{proof}

Equip $\im(F_1)$ and $\im(F_2)$ with the topologies induced by the isomorphisms $F_1$ and $F_2$ from the $\tau$-topologies on $u{\mathcal M}_1$ and $v{\mathcal M}_2$, respectively. The corresponding closure operators will be denoted by $\cl_\tau^1$ and $\cl_\tau^2$. They are given by the right hand side of Lemma \ref{lemma: description of tau-closures}. The next lemma will complete the proof of Proposition \ref{proposition: topological absoluteness}. 

\begin{lem}
For any $D \subseteq \im (F_1)$, $\pi_{12}'[\cl_\tau^1(D)] = \cl_\tau^2(\pi_{12}'[D])$.
\end{lem}

The proof of this lemma will be an elaboration on the proofs of Lemmas \ref{lemma: for every f there is g} and \ref{lemma: for every g there is f}.

\begin{proof}
($\subseteq$) Take any $f \in \cl_\tau^1 (D)$. By Lemma \ref{lemma: description of tau-closures}, there exist nets $(\sigma_i)_i$ in $\aut(\C_1)$ and $(d_i)_i$ in $D$ such that $\lim \sigma_i'= \id_{I_1}$ (where $\sigma_i'$ is the element of $S_X(\C_1)^{I_1}$ induced by $\sigma_i$) and $\lim \sigma_i(d_i) = f$.

Consider any pairwise distinct $p_1,\dots, p_n \in I_2$ and any formulas
$$
\begin{array}{l}
\varphi_1(\bar x, \bar a) \in \pi_{12}'(f) (p_1), \dots, \varphi_n(\bar x, \bar a) \in   \pi_{12}'(f) (p_n),\\
\psi_1(\bar x, \bar a) \in p_1,\dots, \psi_n(\bar x, \bar a) \in p_n.
\end{array}
$$

Let $\bar p =(p_1,\dots,p_n)$, $\bar \varphi = (\varphi_1(\bar x, \bar y),\dots, \varphi_n(\bar x, \bar y))$, $\bar \psi=(\psi_1(\bar x, \bar y),\dots, \psi_n(\bar x, \bar y))$, and $\bar \varphi_{\bar a} = (\varphi_1(\bar x, \bar a),\dots, \varphi_n(\bar x, \bar a))$, $\bar \psi_{\bar a}=(\psi_1(\bar x, \bar a),\dots, \psi_n(\bar x, \bar a))$.

We claim that it is enough to find $\sigma_{\bar p, \bar \varphi_{\bar a}, \bar \psi_{\bar a}} \in \aut(\C_2)$ and an index $i_{\bar p, \bar \varphi_{\bar a},\bar \psi_{\bar a}}$ such that 
$$
\begin{array}{l}
\varphi_1(\bar x, \bar a) \in \sigma_{\bar p, \bar \varphi_{\bar a},\bar \psi_{\bar a}}(\pi_{12}'(d_{i_{\bar p, \bar \varphi_{\bar a},\bar \psi_{\bar a}}})(p_1)), \dots, \varphi_n(\bar x, \bar a) \in \sigma_{\bar p, \bar \varphi_{\bar a},\bar \psi_{\bar a}}(\pi_{12}'(d_{i_{\bar p, \bar \varphi_{\bar a},\bar \psi_{\bar a}}})(p_n)),\\
\psi_1(\bar x, \bar a) \in \sigma_{\bar p, \bar \varphi_{\bar a}, \bar \psi_{\bar a}}(p_1), \dots, \psi_n(\bar x, \bar a) \in \sigma_{\bar p, \bar \varphi_{\bar a}, \bar \psi_{\bar a}}(p_n).
\end{array}
$$ 
Indeed, taking the partial directed order on all triples $(\bar p, \bar \varphi_{\bar a}, \bar \psi_{\bar a})$ defined as in the proof of Lemma \ref{lemma: for every f there is g}, 
the net $(\sigma_{\bar p, \bar \varphi_{\bar a}, \bar \psi_{\bar a}})$ satisfies $\lim \sigma_{\bar p, \bar \varphi_{\bar a}, \bar \psi_{\bar a}}' = \id_{I_2}$ and $\lim \sigma_{\bar p, \bar \varphi_{\bar a}, \bar \psi_{\bar a}}(\pi_{12}'(d_{i_{\bar p, \bar \varphi_{\bar a},\bar \psi_{\bar a}}})) = \pi_{12}'(f)$, which by Lemmas \ref{lemma: description of tau-closures} and \ref{lemma: the image of Im(F1) equals Im(F2)} implies that $\pi_{12}'(f) \in \cl_\tau^2(\pi_{12}'[D])$.

Now, we will explain how to find $\sigma_{\bar p, \bar \varphi_{\bar a}, \bar \psi_{\bar a}}$. By the fact that $\lim \sigma_i'= \id_{I_1}$ and $\lim \sigma_i(d_i) = f$ and by the choice of $\bar \varphi_{\bar a}$ and $\bar \psi_{\bar a}$, there is an index $i_{\bar p, \bar \varphi_{\bar a},\bar \psi_{\bar a}}$ for which  
$$
\begin{array}{l}
\varphi_1(\bar x, \sigma_{i_{\bar p, \bar \varphi_{\bar a}, \bar \psi_{\bar a}}}^{-1}(\bar a)) \in d_{i_{\bar p, \bar \varphi_{\bar a},\bar \psi_{\bar a}}}(\pi_{12}^{-1}(p_1)), \dots, \varphi_n(\bar x, \sigma_{i_{\bar p, \bar \varphi_{\bar a}, \bar \psi_{\bar a}}}^{-1}(\bar a)) \in d_{i_{\bar p, \bar \varphi_{\bar a},\bar \psi_{\bar a}}}(\pi_{12}^{-1}(p_n)),\\
\psi_1(\bar x, \sigma_{i_{\bar p, \bar \varphi_{\bar a}, \bar \psi_{\bar a}}}^{-1}(\bar a)) \in \pi_{12}^{-1}(p_1), \dots, \psi_n(\bar x, \sigma_{i_{\bar p, \bar \varphi_{\bar a}, \bar \psi_{\bar a}}}^{-1}(\bar a)) \in \pi_{12}^{-1}(p_n).
\end{array}
$$

Since $\pi_{12}^{-1}(p_1),\dots, \pi_{12}^{-1}(p_n) \in R_{\C_1}$ and $d_{i_{\bar p, \bar \varphi_{\bar a},\bar \psi_{\bar a}}}(\pi_{12}^{-1}(p_1)), \dots, d_{i_{\bar p, \bar \varphi_{\bar a},\bar \psi_{\bar a}}}(\pi_{12}^{-1}(p_n)) \in R_{\C_1}$, the existence of  $\sigma_{\bar p, \bar \varphi_{\bar a}, \bar \psi_{\bar a}} \in \aut(\C_2)$ with the required properties follows easily from Property (3) in Corollary \ref{corollary: corollary of technical lemmas} (similarly to the final part of the proof of Lemma \ref{lemma: for every f there is g}).\\[1mm]
($\supseteq$) Assume $g\in \cl^2_{\tau}(\pi_{12}'[D])$. Our goal is to
prove that the function $f:=\pi_{12}'^{-1}(g)=\pi_{12}^{-1}\circ
g\circ\pi_{12}|_{I_1}$ belongs to $\cl_{\tau}^1(D)$.

By Lemma \ref{lemma: description of tau-closures} there exist nets
$(\sigma_i)_i$ in $\aut(\C_2)$ and $(d_i)_i$ in $D$ such that  $\lim
\sigma_i'=\id_{I_2}$ (where $\sigma_i'\in S_X(\C_2)^{I_2}$ is induced
by $\sigma_i$) and $\lim \sigma_i(\pi_{12}'(d_i))=g$. Consider any pairwise
distinct $p_1,\dots,p_n\in I_1$ and any formulas
$$\varphi_1(\bar x,\bar a)\in f(p_1),\dots,\varphi_n(\bar x,\bar a)\in
f(p_n),$$
$$\psi_1(\bar x,\bar a)\in p_1,\dots,\psi_n(\bar x,\bar a)\in p_n.$$
As in the proof of ($\subseteq$), it is enough to find
$\sigma=\sigma_{\bar p,{\bar\varphi}_{\bar a},{\bar\psi}_{\bar
    a}}\in\aut(\C_1)$ and $i^*=i_{\bar p,{\bar\varphi}_{\bar
    a},{\bar\psi}_{\bar a}}$ such that $\varphi_j(\bar x,\bar
a)\in\sigma(d_{i^*}(p_j))$ and $\psi_j(\bar x,\bar a)\in\sigma(p_j)$
for $j=1,\dots,n$. We explain how to find $\sigma$ and $i^*$.

We have that $p_j$ and $f(p_j)$ belong to $I_1\subseteq R_{\C_1}$ and
$\pi_{12}(f(p_j))=g(\pi_{12}(p_j))$. Hence, by Property (3) in
Corollary \ref{corollary: corollary of technical lemmas}, we get that
$$c^{\C_1}(f(\bar p),\bar p)=c^{\C_2}(\pi_{12}(f(\bar
    p)),\pi_{12}(\bar p))=c^{\C_2}(g(\pi_{12}(\bar p)),\pi_{12}(\bar
    p)).$$
Hence, there is ${\bar a}'$ in $\C_2$ with ${\bar
  a}'\equiv_{\emptyset}\bar a$ such that $\varphi_j(\bar x,{\bar
  a}')\in g(\pi_{12}(p_j))$ and $\psi_j(\bar x,{\bar
  a}')\in\pi_{12}(p_j)$ for $j=1,\dots,n$. By the fact that $\lim
\sigma_i'=\id_{I_2}$ and $\lim\sigma_i(\pi_{12}'(d_i))=g$, there is
an index $i^*$ such that $\varphi_j(\bar x,{\bar a}')\in
\sigma_{i^*}(\pi_{12}'(d_{i^*}))(\pi_{12}(p_j))$ and $\psi_j(\bar x,{\bar
  a}')\in\sigma_{i^*}(\pi_{12}(p_j))$ for $j=1,\dots,n$.

Let ${\bar a}''=\sigma_{i^*}^{-1}({\bar a}')$. Hence, $\varphi_j(\bar
x,{\bar a}'')\in\pi_{12}(d_{i^*}(p_j))\subseteq d_{i^*}(p_j)$ and
$\varphi_j(\bar x,{\bar a}'')\in \pi_{12}(p_j) \subseteq p_j$ for $j=1,\dots,n$, and also
${\bar a}''\equiv_{\emptyset}{\bar a}'\equiv_{\emptyset}\bar
a$. Therefore, there is $\sigma\in\aut(\C_1)$ with $\sigma({\bar
  a}'')=\bar a$. Clearly, $\sigma$ satisfies our demands. 
\end{proof}

The proof of Proposition \ref{proposition: topological absoluteness} has been completed.  As was mentioned before, Corollaries \ref{corollary: first part of main theorem 1}, \ref{corollary: absoluteness of the Ellis group}, Proposition \ref{proposition: topological absoluteness}, and Proposition \ref{proposition: unboundedly many sorts} give us Theorem \ref{theorem: main theorem 1}, which implies immediately the following

\begin{cor}
Here, let $S$ be a product of an arbitrary (possibly unbounded) number of sorts and $X$ be a $\emptyset$-type-definable subset of $S$. Then, the quotient of the Ellis group $u{\mathcal M}$ of the flow $S_X(\C)$ by the intersection $H(u{\mathcal M})$ of the $\tau$-closures of the $\tau$-neighborhoods of the identity does not depend (as a compact topological group) on the choice of $\C$.
\end{cor}

Theorem \ref{theorem: main theorem 1} together with Proposition \ref{proposition: invariants for c isomorphic to invariants for S} imply that the answer to Question \ref{question: main question 1} is positive.

\begin{cor}
The Ellis groups of the flows $S_{\bar c}(\C)$ and $S_{\bar \alpha}(\C)$ (where $\bar c$ is an enumeration of $\C$ and $\bar \alpha$ is a short tuple in $\C$) are of bounded size and do not depend (as groups equipped with the $\tau$-topology) on the choice of the monster model $\C$. Thus, the quotients of these groups by the intersections of the $\tau$-closures of the $\tau$-neighborhoods of the identity are also absolute as topological groups.
\end{cor}

In order for the proof of Theorem \ref{theorem: main theorem 1} to
proceed smoothly, we assumed that our monster models are
$\kappa$-saturated and strongly $\kappa$-homogeneous, where $\kappa$
was a strong limit cardinal larger than $|T|$. A closer inspection of
the proof shows we may relax the restriction on $\kappa$ to assume
just that $\kappa>\beth_5(|T|)$.

\section{Boundedness and absoluteness of minimal left ideals}\label{section 4}

We will start from an example where the minimal left ideals in the
Ellis semigroup of the flow $S_X(\C)$ (even of $S_\alpha(\C)$ where
$\alpha \in \C$) are of unbounded size, answering negatively Question
\ref{question: main question 2}(i). Then we give a characterization in
terms of Lascar invariant types of when these ideals are of bounded
size, and, if it is the case, we find an explicit bound on the
size. The main part of this section is devoted to the proof of Theorem
\ref{theorem: main theorem 2}. For this we will use strong heir extensions and the aforementioned characterization in terms of Lascar invariant types.

\begin{ex}\label{example: unbounded minimal left ideal}
Consider $M=(S^1, R(x,y,z))$, where $S^1$ is the unit circle on the plane, and $R(x,y,z)$ is the circular order, i.e. $R(a,b,c)$ holds if and only if $a,b,c$ are pairwise distinct and $b$ comes before $c$ going around the circle clockwise starting at $a$. Let $\C \succ M$ be a monster model. It is well-known that this structure has quantifier elimination, so there is a unique complete 1-type over $\emptyset$. Moreover, $\Th(M)$ has NIP.

$S_1(\C)$ is a point-transitive $\aut(\C)$-flow, consisting of
extensions of the unique type in $S_1(\emptyset)$.
Let $NA$  be the set of all non-algebraic types in $S_1(\C)$ (they
correspond to cuts in the non-standard $S^1(\C)$). Let
$C$ be the set of all constant functions $S_1(\C)\to S_1(\C)$, with
values in $NA$. We claim that:
\begin{enumerate}
\item $C\subseteq EL(S_1(\C))$.
\item $EL(S_1(\C))\eta=C$ for every $\eta\in C$.
\item $C$ is the unique minimal left ideal in $EL(S_1(\C))$, and it is
  unbounded.
\item The Ellis group of $S_1(\C)$ is trivial (so bounded).
\item The minimal left ideals in $EL(S_{\bar c}(\C))$ are of unbounded
  size, where $\bar c$ is an enumeration of $\C$.
\end{enumerate}
\begin{proof} (1) For example let $q\in NA$ be the right cut at
  $1$. So $q$ is generated by formulas $\varphi(x,a)=R(1,x,a)$, where
  $a\in S^1(\C)\setminus\{1\}$. Consider any $p_1,\dots,p_n\in
  S_1(\C)$, let $\bar p=(p_1,\dots,p_n)$ and let $a\in
  S^1(\C)\setminus\{1\}$.  By quantifier elimination and the strong
  homogeneity of $\C$, one can easily find $\sigma_{\bar
    p,\varphi(x,a)} \in \aut(\C)$ such that $\varphi(x,a) \in
  \sigma_{\bar p,\varphi(x,a)}(p_i)$ for all $i=1,\dots,n$. Then
  $(\sigma_{\bar p,\varphi(x,a)})$ is a net (with the obvious ordering
  on the indexes) which converges to the function $\eta \in
  EL(S_{1}(\C))$ constantly equal to $q$.

(2) ($\subseteq$) is clear, and ($\supseteq$) follows from (1).

(3) $C$ is a minimal left ideal by (2). It is unbounded as $NA$ is
  unbounded. For every $\eta\in C$ and $\eta'\in EL(S_1(\C))$ we have
  $\eta\circ\eta'=\eta$, hence $C\subseteq EL(S_1(\C))\eta'$, and so $C$ is
  a unique minimal left ideal.

(4) Every element of $C$ is an idempotent.

(5) follows from (3) and Corollary \ref{corollary: the size for alpha
    bounded by the size for c}.
\end{proof}
\end{ex}

Knowing that minimal left ideals may be of unbounded size, a natural goal is to understand when they are of bounded size, and whether boundedness is absolute. So these are our next goals (as to the first goal, a more satisfactory answer is obtained in Section \ref{section 6} under the NIP assumption).

As usual, let $\C$ be a monster model of an arbitrary theory $T$. Let
$S$ be a product of some number of sorts (possibly unbounded, with repetitions allowed), and let $X$ be a $\emptyset$-type-definable subset of $S$. In this section, by $EL$ we will denote the Ellis semigroup $EL(S_X(\C))$. Let ${\mathcal M}$ be a minimal left ideal in $EL$. By $l_S$ denote the length of $S$ (i.e. the number of factors in the product $S$). Let $I_L$ denote the set of all Lascar invariant types in $S_X(\C)$, i.e. types which are invariant under $\autf_L(\C)$.  Note that $I_L$ is either empty or an $\aut(\C)$-subflow of $S_X(\C)$ (in particular,  it is closed) and it is of bounded size. 

\begin{prop}\label{proposition: boundedness of the index implies AutfL(C) is contained}
If $G$ is a closed, bounded index subgroup of $\aut(\C)$ (with $\aut(\C)$ equipped with the pointwise convergence topology), then $\autf_L(\C) \leq G$.
\end{prop}

\begin{proof}
Let $\sigma_i$, $i<\lambda$, by a set of representatives of right costs of $G$ in $\aut(\C)$ (so $\lambda$ is bounded). Then 
$$G':= \bigcap_{\sigma \in \aut(\C)} G^\sigma = \bigcap_{i<\lambda} G^{\sigma_i}$$
is a closed, normal, bounded index subgroup  of $\aut(\C)$.

Consider any $\sigma\in\autf_L(\C)$ and any finite $\bar
a\subseteq\C$. The orbit equivalence relation of $G'$ on the set of
realizations of $\tp(\bar a/\emptyset)$ is a bounded invariant
equivalence relation, so it is coarser than $E_L$. Therefore, there
exists $\tau_{\bar a}\in G'$ with $\tau_{\bar a}(\bar a)=\sigma(\bar
a)$. 
%
%
This shows that $\sigma$ lies in the closure of $G'$, and hence $\sigma \in G'$ (as $G'$ is closed). In this way, we proved that $\autf_L(\C) \leq G' \leq G$.
\end{proof}

\begin{rem}\label{remark: stabilizers are closed}
For every $p \in S_X(\C)$, $\stab_{\aut(\C)}(p): = \{ \sigma \in \aut(\C): \sigma(p)=p\}$ is a closed subgroup of $\aut(\C)$.
\end{rem}

The next corollary follows immediately from Proposition \ref{proposition: boundedness of the index implies AutfL(C) is contained} and Remark \ref{remark: stabilizers are closed}.

\begin{cor}\label{corollary: boundedness of the orbit equivalent to Lascar invariance}
Let $p \in S_X(\C)$. Then the orbit $\aut(\C) p$ is of bounded size if and only if $p \in I_L$.
\end{cor}

The next proposition is our characterization of boundedness of the minimal left ideals, and it yields an explicit bound on their size.

\begin{prop}\label{proposition: characterization of boundedness of minimal ideals}
The following conditions are equivalent.\\
i) The minimal left ideal ${\mathcal M}$ is of bounded size.\\
ii) For every $\eta \in {\mathcal M}$ the image $\im(\eta)$ is contained in $I_L$.\\
iii) For some $\eta \in EL$ the image $\im(\eta)$ is contained in $I_L$.\\
 Moreover, if  ${\mathcal M}$ is of bounded size, then $|{\mathcal M}|\leq \beth_3(|T|)$.
\end{prop}

\begin{proof}
Suppose first that ${\mathcal M}$ is of bounded size. Then for every $\eta \in {\mathcal M}$ and for every $p \in S_X(\C)$ the orbit $\aut(\C)\eta(p)$ is of bounded size, so $\eta(p) \in I_L$ in virtue of Corollary \ref{corollary: boundedness of the orbit equivalent to Lascar invariance}.

Now, suppose that for some $\eta \in EL$ the image $\im(\eta)$ is contained in $I_L$. This means that the size of the orbit $\aut(\C) \eta$ is bounded by $|\gal_{L}(T)| \leq 2^{|T|}$, and so  $|\cl(\aut(\C) \eta)|\leq \beth_3(|T|)$. So the size of ${\mathcal M}$ is bounded by $\beth_3(|T|)$, because it is equal to the size of a minimal $\aut(\C)$-subflow 
contained in $\cl(\aut(\C) \eta)$. 
\end{proof}

The next corollary will be used later.

\begin{cor}\label{corollary: consequence of the characterization of boundedness}
The minimal left ideal ${\mathcal M}$ is of bounded size if and only if for every types $q_1,\dots,q_n \in S_X(\C)$ and formulas $\varphi_1(\bar x, \bar y_1), \dots, \varphi_n(\bar x, \bar y_n)$, for every $(\bar a_1, \bar b_1) \in E_L, \dots, (\bar a_n, \bar b_n) \in E_L$ there exists $\sigma \in \aut(\C)$ such that $$\sigma(\varphi_1(\bar x, \bar a_1) \wedge \neg \varphi_1(\bar x, \bar b_1)) \notin q_1, \dots, \sigma(\varphi_n(\bar x, \bar a_n) \wedge \neg \varphi_n(\bar x, \bar b_n)) \notin q_n.$$
\end{cor}

\begin{proof}
Suppose first that ${\mathcal M}$ is of bounded size. Take any $\eta \in {\mathcal M}$. Then, by Proposition \ref{proposition: characterization of boundedness of minimal ideals}, for every $q \in S_X(\C)$, $\eta(q)$ is Lascar invariant. So, for any data as on the right hand side of the required equivalence we have that $\varphi_i(\bar x, \bar a_i) \wedge \neg \varphi_i(\bar x, \bar b_i) \notin \eta (q_i)$ for $i=1,\dots,n$. Hence, since $\eta$ is approximated by elements of $\aut(\C)$, there exists $\sigma \in \aut(\C)$ with the required property.

Now, assume that the right hand side holds. 
The $\sigma$ from our assumptions depends on $\bar
q:=(q_1,\dots,q_n),\ \bar\varphi:=(\varphi_1,\dots,\varphi_n),\ \bar
a:=(\bar a_1,\dots,\bar a_n)$ and $\bar b:=(\bar b_1,\dots,\bar b_n)$. So we may write
$\sigma$ as $\sigma_{\bar q,\bar\varphi,\bar a,\bar b}$. We order the
set of indexes $(\bar q,\bar\varphi,\bar a,\bar b)$ by:
$$((q_i)_{i\leq m},(\varphi_i)_{i \leq m},(\bar a_i)_{i \leq m},(\bar b_i)_{i \leq m})\leq ((q_i')_{i\leq n},(\varphi_i')_{i \leq n},(\bar a_i')_{i \leq n},(\bar b_i')_{i \leq n})$$
if and only if $(q_i,\varphi_i,\bar a_i,\bar b_i)_{i \leq m}$ is a subsequence of $(q_i',\varphi_i',\bar a_i',\bar b_i')_{i \leq n}$.
Taking the limit of a convergent subnet of the net $(\sigma^{-1}_{\bar
  q,\bar\varphi,\bar a,\bar b})$,
 we get an element $\eta' \in EL$ such that $\im(\eta') \subseteq I_L$. 
Therefore, ${\mathcal M}$ is of bounded size by Proposition \ref{proposition: characterization of boundedness of minimal ideals}.
\end{proof}

From now on, fix any monster models $\C_1 \succ \C_2$ of the theory $T$. (For the purpose of our main results, without loss of generality, we can always assume that $\C_1$ is a monster model with respect to the size of $\C_2$.) Let $EL_1=EL(S_X(\C_1))$ and $EL_2=EL(S_X(\C_2))$.
By $I_{L,\C_i}$ denote the set of all Lascar invariant types in $S_X(\C_i)$, for $i =1,2$.

The following remark is folklore.

\begin{rem}
Let $\pi_{12}\colon S_X(\C_1) \to S_X(\C_2)$ be the restriction
map. Then $\pi_{12}|_{I_{L,\C_1}} \colon  I_{L,\C_1} \to I_{L,\C_2}$
is a homeomorphism.
(For $p\in I_{L,\C_2}$, $(\pi_{12}|_{I_{L,\C_1}})^{-1}(p)$ is the
unique $M$-invariant extension of $p$ to a type in $S_X(\C_1)$, where
$M\prec\C_2$ is small.)
\end{rem}

From now on, $\pi_{12}$ will denote the above homeomoprhism from $I_{L,\C_1}$ to $I_{L,\C_2}$. 

\begin{prop}\label{proposition: M_1 bounded iff M_2 bounded}
The minimal left ideals in $EL_1$ are of bounded size if and only if
the minimal left ideals in $EL_2$ are of bounded size.
\end{prop}
\begin{proof}
($\rightarrow$) Suppose the minimal left ideals in $EL_2$ are of
  unbounded size. By Corollary \ref{corollary: consequence of the
    characterization of boundedness}, there are types $q_1,\dots,q_n\in
  S_X(\C_2)$, formulas $\varphi_1(\bar x,\bar y_1),\dots,\varphi_n(\bar
  x, \bar y_n)$ and tuples $(\bar a_1,\bar b_1),\dots,(\bar a_n,\bar
  b_n)\in E_L$ (in $\C_2$) such that for every $\sigma\in\aut(
\C_2)$,
$$\sigma(\varphi_i(\bar x,\bar a_i)\land\neg\varphi_i(\bar x,\bar
b_i))\in q_i\mbox{ for some } i=1,\dots,n.$$
This means that $$(\psi_1,\dots,\psi_n,p)\not\in c(q_1,\dots,q_n),$$
where $\psi_i(\bar x,\bar y)=\neg\varphi_i(\bar x,\bar
y_i)\lor\varphi_i(\bar x,\bar y_i'),\ \bar y=(\bar y_1,\bar
y_1',\dots,\bar y_n,\bar y_n')$ and $p(\bar y)=\tp(\bar a_1,\bar
b_1,\dots,\bar a_n,\bar b_n)$. By Lemma \ref{lemma: technical lemma 1}, there are types
$q_1',\dots,q_n'\in S_X(\C_1)$ with
$c(q_1,\dots,q_n)=c(q_1',\dots,q_n')$. The types $q_i'$, formulas
$\varphi_i$ and tuples $(\bar a_i,\bar b_i),i=1,\dots,n$, witness that
the right hand side of Corollary \ref{corollary: consequence of the
  characterization of boundedness} fails for $\C_1$. Hence the minimal
left ideals in $EL_1$ are of unbounded size.

($\leftarrow$) Suppose that the minimal left ideals in $EL_2$ are of bounded size.
To deduce that the same is true in $EL_1$, we have to check that the right hand side of Corollary \ref{corollary: consequence of the characterization of boundedness} holds for $\C_1$. So consider any $q_1,\dots,q_n \in S_X(\C_1)$, formulas $\varphi_1(\bar x, \bar y_1), \dots, \varphi_n(\bar x, \bar y_n)$, and tuples $(\bar a_1, \bar b_1) \in E_L, \dots, (\bar a_n, \bar b_n) \in E_L$ (where $\bar a_i$ corresponds to $\bar y_i$), and the goal is to find an appropriate $\sigma \in \aut(\C_1)$.  Choose a model $\C_2' \cong \C_2$ such that $\C_2' \prec \C_1$ and $\bar a_i, \bar b_i$ are contained in $\C_2'$ for all $i=1,\dots,n$. Let $q_1' = q_1|_{\C_2'}, \dots q_n'= q_n|_{\C_2'}$. By assumption, the minimal left ideals in $EL(S_X(\C_2'))$ are of bounded size, so Corollary \ref{corollary: consequence of the characterization of boundedness} yields $\sigma' \in \aut(\C_2')$ such that 
$$\sigma'(\varphi_1(\bar x, \bar a_1) \wedge \neg \varphi_n(\bar x, \bar b_1)) \notin q_1', \dots, \sigma'(\varphi_n(\bar x, \bar a_n) \wedge \neg \varphi_n(\bar x, \bar b_n)) \notin q_n'.$$
Then any extension $\sigma \in \aut(\C_1)$ of $\sigma'$ does the job.
\end{proof}

Proposition \ref{proposition: M_1 bounded iff M_2 bounded} gives us Theorem \ref{theorem: main theorem 2}(i).

\begin{cor}\label{corollary: main theorem 2 (i)}
Let $S$ be a product of an arbitrary (possibly unbounded) number of sorts, and $X$ be $\emptyset$-type-definable subset of $S$. Then the property that a minimal left ideal of the Ellis semigroup of the flow $S_X(\C)$ is of bounded size is absolute (i.e. does not depend on the choice of $\C$).
\end{cor}

By Proposition \ref{proposition: characterization of boundedness of minimal ideals}, Corollary \ref{corollary: main theorem 2 (i)}, Proposition \ref{proposition: invariants for c isomorphic to invariants for S}, we get

\begin{cor}\label{corollary: explicit bound on S_alpha and S_c} Let $\bar \alpha$ be a short tuple, and $\bar c$ be an enumeration of $\C$. The property that the minimal left ideals in $EL(S_{\bar c}(\C))$ [resp. in $EL(S_{\bar \alpha}(\C))$] are of bounded size is absolute. Moreover, in each of these two cases, if the minimal left ideals are of bounded size, then this size is bounded by $\beth_3(|T|)$. 
\end{cor} 

To show the second part of Theorem \ref{theorem: main theorem 2}, we need some preparatory observations, which explain better the whole picture. 
Assume for the rest of this section that $I_L$ is not empty. Then
$I_L$ is an $\aut(\C)$-flow, hence it has its own Ellis semigroup $EL(I_L)$ which is a subflow of the $\aut(\C)$-flow $I_L^{I_L}$.

\begin{prop}\label{proposition: F for minimal flows}
Let $\bar F \colon EL \to S_X(\C)^{I_L}$ be the restriction map. Then:\\
i) $\bar F$ is a homomorphism of $\aut(\C)$-flows.\\
ii) $\im (\bar F) = EL(I_L) \subseteq I_L^{I_L}$, so from now on we treat $\bar F$ as a function going to $EL(I_L)$. Then $\bar F$ is an epimorphism of $\aut(\C)$-flows and of semigroups.\\
iii) $\bar F[{\mathcal M}]$ is a minimal subflow (equivalently, minimal left ideal) of $EL(I_L)$. Let $F = \bar F|_{\mathcal M}\colon {\mathcal M} \to \bar F[{\mathcal M}]$.\\
iv) If ${\mathcal M}$ is of bounded size, then $F \colon {\mathcal M} \to \im (F)$ is an isomorphism of $\aut(\C)$-flows and of semigroups.
\end{prop}

\begin{proof}
Point (i) is obvious. Point (ii) follows from (i) and compactness of the spaces in question, namely: $\im (\bar F) = \bar F[\cl(\aut(\C) \id_{S_X(\C)})]= \cl(\aut(\C) \id_{I_L})= EL(I_L)$. Point (iii) follows from (ii).
Let us show (iv). Take an idempotent $u \in {\mathcal M}$. By Proposition \ref{proposition: characterization of boundedness of minimal ideals}, $\im (u) \subseteq I_L$. We need to show that $F$ is injective. This follows by the same simple argument as in the paragraph following Proposition \ref{proposition: topological absoluteness}:  if $F(f)=F(g)$ for some $f,g \in {\mathcal M}$ (which implies that $f=fu$ and $g=gu$), then $f|_{I_L} = g|_{I_L}$, so, since $\im (u) \subseteq  I_L$, we get $f=fu=gu=g$.
\end{proof}

Coming back to the situation with two monster models $\C_1 \succ \C_2$, note that each of the spaces $I_{L,\C_1}$, $I_{L,\C_1}^{S_X(\C_1)}$, $I_{L,\C_1}^{I_{L,\C_1}}$, and $EL(I_{L,\C_1})$ is naturally an $\aut(\C_2)$-flow with the action of $\aut(\C_2)$ defined by: $\sigma x: = \sigma' x$, where $\sigma' \in \aut(\C_1)$ is an arbitrary extension of $\sigma \in \aut(\C_2)$. Using the fact that for every $\sigma \in \aut(\C_1)$ there is $\tau \in \aut(\C_1)$ which preserves $\C_2$ setwise (i.e. which is an extension of an automorphism of $\C_2$) and such that $\sigma /\autf_L(\C_1)= \tau /\autf_L(\C_1)$, we get that in each of the above four spaces, the $\aut(\C_1)$-orbits coincide with the $\aut(\C_2)$-orbits, so the minimal $\aut(\C_1)$-subflows coincide with the minimal $\aut(\C_2)$-subflows. Hence, $EL(I_{L,\C_1}):=\cl(\aut(\C_1) \id_{I_{L,\C_1}}) = \cl(\aut(\C_2) \id_{I_{L,\C_1}})$. If the minimal left ideals in $EL_1$ are of bounded size, then (by Proposition \ref{proposition: characterization of boundedness of minimal ideals}) they are contained in $I_{L,\C_1}^{S_X(\C_1)}$, so they are naturally minimal $\aut(\C_2)$-flows, and for every minimal left ideal ${\mathcal M_1}$ of $EL_1$ the restriction isomorphism $F_1 \colon {\mathcal M}_1 \to \im (F_1) \subseteq EL(I_{L,\C_1})$ of $\aut(\C_1)$-flows (see Proposition \ref{proposition: F for minimal flows}(iv)) is also an isomorphism of $\aut(\C_2)$-flows.
Let now $\pi_{12}$ denote the homeomorphism from  $I_{L,\C_1}$ to $I_{L,\C_2}$. It  induces a homeomorphism $\pi_{12}'\colon I_{L,\C_1}^{ I_{L,\C_1}} \to I_{L,\C_2}^{ I_{L,\C_2}}$. 

\begin{rem}
i) $\pi_{12}$ and $\pi_{12}'$ are both isomorphisms of $\aut(\C_2)$-flows and $\pi_{12}'$ is an isomorphism of semigroups.\\
ii) $\pi_{12}'[EL(I_{L,\C_1})] = EL(I_{L,\C_2})$.\\
iii) $\pi_{12}'$ maps the collection of all minimal left ideals in $EL(I_{L,\C_1})$ bijectively onto the collection of all minimal left ideals in $EL(I_{L,\C_2})$.
\end{rem}

\begin{proof}
Point (i) follows from the definition of the action of $\aut(\C_2)$ given in the above discussion. Point (ii) follows from (i), namely: $\pi_{12}'[EL(I_{L,\C_1})] = 
\pi_{12}'[\cl(\aut(\C_1) \id_{I_{L,\C_1}})]  =  \pi_{12}'[\cl(\aut(\C_2) \id_{I_{L,\C_1}})] = \cl(\aut(\C_2) \id_{I_{L,\C_2}}) = EL(I_{L,\C_2})$. Point (iii) follows from (i) and (ii).
\end{proof}

From now on, let $\pi_{12}'$ be the restriction of the old $\pi_{12}'$ to $EL(I_{L,\C_1})$. So now  $\pi_{12}' \colon EL(I_{L,\C_1}) \to EL(I_{L,\C_2})$ is an isomorphism of $\aut(\C_2)$-flows and of semigroups. 

The next proposition gives us Theorem \ref{theorem: main theorem 2}(ii) with some additional information.

\begin{prop}\label{proposition: main theorem 2 (ii)}
Assume that a minimal left ideal of the Ellis semigroup of $S_X(\C)$ is of bounded size. Then:\\
i) For every minimal left ideal $\mathcal{M}_{1}$ in $EL_1$ there exists a minimal left ideal ${\mathcal M}_{2}$ in $EL_2$ which is isomorphic to ${\mathcal M}_{1}$ as a semigroup.\\
ii)  For every minimal left ideal $\mathcal{M}_{2}$ in $EL_2$ there exists a minimal left ideal ${\mathcal M}_{1}$ in $EL_1$ which is isomorphic to ${\mathcal M}_{2}$ as a semigroup.\\
iii) All minimal left ideals in $EL_1$ are naturally pairwise isomorphic minimal $\aut(\C_2)$-flows and are isomorphic as $\aut(\C_2)$-flows to the minimal $\aut(\C_2)$-subflows of $EL_2$.
\end{prop}
\begin{proof}
First of all, by assumption and Corollary \ref{corollary: main theorem 2 (i)}, the minimal left ideals in $EL_1$ and the minimal left ideals in  $EL_2$ are all of bounded size.\\[1mm]
(i)  Let ${\mathcal M}_1$ be a minimal left ideal of $EL_1$. 
By Proposition \ref{proposition: F for minimal flows}, the restriction maps $\bar F_1 \colon EL_1 \to EL(I_{L,\C_1})$ and $\bar F_2 \colon EL_2 \to EL(I_{L,\C_2})$ are semigroup epimorphisms. Moreover, $F_1: = \bar F_1|_{{\mathcal M}_1} \colon {\mathcal M}_1 \to \im (F_1)$ is a semigroup isomorphism and $\im (F_1)$ is a minimal left ideal of $EL(I_{L,\C_1})$. By the above discussion, $\pi_{12}' \colon EL(I_{L,\C_1}) \to EL(I_{L,\C_2})$ is an isomorphism, so $\pi_{12}'[\im (F_1)]$ is a minimal left ideal of $EL(I_{L,\C_2})$. Then $\bar F_2^{-1}[\pi_{12}'[\im (F_1)]]$ is a left ideal of $EL_2$, which contains some minimal left ideal ${\mathcal M}_2$. Then, by the minimality of $\pi_{12}'[\im (F_1)]$, we get that $\bar F_2[{\mathcal M}_2] = \pi_{12}'[\im (F_1)]$, and, by Proposition \ref{proposition: F for minimal flows}, $F_2:= \bar F_2 |_{{\mathcal M}_2} \colon {\mathcal M}_2 \to \pi_{12}'[\im (F_1)]$ is a semigroup  isomorphism. Therefore, $F_2^{-1} \circ \pi_{12}'|_{\im (F_1)} \circ F_1 \colon {\mathcal M}_1 \to {\mathcal M}_2$ is a semigroup isomorphism, and we are done.\\[1mm]
(ii) This can be shown analogously to (i), but ``going in the opposite direction''. \\[1mm]
(iii) All minimal left ideals of $EL_1$ are pairwise isomorphic as $\aut(\C_1)$-flows. So, by the description of the natural $\aut(\C_2)$-flow structure on these  ideals, we get that they are also isomorphic as $\aut(\C_2)$-flows and that they are minimal $\aut(\C_2)$-flows. Of course, all minimal left ideals (equivalently, minimal subflows) of $EL_2$ are also isomorphic as $\aut(\C_2)$-flows. Then, we apply the proof of (i), noticing that the discussion preceding Proposition \ref{proposition: main theorem 2 (ii)} implies that $F_1$, $F_2$, and $\pi_{12}'|_{\im (F_1)}$ are all $\aut(\C_2)$-flow isomorphisms.
\end{proof}

Corollary \ref{corollary: main theorem 2 (i)} and Proposition \ref{proposition: main theorem 2 (ii)} imply Theorem \ref{theorem: main theorem 2}. By virtue of  Proposition \ref{proposition: invariants for c isomorphic to invariants for S}, we get the obvious counterpart of Theorem \ref{theorem: main theorem 2} for both $S_{\bar \alpha}(\C)$ and $S_{\bar c}(\C)$ in place of $S_X(\C)$, which answers Question \ref{question: main question 2}. An explicit bound on the size of the minimal left ideals is provided in Proposition \ref{proposition: characterization of boundedness of minimal ideals} and Corollary \ref{corollary: explicit bound on S_alpha and S_c}.

\section{The stable case}\label{section 5}

In the context of topological dynamics for definable groups, when the theory in question is stable, all notions (such as the Ellis semigroup, minimal ideals, the Ellis group) are easy to describe and coincide with well-known and important model-theoretic notions. This was a starting point to generalize some phenomena to wider classes of theories. 

In this section, we are working in the Ellis semigroup of the
$\aut(\C)$-flow $S_{\bar c}(\C)$, where $\C$ is a monster model  of a
{\em stable} theory and $\bar c$ is an enumeration of $\C$, and we
give very easy descriptions of the minimal left ideals and of the
Ellis group. The reason why we work with $S_{\bar c}(\C)$ is that by
the appendix of \cite{KrPiRz} we know that stability is equivalent to the existence of a left-continuous semigroup operation $*$ on $S_{\bar c}(\C)$ extending the action of $\aut(\C)$ (i.e. $\sigma(\tp(\bar c/\C)) * q = \sigma(q)$ for $\sigma \in \aut(\C)$), which implies  that $S_{\bar c}(\C) \cong EL(S_{\bar c}(\C))$ via the map $p \mapsto l_p$, where $l_p \colon S_{\bar c}(\C) \to S_{\bar c}(\C)$ is given by $l_p(q)=p*q$. For a short tuple $\bar \alpha$ in place of $\bar c$ this is not true (e.g. in the theory of an infinite set with equality, taking $\alpha$ to be a single element, we have that $|S_\alpha(\C)|=|S_1(\C)|=|\C|$ while $EL(S_1(\C))$ has size at least $|\aut(\C)| = 2^{|\C|}$). However, at the end of this section, as corollaries from our observations made for $\bar c$, we describe what happens in $EL(S_{\bar \alpha}(\C))$.

From now on, in this section, $T$ is a stable theory. 
For any (short or long) tuple $\bar a$ of elements of $\C$ and a  (small or large) set of parameters $B$ let
$$NF_{\bar a}(B) = \{ p \in S_{\bar a}(B) : p\;\, \mbox{does not fork over}\;\, \emptyset\}.$$ 
$NF_{\bar a}(\C)$ is clearly a subflow of $S_{\bar a}(\C)$. Note that
$NF_{\bar a}(\acl^{eq}(\emptyset))=S_{\bar a}(\acl^{eq}(\emptyset))$
is also an $\aut(\C)$-flow which is transitive.  Let $\bar \epsilon$
denote an enumeration of $\acl^{eq}(\emptyset)$. Then $NF_{\bar
  \epsilon}(\acl^{eq}(\emptyset)) = S_{\bar
  \epsilon}(\acl^{eq}(\emptyset))$ is an $\aut(\C)$-flow, and note
that $S_{\bar \epsilon}(\acl^{eq}(\emptyset))$ can be naturally
identified with the profinite group $\aut(\acl^{eq}(\emptyset))$ of
all elementary permutations of $\acl^{eq}(\emptyset)$. Notice that
$NF_{\bar c}(\C)$ is closed under $*$.

The key basic consequences of stability that we will be using are:
\begin{itemize}
\item each type over an algebraically closed set (e.g. a model) is stationary, i.e. it has a unique non-forking extension to any given superset,
\item $NF_{\bar a}(\C)$ is a transitive $\aut(\C)$-flow,
\item the $\aut(\C)$-orbit of each type in $S_{\bar a}(\C) \setminus NF_{\bar a}(\C)$ is of unbounded size.
\end{itemize}

Recall that $\autf_{Sh}(\C): = \aut(\C/\acl^{eq}(\emptyset))$, $E_{Sh}$ is the orbit equivalence relation of $\autf_{Sh}(\C)$ on any given product of sorts, and $\gal_{Sh}(T):= \aut(\C)/\autf_{Sh}(\C) \cong \aut(\acl^{eq}(\emptyset))$. In the stable context, $\autf_{Sh}(\C) =\autf_{KP}(\C)=\autf_{L}(\C)$, so the corresponding orbit equivalence relations coincide, and the corresponding Galois groups coincide, too.

\begin{rem}\label{remark: restriction is a homeomorphism}
Let $\bar a$ be an enumeration of a (small or large) algebraically closed subset of $\C$ (e.g. $\bar a =\bar c$ or $\bar a = \bar \epsilon$).\\
i)  The restriction map $r \colon NF_{\bar c}(\C) \to NF_{\bar a}(\C)$ is a flow isomorphism. In particular, $r$ induces a unique left-continuous semigroup operation (also denoted by $*$) on $NF_{\bar a}(\C)$ 
such that whenever $(\sigma_i)$ is a net in $\aut(\C)$ satisfying $\lim \sigma_i(\tp(\bar a/\C)) = p \in NF_{\bar a}(\C)$ and $q \in NF_{\bar a}(\C)$, then $p * q= \lim \sigma_i(q)$.\\
ii) The restriction map $r_\emptyset \colon NF_{\bar a}(\C) \to S_{\bar a}(\acl^{eq}(\emptyset))$ is a flow isomorphism, which induces a unique left-continuous semigroup operation on $S_{\bar a}(\acl^{eq}(\emptyset))$ (still denoted by $*$) which coincides with the action of $\aut(\C)$ (i.e. $\sigma(\tp(\bar a/\acl^{eq}(\emptyset))) * q = \sigma(q)$ for $\sigma \in \aut(\C)$ and $q \in S_{\bar a}(\acl^{eq}(\emptyset))$).\\
iii) The restriction map $R \colon NF_{\bar a}(\C) \to S_{\bar \epsilon}(\acl^{eq}(\emptyset))=\aut(\acl^{eq}(\emptyset))$ is an isomorphism of $\aut(\C)$-flows and of semigroups (where we take the obvious group structure on $\aut(\acl^{eq}(\emptyset)$). Thus, $NF_{\bar a}(\C)$ is a group. 
\end{rem}

\begin{proof}
i) In the first part, only injectivity of $r$ requires an
explanation. So consider $p,q \in  NF_{\bar c}(\C)$ such that
$r(p)=r(q)$. Let $\C'\succ \C$ be a bigger monster model in which we
will take all realizations. Choose $\bar a'$ realizing $r(p)=r(q)$ and
extend it to $\bar c'\models p$ and $\bar c''\models q$. Since $\bar
c'\equiv\bar c''$, we get $\bar c'\equiv_{\bar a'}\bar c''$. Since
$\acl^{eq}(\emptyset)\subseteq\bar a'$ we get that $\bar
c'\equiv_{\acl^{eq}(\emptyset)}\bar c''$. Both $p$ and $q$ do not fork
over $\emptyset$, hence $p=q$.

It is clear that $r$ and the original semigroup operation on $NF_{\bar
  c}(\C)$ induce a left-continuous semigroup operation on $NF_{\bar
  a}(\C)$ such that for any net $(\tau_j)$ in $\aut(\C)$ satisfying
$\lim \tau_j(\tp(\bar c/\C)) = p' \in NF_{\bar c}(\C)$ and any $q \in
NF_{\bar a}(\C)$ we have $r(p') * q= \lim \tau_j(q)$. Since each such a
net $(\tau_j)$ satisfies $\lim \tau_j(\tp(\bar a/\C)) = r(p')$, it is
enough to prove that for any net $(\sigma_i)$ in $\aut(\C)$ such that
$\lim \sigma_i(\tp(\bar a/\C)) = p \in NF_{\bar a}(\C)$ and any $q \in
NF_{\bar a}(\C)$, the limit $\lim \sigma_i(q)$ exists. This in turn
follows since $q$ is definable over
$\acl^{eq}(\emptyset)\subseteq\bar a$: if $\bar
a_0\subseteq\acl^{eq}(\emptyset)$ and $\psi(\bar y,\bar a_0)$ defines
$q|_{\varphi}$, then $(\lim\sigma_i(q))|_{\varphi}$ is defined by
  $\psi(\bar y,\bar a_1)$, where $\bar a_1=\lim\sigma_i(\bar a_0)$.

ii) The fact that $r_\emptyset$ is a flow isomorphism is immediate from stationarity of the types over algebraically closed sets.

For the other part, let $\sigma\in\aut(\C)$ and $q \in S_{\bar a}(\acl^{eq}(\emptyset))$. 
Take $p_0\in NF_{\bar  a}(\C)$ extending $\tp(\bar a/\acl^{eq}(\emptyset))$. Since the $\aut(\C)$-orbit of $\tp(\bar a/\C)$ is dense in $S_{\bar a}(\C)$, we can find a net $(\sigma_i)$ in $\aut(\C)$ such that $\lim\sigma_i(\tp(\bar a/\C))=\sigma(p_0)$. Then $\lim\sigma_i|_{\acl^{eq}(\emptyset)}=\sigma|_{\acl^{eq}(\emptyset)}$. Also, $\sigma(\tp(\bar a/\acl^{eq}(\emptyset))=r_{\emptyset}(\sigma(p_0))$.
%
Hence, by (i),
$$\sigma(\tp(\bar a/\acl^{eq}(\emptyset))*q=r_{\emptyset}(\sigma(p_0)*r_\emptyset^{-1}(q))=r_{\emptyset}(\lim\sigma_i(r_\emptyset^{-1}(q)))= \lim \sigma_i (q)=\sigma(q).$$

%
iii) The fact that $R$ is a flow isomorphism follows from (i) and (ii).  To see that $R$ is a semigroup isomorphism, one should check that the natural identification of $S_{\bar \epsilon}(\acl^{eq}(\emptyset))$ with $\aut(\acl^{eq}(\emptyset))$ is a homomorphism, where $S_{\bar \epsilon}(\acl^{eq}(\emptyset))$ is equipped with the semigroup operation from (ii). But this is obvious by (ii).
\end{proof}

\begin{prop}\label{proposition: unique minimal left ideal}
$NF_{\bar c}(\C)$ is the unique minimal left ideal in $S_{\bar c}(\C)$, and $|NF_{\bar c}(\C)| \leq 2^{|T|}$.
\end{prop}

\begin{proof}
Minimality of $NF_{\bar c}(\C)$ is clear, as $NF_{\bar c}(\C)$ is a transitive $\aut(\C)$-flow.
By Remark \ref{remark: restriction is a homeomorphism},  $|NF_{\bar c}(\C)|= |S_{\bar \epsilon}(\acl^{eq}(\emptyset))| \leq 2^{|T|}$ is bounded. Hence, all minimal left ideals are of bounded size.
On the other hand, stability implies that the $\aut(\C)$-orbit of any $p \in S_{\bar c}(\C) \setminus NF_{\bar c}(\C)$ is unbounded. This shows uniqueness of the minimal left ideal.
\end{proof}

So let ${\mathcal M}:=NF_{\bar c}(\C)$ be the unique minimal left ideal of $S_{\bar c}(\C)$. Let $u \in {\mathcal M}$ be an idempotent. By Remark \ref{remark: restriction is a homeomorphism}(iii), ${\mathcal M}$ is a group, so ${\mathcal M} =u{\mathcal M}$ and $u$ is the unique idempotent in ${\mathcal M}$. Moreover, the restriction map $R \colon {\mathcal M} \to \aut(\acl^{eq}(\emptyset))$ is a group isomorphism, which explicitly shows absoluteness of the Ellis group of the flow $S_{\bar c}(\C)$.

Under the natural identification of $EL(S_{\bar c}(\C))$ with $S_{\bar c}(\C)$ described in the second paragraph of this section, the semigroup epimorphism  $\hat{f}\colon EL(S_{\bar c}(\C)) \to \gal_L(T)$ from \cite{KrPiRz} (recalled in Section \ref{section 1}) is given by $\hat{f}(p)= \sigma/\autf_{Sh}(\C')$ for any $\sigma \in \aut(\C')$ such that $\sigma(\bar c) \models p \in S_{\bar c}(\C)$, where $\C'\succ \C$ is a bigger monster model. 
As was recalled in Section \ref{section 1}, $f: = \hat{f}|_{u\mathcal M}$ is a group epimorphism onto $\gal_{Sh}(T)$. Using the natural identification of $\gal_{Sh}(T)$ with $\aut(\acl^{eq}(\emptyset))$, one gets that $R = \hat{f}|_{\mathcal M}$, so we have

\begin{cor}
$f\colon u{\mathcal M} \to \gal_{Sh}(T)$ is a group isomorphism.
\end{cor}

 From Remark \ref{remark: restriction is a homeomorphism}(iii), we get

\begin{cor}
The unique idempotent $u \in {\mathcal M}$ is the unique global non-forking extension of $\tp(\bar c/\acl^{eq}(\emptyset))$.
\end{cor}

Now we give a description of the group operation $*$ on ${\mathcal M}=NF_{\bar c}(\C)$.  

\begin{prop}
Let $p,q,r \in NF_{\bar c}(\C)$. Take any $\bar c' \models q$. Then $p * q =r$ if and only if there exists $\sigma \in \aut(\C')$ such that $\sigma (\bar c) \models p$ and $\sigma(\bar c') \models r$.
\end{prop}

\begin{proof}
$(\rightarrow)$ There is a net $(\sigma_i)$ in $\aut(\C)$ such that $\lim \sigma_i = p$, which formally means that $\lim \sigma_i(\tp(\bar c/\C)) =p$. By the left continuity of $*$, we get $\lim \sigma_i(q) =r$. Thus, an easy compactness argument yields the desired $\sigma$. (Note that this implication does not use the assumption that $p,q,r \in NF_{\bar c}(\C)$.)\\[1mm]
$(\leftarrow)$ This follows from Remark \ref{remark: restriction is a homeomorphism}(iii), but we give a direct computation. Take $\sigma$ satisfying the right hand side. By $(\rightarrow)$, we can find $\sigma' \in \aut(\C')$ such that $\sigma' (\bar c) \models p$ and $\sigma'(\bar c') \models p * q$. Since $\sigma(\bar c) \models p$, we get that $\sigma|_{\acl^{eq}(\emptyset)} = \sigma'|_{\acl^{eq}(\emptyset)}$. Hence, $r|_{\acl^{eq}(\emptyset)}= \tp(\sigma(\bar c')/\acl^{eq}(\emptyset)) =  \sigma(\tp(\bar c'/\acl^{eq}(\emptyset))) =  \sigma'(\tp(\bar c'/\acl^{eq}(\emptyset))) = \tp(\sigma'(\bar c')/\acl^{eq}(\emptyset)) = p * q|_{\acl^{eq}(\emptyset)}$. This implies that $r=p*q$, because $r,p*q \in NF_{\bar c}(\C)$. 
\end{proof}

As was recalled in the second paragraph of this section, there is a
semigroup isomorphism $l \colon S_{\bar c}(\C) \to EL(S_{\bar c}(\C))$
given by $l(p) = l_p$. By Proposition \ref{proposition: unique minimal left ideal}, $l[\M]$ is the unique minimal left ideal of $EL(S_{\bar c}(\C))$ and it is of bounded size.  In stable theories, Lascar invariant global
types are the types that do not fork over $\emptyset$. Hence, by
Proposition \ref{proposition: characterization of boundedness of
  minimal ideals}, for every $\eta \in  l[\M]$, $\im(\eta) \subseteq NF_{\bar c}(\C)$, but since $\M=NF_{\bar c}(\C)$ is a group and $l_p(q)=p*q$, we easily conclude that  $\im(\eta) = NF_{\bar c}(\C)$.

Now, consider a short tuple $\bar \alpha$. By the above conclusions and Remark  \ref{remark: very easy 2}, we get

\begin{cor}
There is a unique minimal left ideal ${\mathcal M}_{\bar \alpha}$ in $EL(S_{\bar \alpha}(\C))$, it coincides with the Ellis group, and for every $\eta \in {\mathcal M}_{\bar \alpha}$,
$\im(\eta)= NF_{\bar \alpha}(\C)$. The size of ${\mathcal M}_{\bar \alpha}$ is bounded by $2^{|T|}$.
\end{cor}

\begin{cor}
If $\bar \alpha$ enumerates an algebraically closed set, then the unique minimal left ideal ${\mathcal M}_{\bar \alpha}$ in $EL(S_{\bar \alpha}(\C))$ is isomorphic as a group with $NF_{\bar \alpha}(\C)$ (with the semigroup structure provided by Remark \ref{remark: restriction is a homeomorphism}(i)) which is further isomorphic with $\aut(\acl^{eq}(\emptyset))$.
\end{cor}

\begin{proof}



By Remark \ref{remark: restriction is a homeomorphism} (i) and (iii), $NF_{\bar \alpha}(\C)$ is  a group isomorphic with $\aut(\acl^{eq}(\emptyset))$ and the assignment $p \mapsto l_p$ (where $l_p(q)=p*q$) yields an isomorphism from $NF_{\bar \alpha}(\C)$ to $EL(NF_{\bar \alpha}(\C))$.

Since $NF_{\bar \alpha}(\C)$ coincides with the set of Lascar invariant types in $S_{\bar \alpha}(\C)$, by the previous corollary together with Proposition \ref{proposition: F for minimal flows}, 
we get that the restriction of the domains from $S_{\bar\alpha}(\C)$ to $NF_{\bar\alpha}(\C)$ yields a monomorphism from ${\mathcal M}_{\bar \alpha} \subseteq NF_{\bar \alpha}(\C)^{S_{\bar \alpha}(\C)}$ to $EL(NF_{\bar \alpha}(\C)) \subseteq NF_{\bar \alpha}(\C)^{NF_{\bar \alpha}(\C)}$ whose image is a minimal left ideal in $EL(NF_{\bar \alpha}(\C))$ and so coincides with $EL(NF_{\bar \alpha}(\C))$ (as  $EL(NF_{\bar \alpha}(\C))$ is a group by the first paragraph of this proof).
%
\end{proof}

\section{The NIP case}\label{section 6}

Throughout this section, we assume that $\C$ is a monster model of a theory $T$ with NIP. Let $\kappa$ be the degree of saturation of $\C$. As usual, $\bar c$ is an enumeration of $\C$, and ${\bar \alpha}$ is a short tuple in $\C$. Let $S$ be an arbitrary product of sorts (with repetitions allowed), and $X$ be a $\emptyset$-type-definable subset of $S$. Let ${\mathcal M}$ be a minimal left ideal of $EL(S_X(\C))$.

In the first subsection, after giving some characterizations of when the minimal left ideals in $EL(S_X(\C))$ are of bounded size, we prove the main result of the first subsection which yields several characterizations (in various terms) of when the minimal left ideals of $EL(S_{\bar c}(\C))$ are of bounded size.  We also make some observations and state questions concerning boundedness of the minimal left ideals of $EL(S_{\bar \alpha}(\C))$. 

In the second subsection, we give a better bound (than the one from Corollaries \ref{corollary: first part of main theorem 1} and \ref{corollary: final corollary concerning the size of the Ellis group}) on the size of the Ellis group of the flow $S_X(\C)$, and, as a consequence -- of the flows $S_{\bar c}(\C)$ and $S_{\bar \alpha}(\C)$. The main point is that instead of the set $R$ obtained in Section \ref{section 3} via contents, under the NIP assumption we can just use types invariant over a model.

In the last subsection, we adapt the proof of \cite[Theorem 5.7]{ChSi} to show Theorem \ref{theorem theorem from ChSi}.  We also find a counterpart of the epimorphism $f$ described in Section \ref{section 1} which goes from the Ellis group of $S_{\bar \alpha}(\C)$ to a certain new Galois group introduced in \cite{DoKiLe}, and we 
give an example showing that the obvious counterpart of Theorem \ref{theorem theorem from ChSi} does not hold for this new epimorphism.

\subsection{Characterizations of boundedness of minimal left ideals}\label{subsection 6.1}

Recall that all the time we assume NIP.

\begin{fct}\label{fact: forking and boundedness in NIP} 
Let $p \in S_X(\C)$. 
Then, the following conditions are equivalent.\\
i) The $\aut(\C)$-orbit of $p$ is of bounded size.\\
ii) $p$ does not fork over $\emptyset$.\\
iii) $p$ is Kim-Pillay invariant (i.e. invariant under $\autf_{KP}(\C)$).\\
iv) $p$ is Lascar invariant.
\end{fct}

\begin{proof}
(i) $\rightarrow$ (ii). Suppose $p$ forks over $\emptyset$. Then there is $\varphi(\bar x, \bar a) \in p$ which divides over $\emptyset$. So there is a tuple $(a_i)_{i<\kappa}$ in $\C$ which is indiscernible over $\emptyset$ and such that the sequence $\langle \varphi(\bar x, \bar a_i): i <\kappa\rangle $ is $k$-inconsistent for some $k<\omega$. This implies that the orbit $\aut(\C)p$ is of size at least $\kappa$ which is unbounded.\\[1mm]
The implication (ii) $\rightarrow$ (iii) is Proposition 2.11 of \cite{HrPi}. The implications (iii) $\rightarrow$ (iv) and (iv) $\rightarrow$ (i) are obvious.
\end{proof}

An immediate corollary of this fact and Proposition \ref{proposition: characterization of boundedness of minimal ideals} is

\begin{cor}\label{corollary: characterization of boundedness of minimal ideals in terms of non-forking}
The following conditions are equivalent.\\
i) ${\mathcal M}$ is of bounded size.\\
ii) For every $\eta \in {\mathcal M}$ the image $\im(\eta)$ consists of types which do not fork over $\emptyset$.\\
iii) For some $\eta \in EL$ the image $\im(\eta)$ consists of types which do not fork over $\emptyset$.
\end{cor}

\begin{prop}\label{proposition: easy characterization with many items}
The following conditions are equivalent.\\
i) ${\mathcal M}$ is of bounded size.\\
ii) For every natural number $n$, for every types $q_1,\dots,q_n \in S_X(\C)$ and for every formulas $\varphi_1(\bar x,\bar a_1),\dots, \varphi_n(\bar x, \bar a_n)$ (where $\bar x$ corresponds to $S$) which fork over $\emptyset$, there exists $\sigma \in \aut(\C)$ such that $\sigma(\varphi_i(\bar x,\bar a_i)) \notin q_i$ for all $i=1,\dots, n$.\\
iii) The same condition as (ii) but with ``forking'' replaced by ``dividing''.\\
iv) For every natural number $n$, for every type $q \in S_X(\C)$ and for every formulas $\varphi_1(\bar x,\bar a_1),\dots, \varphi_n(\bar x, \bar a_n)$ (where $\bar x$ corresponds to $S$) which fork over $\emptyset$, there exists $\sigma \in \aut(\C)$ such that $\sigma(\varphi_i(\bar x,\bar a_i)) \notin q$ for all $i=1,\dots, n$.\\
v) The same condition as (iv) but with ``forking'' replaced by ``dividing''.\\
vi) For every $q \in S_X(\C)$ the closure $\cl(\aut(\C)q)$ of the $\aut(\C)$-orbit of $q$ contains a type which does not fork over $\emptyset$. 
\end{prop}

\begin{proof}
(i) $\rightarrow$ (ii).  Consider any types $q_1,\dots,q_n \in S_X(\C)$ and formulas $\varphi_1(\bar x,\bar a_1),\dots, \varphi_n(\bar x, \bar a_n)$ which fork over $\emptyset$. Take $\eta \in {\mathcal M}$. By Corollary \ref{corollary: characterization of boundedness of minimal ideals in terms of non-forking}, $\eta(q_i)$ does not fork over $\emptyset$,  and so $\neg \varphi_i(\bar x,\bar a_i) \in \eta(q_i)$, for all $i=1,\dots,n$. Since $\eta$ is approximated by automorphisms, there is $\sigma' \in \aut(\C)$ such that $\neg \varphi_i(\bar x,\bar a_i) \in \sigma'(q_i)$ for $i=1,\dots,n$. Thus, $\sigma: = \sigma'^{-1}$ does the job.\\[1mm]
The implications (ii) $\rightarrow$ (iii), (ii) $\rightarrow$ (iv), (iii) $\rightarrow$ (v), (iv) $\rightarrow$ (v) are obvious.\\[1mm]
(v) $\rightarrow$ (vi). This follows immediately by taking the limit of a convergent subnet of the net of inverses of the automorphisms which we get for all possible finite sequences of formulas as in (v). (Here, the partial directed order on finite sequences of dividing formulas is given by being a subsequence.)\\[1mm]
(vi) $\rightarrow$ (i). Consider any $\eta \in {\mathcal M}$. By Corollary \ref{corollary: characterization of boundedness of minimal ideals in terms of non-forking}, it is enough to show that $\eta(q)$ does not fork over $\emptyset$ for all $q \in S_X(\C)$. So take any $q \in S_X(\C)$.

Let $\pi_q \colon {\mathcal M} \to S_X(\C)$ be given by
$\pi_q(\xi)=\xi(q)$. This is a homomorphism of $\aut(\C)$-flows, so
$\im (\pi_q)$ is a minimal subflow of $S_X(\C)$, hence
$\im(\pi_q)=\cl(\aut(\C)p)$ for every $p\in\im(\pi_q)$. By (vi), there
is $p\in\im(\pi_q)$ that does not fork over $\emptyset$. By invariance
and closedness of the collection of non-forking types, every type in
$\im(\pi_q)$ does not fork over $\emptyset$, in particular
$\pi_q(\eta)=\eta(q)$ does not fork over $\emptyset$. 
%
\end{proof}

We will say that a formula $\varphi(\bar x, \bar a)$ is {\em weakly invariant in $X$} if the collection of formulas $\{\sigma(\varphi(\bar x,\bar a)): \sigma \in \aut(\C)\} \cup X$ is consistent (where $X$ is treated as a partial type over $\emptyset$). Equivalently, this means that the collection of formulas   $\{\sigma(\varphi(\bar x,\bar a)): \sigma \in \aut(\C)\}$ extends to a type in $S_X(\C)$. 

\begin{cor}\label{corollary: bounded = weakly invariant do not fork}
${\mathcal M}$ is of bounded size if and only if each formula which is weakly invariant in $X$ does not fork over $\emptyset$.
\end{cor}

\begin{proof}
($\rightarrow$) This implication does not require NIP, although the
  NIP assumption greatly simplifies the proof. Suppose for a
  contradiction that $\pi(\bar x):=\{\sigma(\varphi(\bar x,\bar a)):
  \sigma \in \aut(\C)\}$ extends to a type $p \in S_X(\C)$, but
  $\varphi(\bar x, \bar a)$ forks over $\emptyset$. Note that for
  every $\sigma \in \aut(\C)$, $\pi(\bar x) \subseteq \sigma (p)$ and
  $\pi(\bar x)$ forks over $\emptyset$. Using NIP, by Proposition
  \ref{proposition: easy characterization with many items}, we get
  $\sigma\in\aut(\C)$ with $\sigma(\varphi(\bar x,\bar a))\not\in p$,
  a contradiction. Without NIP we proceed as follows.

  Let $[\pi(\bar x)]=\{ q \in S_X(\C): \pi(\bar x) \subseteq q\}$. 
Let $\pi_p \colon EL(S_X(\C)) \to S_X(\C)$ be given by
$\pi_p(\eta)=\eta(p)$. This is a homomorphism of
$\aut(\C)$-flows. Therefore, $\im (\pi_p) = \cl (\aut(\C)p) \subseteq
[\pi(\bar x)]$. So for every $\eta \in EL(S_X(\C))$, $\eta(p)$ forks
over $\emptyset$. By (1)$\rightarrow$(2) in Fact \ref{fact: forking
  and boundedness in NIP} (which does not use NIP), the orbit
$\aut(\C)\eta(p)$ is of unbounded size which implies that every
$\aut(\C)$-orbit in $EL(S_X(\C))$ is of unbounded size, and so is
$\M$, a contradiction.

($\leftarrow$)  We check that item (iv) from Proposition \ref{proposition: easy characterization with many items} holds. Since in (iv) we are talking about forking (and not about dividing), it is enough to show that for any $q \in S_X(\C)$ and a formula $\varphi(\bar x, \bar a)$ which forks over $\emptyset$ there is $\sigma \in \aut(\C)$ such that $\sigma(\varphi(\bar x, \bar a)) \notin q$. But this is clear, because by assumption, the fact that $\varphi(\bar x, \bar a)$ forks over $\emptyset$ implies that $\{\sigma(\varphi(\bar x,\bar a)): \sigma \in \aut(\C)\} \cup X$ is inconsistent.
\end{proof}

Note that each formula $\varphi(\bar x,\bar a)$ weakly invariant in
$X$ does not divide over $\emptyset$, and even $\varphi(\bar x,\bar a)
\cup X$ does not divide over $\emptyset$. Therefore, by the last
corollary, if for each formula $\varphi(\bar x,\bar a)$ such that
$\varphi(\bar x,\bar a) \cup X$ does not divide over $\emptyset$ we
have that $\varphi(\bar x,\bar a) \cup X$ does not fork over
$\emptyset$ (in such a case, we will say that {\em forking equals
  dividing on $X$}), then ${\mathcal M}$ is of bounded size. Does the
converse hold? 

\begin{ques}\label{question: bounded = forking equals dividing}
Is it true that ${\mathcal M}$ is of bounded size if and only if forking equals dividing on $X$?
\end{ques}

Now, we recall a few notions. A subset $D$ of a point-transitive $G$-flow $Y$ (i.e. $Y=\cl(Gy)$ for some $y \in Y$) is said to be: 
\begin{itemize}
\item {\em generic} (or {\em syndetic}) if finitely many translates of $D$ by elements of $G$ cover $Y$,
\item {\em weakly generic} if there is a non-generic $F \subseteq Y$ such that $D \cup F$ is generic.
\end{itemize}
The first notion is classical in topological dynamics and in model
theory. The second one was introduced by the second author as a
substitute for the notion of a generic set in the unstable
context. Recall that $p \in Y$ is called {\em almost periodic} if it
belongs to a minimal subflow of $Y$; it is {\em [weakly] generic} if
every open subset of $Y$ containing $p$ is [weakly] generic. The
second author  proved that the set of all weakly generic points is the
closure of the set of all almost periodic points, and he suggested
that in the context of a group $G$ definable in a model $M$ acting on
$S_G(M)$, the notion of a weak generic type may be the right
substitute of the notion of a generic type in an unstable context \cite{Ne1}. 

Let  us return to our context. From now on, we will consider the case where $X=\tp(\bar c/\emptyset)$ or $X=\tp(\bar \alpha/\emptyset)$. Then $S_X(\C)$ is a point-transitive $\aut(\C)$-flow (it is an $\aut(\C)$-ambit with the distinguished point $\tp(\bar c/\emptyset)$ or $\tp(\bar \alpha/\emptyset)$). We adapt the above general terminology to our context and say that a formula $\varphi(\bar x, \bar a)$ is {\em [weakly] generic in $X$} if  $[\varphi(\bar x,\bar a)] \cap S_X(\C)$ is a [weakly] generic subset of the $\aut(\C)$-flow $S_X(\C)$. Note that a formula $\varphi(\bar x, \bar a)$ is weakly invariant in $X$ (in our earlier terminology) if and only if $\neg \varphi(\bar x, \bar a)$ is not generic in $X$. A type $p \in S_X(\C)$ is {\em [weakly] generic} if it is so as an element of the $\aut(\C)$-flow $S_X(\C)$, which is equivalent to saying that every formula in $p$ is [weakly] generic in $X$. These notions were considered in this context already in \cite{NePe} under the name of ``[weakly] c-free'' instead of ``[weakly] generic''. Similarly, $p$ is said to be {\em almost periodic} if it is so as an element of the flow. 

Recall that a set $A$ is called an {\em extension base} if every type $p \in S(A)$ does not fork over $A$. This is equivalent to saying that each $p \in S(A)$ has a non-forking extension to any given superset of $A$. 
The following facts come from \cite{ChKa}. 

\begin{fct}\label{fact: forking equals dividing from CaKa}
Recall that we assume NIP.\\
i) Each model is an extension base (also without NIP).\\
ii) A set $A$ is an extension base if and only if forking equals dividing over $A$ (i.e. a formula $\varphi(\bar y, \bar a)$ divides over $A$ iff it
forks over A).
\end{fct}

Now, we will prove our main characterization result in the case of $X=\tp(\bar c/\emptyset)$. In particular, it contains Proposition \ref{proposition: characterization of boundedness in NIP theories}.

\begin{thm}\label{theorem: main characterization theorem under NIP}
Let $X = \tp(\bar c/\emptyset)$ (so ${\mathcal M}$ is a minimal left ideal of $EL(S_{\bar c}(\C)$). Then the following conditions are equivalent.\\
i) $\emptyset$ is an extension base.\\
ii) Forking equals dividing on $X$.\\
iii) The weakly generic formulas in $X$ do not fork over $\emptyset$.\\
iv) The almost periodic types in $S_{\bar c}(\C)$ have orbits of bounded size.\\
v) ${\mathcal M}$ is of bounded size.\\
vi) There is a type in $S_{\bar c}(\C)$ whose $\aut(\C)$-orbit is of bounded size; equivalently, there is a type in $S_{\bar c}(\C)$ which does not fork over $\emptyset$; equivalently, $\tp(\bar c/\emptyset)$ does not fork over $\emptyset$.
\end{thm}

\begin{proof}
(i) $\rightarrow$ (ii). By Fact \ref{fact: forking equals dividing from CaKa}(ii), point (i) implies that forking equals dividing over $\emptyset$, let alone forking equals dividing on $X$.\\[1mm]
(ii) $\rightarrow$ (iii). 
This is essentially \cite[Lemma 5.9]{NePe}.
Consider $\varphi(\bar x, \bar a)$ which is weakly generic in $X$. Take a formula $\psi(\bar x, \bar b)$ which is non-generic in $X$ and such that $\varphi(\bar x, \bar a) \lor \psi(\bar x, \bar b)$ is generic in $X$. Then there are $\sigma_1, \dots, \sigma_n \in \aut(\C)$ such that 
for $\varphi'(\bar x, \bar a'):=\sigma_1(\varphi(\bar x, \bar a)) \lor \dots \lor \sigma_n(\varphi(\bar x, \bar a))$ and $\psi'(\bar x, \bar b'):=\sigma_1(\psi(\bar x, \bar b)) \lor \dots \lor \sigma_n(\psi(\bar x, \bar b))$ one has 
$$[\varphi'(\bar x,\bar a') \lor \psi'(\bar x, \bar b')] \cap S_X(\C) =S_X(\C).$$

Suppose for a contradiction that $\varphi(\bar x, \bar a)$ forks over $\emptyset$. Then $\varphi'(\bar x, \bar a')$ also forks over $\emptyset$. 
By (ii), $\varphi'(\bar x, \bar a') \cup X$ divides over $\emptyset$, so there are $\tau_1, \dots, \tau_m \in \aut(\C)$ such that 
$$[\tau_1(\varphi'(\bar x, \bar a')) \wedge \dots \wedge \tau_m(\varphi'(\bar x, \bar a'))] \cap S_X(\C)=\emptyset.$$ 
Since $([\tau_i(\varphi'(\bar x,\bar a'))] \cup [\tau_i[\psi'(\bar x, \bar b')]) \cap S_X(\C) =S_X(\C)$ for all $i=1,\dots,m$, we get that $S_X(\C) \setminus \bigcup_{i \leq m} [\tau_i(\psi'(\bar x,\bar b'))] \subseteq \bigcap_{i \leq m}  [\tau_i(\varphi'(\bar x, \bar a')] \cap S_X(\C)= \emptyset$. Hence, $\psi(\bar x, \bar b)$ is generic in $X$, a contradiction.\\[1mm]
(iii) $\rightarrow$ (iv). Each almost periodic type is weakly generic
  \cite[Corollary 1.8]{Ne1}, hence it does not fork over $\emptyset$ by (iii), and so its orbit is of bounded size by Fact \ref{fact: forking and boundedness in NIP}.\\[1mm]
(iv) $\rightarrow$ (v). Since for every $q \in S_X(\C)$ the function $\pi_q \colon {\mathcal M} \to S_X(\C)$ defined by $\pi_q(\eta):=\eta(q)$ is a homomorphism of $\aut(\C)$-flows, we see that $\im (\pi_q)$ is a minimal subflow of $S_X(\C)$. Thus, we get that for every $\eta \in {\mathcal M}$, $\im(\eta)$ consists of almost periodic types, which by (iv) implies that each type in $\im (\eta)$ has bounded orbit. Hence, by Fact \ref{fact: forking and boundedness in NIP}, for every $\eta \in {\mathcal M}$, $\im(\eta)$ consists of types which do not fork over $\emptyset$, so ${\mathcal M}$ is of bounded size by Corollary \ref{corollary: characterization of boundedness of minimal ideals in terms of non-forking}.\\[1mm]
%
(v) $\rightarrow$ (vi). It is enough to take any $\eta \in {\mathcal M}$ and a type in the image of $\eta$. The orbit of this type is of bounded size by Proposition \ref{proposition: characterization of boundedness of minimal ideals}.\\[1mm]
(vi) $\rightarrow$ (i). Since $\bar c$ enumerates the whole monster model, the fact that some type in $S_{\bar c}(\C)$ does not fork over $\emptyset$ implies that every type in $S(\emptyset)$ (in arbitrary variables) does not fork over $\emptyset$, i.e. $\emptyset$ is an extension base. 
\end{proof}

Alternatively, one could prove (i) $\rightarrow$ (v) directly.
Namely, by (i) $\rightarrow$ (ii),  if $\emptyset$ is an extension base, then forking equals dividing on $X$. By the paragraph following Corollary \ref{corollary: bounded = weakly invariant do not fork}, we conclude that ${\mathcal M}$ is of bounded size.
 
Question \ref{question: bounded = forking equals dividing} is quite interesting. 
Note that Theorem \ref{theorem: main characterization theorem under NIP} shows that the answer is positive in the case when $X=\tp(\bar c/\emptyset)$.

By the above theorem and \cite[Corollary 2.10]{HrPi}, we immediately get

\begin{cor}\label{corollary: G-compactness}
If ${\mathcal M}$ is of bounded size, then the theory $T$ is G-compact.
\end{cor}

The proof of Theorem \ref{theorem: main characterization theorem under NIP} yields the next proposition.
\begin{prop}\label{proposition: only one direction for bar alpha}
Let $X = \tp(\bar \alpha/\emptyset)$ (so ${\mathcal M}$ is a minimal left ideal of $EL(S_{\bar \alpha}(\C)$). Consider the following conditions.\\
i) $\emptyset$ is an extension base.\\
ii) Forking equals dividing on $X$.\\
iii) The weakly generic formulas in $X$ do not fork over $\emptyset$.\\
iv) The almost periodic types in $S_{\bar \alpha}(\C)$ have orbits of bounded size.\\
v) ${\mathcal M}$ is of bounded size.\\
vi) $\tp(\bar \alpha/\emptyset)$ does not fork over $\emptyset$.\\
Then (i) $\rightarrow$ (ii) $\rightarrow$ (iii) $\rightarrow$ (iv) $\rightarrow$ (v) $\rightarrow$ (vi). 
\end{prop}

\begin{ques}\label{question: characterization for bar alpha}
Are Conditions (ii) - (v) in the above proposition equivalent? 
\end{ques}

To see that (v) does not imply (i), even under NIP, take a 2-sorted structure with sorts $S_1$ and $S_2$ such that there is no structure on $S_1$, there are no interactions between $S_1$ and $S_2$, and the structure on $S_2$ is such that $\emptyset$ is not an extension base. Then, taking any $\alpha \in S_1$, we get that 
the minimal left ideal of $EL(S_{\alpha}(\C))$ is trivial, but (i) does not hold.

It is more delicate to build an example showing that (vi) does not imply (v). This will be dealt with in the appendix.

\subsection{A better bound on the size of the Ellis group}\label{subsection 6.2}

Under the NIP assumption, we will give a better bound on the size of the Ellis group of the flow $S_X(\C)$ than the one in Corollaries \ref{corollary: first part of main theorem 1} and \ref{corollary: final corollary concerning the size of the Ellis group}. 
Instead of contents and the family $R$ of types used in Section \ref{section 3}, we will show that one can use the family of global types invariant over a given small model $M$. The families of global types invariant over $M$ considered over monster models $\C_1 \succ \C_2$ can be also used instead of the families $R_{\C_1}$ and $R_{\C_2}$  obtained in Corollary \ref{corollary: corollary of technical lemmas} to get absoluteness of the Ellis group (namely, using these families, one can still show Lemmas \ref{lemma: for every f there is g}, \ref{lemma: for every g there is f}, \ref{lemma: existence of v such that Im(u)=Im(v)}, and \ref{lemma: the image of Im(F1) equals Im(F2)}, to get Corollary \ref{corollary: absoluteness of the Ellis group}; the further material on $\tau$-topologies also goes through).

So fix a small model $M \prec \C$, and let $I_M$ be the family of all types in $S_X(\C)$ which are invariant over $M$. Note that by Fact \ref{fact: forking and boundedness in NIP}, $I_M$ coincides with the set of all types in $S_X(\C)$ which do not fork over $M$. 
Recall that ${\mathcal M}$ denotes a minimal left ideal of $EL(S_X(\C))$.
The key observation is the following.

\begin{prop}\label{proposition: Im(u) contained in invariant types}
There exists an idempotent $u \in {\mathcal M}$ such that $\im(u) \subseteq I_M$.
\end{prop}

\begin{proof}
By Lemma \ref{lemma: u in place of eta}, it is enough to show that there exists $\eta \in EL(S_X(\C))$ such that $\im(\eta) \subseteq I_M$. For a type $p \in S_X(\C)$ and a formula $\varphi(\bar x,\bar a)$ which forks over $M$, let 
$$X_{p,\varphi}:= \{ \eta \in EL(S_X(\C)): \neg \varphi(\bar x,\bar a) \in \eta(p)\}.$$ 
It is is enough to show that the intersection of all possible sets
$X_{p,\varphi}$ is non-empty, as then any element $\eta$ in this
intersection will do the job. By compactness of $EL(S_X(\C))$, it
remains to show that the family of sets of the form $X_{p,\varphi}$
has the finite intersection property.  So consider any types
$p_1,\dots,p_n \in S_X(\C)$ and any formulas $\varphi_1(\bar x,\bar
a_1), \dots, \varphi_n(\bar x,\bar a_n)$ which all fork over $M$, and
let $\varphi(\bar x,\bar a)$ be the disjunction of them. 
By Fact \ref{fact: forking equals dividing from CaKa}, 
$\varphi$ divides over $M$, so choose an $M$-indiscernible sequence
$(\bar b_j)_{j<\omega}$ witnessing this. For each $i=1,\dots,n$ there
are at most finitely many $j$ with $\varphi(\bar x,\bar b_j)\in
p_i$. So chosse $j$ such that $\varphi(\bar x,\bar b_j)\not\in p_i$
for $i=1,\dots,n$. Let $\sigma\in\aut(\C)$ map $\bar b_j$ to $\bar
a$. Then $\sigma$, regarded as an element of $EL(S_X(\C))$, belongs to
$X_{p_1,\varphi_1}\cap\cdots\cap X_{p_n\varphi_n}$.
%
\end{proof}

Recall the following fact (see e.g. \cite[Lemma 2.5]{HrPi}). 
\begin{fct}\label{fact: few invariant types}
	If $T$ is NIP and $A\subseteq \C$, then there are at most $2^{|T|+|\bar x|+|A|}$ types in $S_{\bar x}(\C)$ which are invariant over $A$.
\end{fct}

As usual, $l_S$ denotes the length of the product $S$.
Choose any idempotent $u$ satisfying the conclusion of Proposition \ref{proposition: Im(u) contained in invariant types}.

\begin{cor}\label{corollary: better bound}
The function $F \colon u{\mathcal M} \to I_M^{I_M}$ given by $F(h)= h |_{I_M}$ is a group isomorphism onto the image of $F$, so the size of the Ellis group of the flow $S_X(\C)$ is bounded by $|I_M|^{|I_M|}$. In particular, the size of the Ellis group is bounded by 
$2^{2^{l_S+ |T|}}$. In the case when $l_S \leq 2^{|T|}$, this is bounded by $\beth_3(|T|)$, and when $l_S \leq |T|$, the bound equals  $2^{2^{|T|}}$.
\end{cor}

\begin{proof}
The first part follows from Proposition \ref{proposition: Im(u) contained in invariant types} and Lemma \ref{lemma: hM=uM}(3). For the second part, take $M$ of cardinality at most $|T|$. Then 
$|I_M| \leq 2^{l_S+|T|}$ by Fact \ref{fact: few invariant types}, and we finish using the first part.
\end{proof}

By Corollary \ref{corollary: the size for alpha bounded by the size for c}, Propositions \ref{proposition: invariants for c isomorphic to invariants for S} and \ref{proposition: unboundedly many sorts}, and Corollary \ref{corollary: better bound}, we get

\begin{cor}
The size of the Ellis group of the flow $S_X(\C)$ is bounded by $\beth_3(|T|)$. The sizes of the Ellis groups of the flows $S_{\bar c}(\C)$ and $S_{\bar \alpha}(\C)$ are bounded by ${2^{2^{|T|}}}$.
\end{cor}

As was mentioned at the beginning of this subsection, in the NIP case, one can also simplify the proof of absoluteness of the Ellis group from Section \ref{section 3} by omitting technical lemmas \ref{lemma: technical lemma 1}, \ref{lemma: technical lemma 2}, \ref{lemma: technical lemma 3} and Corollary \ref{corollary: technical corollary}, and then proceeding with $R_{\C_1}$ and $R_{\C_2}$ replaced by types invariant over $M$ (and using Proposition \ref{proposition: Im(u) contained in invariant types}). We leave the details as an exercise.

\subsection{The Ellis group conjecture for groups of automorphisms}\label{subsection 6.3}

Recall that  the Ellis group conjecture was formulated by the second
author in the case of a group $G$ definable in a model $M$. It says
that a certain natural epimorphism $\Phi$ (more precisely, taking the
coset of a realization of the given type) from the Ellis group of the
flow $S_{G,\ext}(M)$ (of all external types over $M$) to
$G^*/{G^*}^{00}_M$ is an isomorphism, at least under some reasonable
assumptions. In general, this turned out to be false,
e.g. $G:=\Sl_2(\R)$ treated as a group definable in the field of reals
(so an NIP structure) is a counter-example \cite{GiPePi}. Much more
counter-examples can be obtained via the obvious observation that
the epimorphism $\Phi$ factors through $G^*/{G^*}^{000}_M$ (see
\cite{KrPi}), namely we get a counter-example whenever ${G^*}^{000}_M
\ne {G^*}^{00}_M$.  On the other hand, \cite[Theorem 5.7]{ChSi}
confirms the Ellis group conjecture for definably amenable groups definable in NIP theories. 

Here, we study an analogous problem for the group $\aut(\C)$ in place
of the definable group $G$. Take the notation from Section
\ref{section 1}, so ${\mathcal M}$ is a minimal left ideal of
$EL:=EL(S_{\bar c}(\C))$, $u \in {\mathcal M}$ is an idempotent, and
we have maps $\hat{f}:EL\to\gal_L(T),\ f:u\M\to\gal_L(T)$ and $h:\gal_L(T)\to\gal_{KP}(T)$. We get the epimorphism $\hat{F}:= h \circ \hat{f} \colon EL \to \gal_{KP}(T)$. Let $F=\hat{F}|_{u{\mathcal M}} \colon u{\mathcal M} \to \gal_{KP}(T)$. This is also an epimorphism.
A natural counterpart of the Ellis group conjecture in our context says that $F$ is an isomorphism. It is clearly false for all non G-compact theories (as then $h$ is not injective). But we will prove it under the assumption of boundedness of the minimal left ideals in $EL$, and this is the precise content of Theorem \ref{theorem theorem from ChSi}.

\begin{reptheorem}{theorem theorem from ChSi}
Assume NIP. If ${\mathcal M}$ is of bounded size, then $F$ is an isomorphism.
\end{reptheorem}

Note that boundedness of the minimal left ideals in $EL$ exactly corresponds to the definable amenability assumption for definable NIP groups, because both conditions are equivalent to the existence of 
a type (in the type space appropriate for each of  the two contexts) with bounded orbit
(see \cite[Theorem 3.12]{ChSi} for the proof of this in the definable group case). In fact, 
the first author has proved that boundedness of the minimal left ideals in $EL$ is equivalent to the appropriate version of amenability of $\aut(\C)$ (which he calls {\em relative definable amenability}), but this belongs to a separate topic.

The proof of Theorem \ref{theorem theorem from ChSi} will be an adaptation of \cite[Theorem 5.7]{ChSi}.
%

Before the proof, let us recall some notation. Suppose that $p \in S(\C)$ is invariant over $A$. Then a sequence $(a_i)_{i<\lambda}$ is called a {\em Morley sequence in $p$ over $A$} if $a_i \models p|_{Aa_{<i}}$ for all $i$. Such a sequence is always indiscernible over $A$, and the type of a Morley sequence of length $\lambda$ in $p$ over $A$ does depend on the choice of this sequence and is denoted by $p^{(\lambda)}|_A$.

\begin{lem}[Counterpart of Proposition 5.1 of \cite{ChSi}]\label{lemma: 5.1 from ChSi}
Let $\varphi(\bar x, \bar a)$ be a formula, $\bar \alpha$ any (short or long) tuple in $\C$, and let 
$p \in S_{\bar \alpha}(\C)$ 
be a type whose $\aut(\C)$-orbit is of bounded size. Put $U_\varphi:=\{ \sigma/\autf_{KP}(\C) : \varphi(\bar x, \sigma(\bar a)) \in p\}$. Then $U_\varphi$ is constructible (namely, a Boolean combination of closed sets).
\end{lem}

\begin{proof}
Let $S=\{ \bar b: \varphi(\bar x, \bar b) \in p\}$, $V=\{\sigma \in \aut(\C) : \sigma(\bar a) \in S\}$, $\pi\colon \aut(\C) \to \aut(\C)/\autf_{KP}(\C)$ be the quotient map, and $\rho \colon \aut(\C) \to \C$ be given by $\rho(\sigma)=\sigma(\bar a)$. Then $V=\rho^{-1}[S]$ and $U_\varphi = \pi[V]$.

By assumption and Fact \ref{fact: forking and boundedness in NIP}, $p$
is invariant under $\autf_{KP}(\C)$ and does not fork over
$\emptyset$. Choose a small model $M\prec\C$ and let $\bar d$ realize
$p'^{(\omega)}|_M$, where $p'$ is the restriction of $p$ to the finitely many variables occurring in $\varphi(\bar x, \bar a)$; from now on, $\bar x$ is this finite tuple of variables. Let 
$$\alt_n(\bar x_0,\dots, \bar x_{n-1};\bar y) = \bigwedge_{i<n-1} \neg(\varphi(\bar x_i, \bar y) \leftrightarrow \varphi(\bar x_{i+1}, \bar y)).$$
By NIP and the argument from Proposition 2.6 of \cite{HrPi}, there is $N<\omega$ such that
$$S=\bigcup_{n<N} A_n \cap B_{n+1}^c,$$
where $A_n$ is the collection of all tuples $\bar b$ in $\C$ for which
$$(\exists \bar x_0,\dots, \bar x_{n-1}) (\bar x_0\dots \bar
x_{n-1}\equiv_{KP}\bar d|_n \wedge \alt_n(\bar x_0,\dots, \bar x_{n-1};\bar b) \wedge \varphi(\bar x_{n-1},\bar b)),$$
and $B_n$ is the collection of all tuples $\bar b$ in $\C$ for which
$$(\exists \bar x_0,\dots, \bar x_{n-1}) (\bar x_0\dots \bar
x_{n-1}\equiv_{KP}\bar d|_n \wedge \alt_n(\bar x_0,\dots, \bar x_{n-1};\bar b) \wedge \neg\varphi(\bar x_{n-1},\bar b)).$$

We see that each $A_n$ and $B_n$ is invariant under $\autf_{KP}(\C)$, so $\rho^{-1}[A_n]$ and $\rho^{-1}[B_n]$ are unions of $\autf_{KP}(\C)$-cosets, i.e. they are unions of cosets of $\ker(\pi)$. Therefore,
$$U_\varphi= \pi[\rho^{-1}[\bigcup_{n<N} A_n \cap B_{n+1}^c]] =
\pi[\bigcup_{n<N} \rho^{-1}[A_n] \cap \rho^{-1}[B_{n+1}]^c] =$$
$$ \bigcup_{n<N} \pi[\rho^{-1}[A_n] \cap \rho^{-1}[B_{n+1}]^c] =  \bigcup_{n<N} \pi[\rho^{-1}[A_n]] \cap \pi[\rho^{-1}[B_{n+1}]]^c.$$
So it remains to check that  $\pi[\rho^{-1}[A_n]]$ and $ \pi[\rho^{-1}[B_n]]$ are closed. 

Using the observation that $\rho^{-1}[A_n]$ and $\rho^{-1}[B_n]$ are unions of cosets of $\ker(\pi)$, we get
$$\pi^{-1}[\pi[\rho^{-1}[A_n]]] = \rho^{-1}[A_n] = \{ \sigma \in \aut(\C): \sigma(\bar a) \in A_n\},$$
$$\pi^{-1}[\pi[\rho^{-1}[B_n]]] = \rho^{-1}[B_n] = \{ \sigma \in \aut(\C): \sigma(\bar a) \in B_n\}.$$
%
By \cite[Lemma 4.10]{LaPi}, \cite[Fact 2.3(i)]{CLPZ}, and the fact that the topology on $\gal_{KP}(T)$ is the quotient topology induced from $\gal_L(T)$, since $A_n$ and $B_n$ are type-definable, we conclude that $\pi[\rho^{-1}[A_n]]$ and $ \pi[\rho^{-1}[B_n]]$ are closed.
\end{proof}

As usual, $\C' \succ \C$ is a monster model with respect to $\C$. Recall that $\hat{F}:= h \circ \hat{f} \colon EL \to \gal_{KP}(T)$ is given by $\hat{F}(\eta) = \sigma'/\autf_{KP}(\C')$, where $\sigma' \in \aut(\C')$ is such that $\sigma'(\bar c) \models \eta(\tp(\bar c/\C))$. Enumerate $S_{\bar c}(\C)$ as $\langle \tp({\bar c}_k/\C) : k<\lambda\rangle$ for some cardinal $\lambda$ and some tuples ${\bar c}_k\equiv \bar c$, where ${\bar c}_0=\bar c$. Let $q_k=\tp(\bar c_k/\C)$ for $k<\lambda$. By \cite[Remark 2.3]{KrPiRz}, it is true that for any $k<\lambda$, $\hat{F}(\eta) = \sigma'/\autf_{KP}(\C')$, where $\sigma' \in \aut(\C')$ is such that $\sigma'(\bar c_k) \models \eta(q_k)$. Let $\pi_k \colon EL \to S_{\bar c}(\C)$ be given by $\pi_k(\eta)=\eta(q_k)$.

\begin{lem}[Counterpart of Theorem 5.3 of \cite{ChSi}]\label{lemma: 5.2 from ChSi}
Assume that $\eta \in EL$, $p \in S_{\bar c}(\C)$ and $k<\lambda$ are
such that $\pi_k(\eta)=p$ and the $\aut(\C)$-orbit of $p$ is of
bounded size. Let $C = \cl(\aut(\C) \eta) \subseteq EL$, and
$\varphi(\bar x, \bar a)$ be a formula over $\C$. Put 
$$E_\varphi := \hat{F}[\pi_k^{-1}[[\varphi(\bar x, \bar a)]] \cap C] \cap \hat{F}[\pi_k^{-1}[[\neg \varphi(\bar x, \bar a)]] \cap C].$$
Then $E_\varphi$ is  closed with empty interior.
\end{lem}

\begin{proof}
Since $\hat{F}$ is continuous (by Remark 2.5 of \cite{KrPiRz}), we get that $E_{\varphi}$ is closed. 

Replacing $\eta$ by some element from the $\aut(\C)$-orbit of $\eta$, we can assume that $p = \tp(\sigma(\bar c_k)/\C)$ for some $\sigma \in \autf_{KP}(\C')$. 
Then $\hat{F}(\eta)=\sigma/\autf_{KP}(\C') = \id/\autf_{KP}(\C')$.  

By assumption and Fact \ref{fact: forking and boundedness in NIP}, $p$ is $\autf_{KP}(\C)$-invariant, so we can extend it uniquely to an $\autf_{KP}(\C')$-invariant type $\bar p \in S_{\bar c}(\C')$.

Let $U_\varphi$ be the set defined in Lemma \ref{lemma: 5.1 from ChSi}
(but for the type $\bar p$ in place of $p$, and working in $\C'$). We will show that $E_\varphi \subseteq \partial
U_\varphi^{-1}$, 
where $\partial A$ is the topological border of $A$: $\partial A = \cl(A)\cap \cl(A^c)$. This will finish the proof, as the topological border of a constructible set always has empty interior.

Consider any $g \in E_\varphi$ and an open neighborhood $V$ of $g$. By the definition of $E_\varphi$, there are $q,q' \in C$ such that $\hat{F}(q)=\hat{F}(q')=g$ and $q \in \pi_k^{-1}[[\varphi(\bar x,\bar a)]]$, $q' \in \pi_k^{-1}[[\neg\varphi(\bar x,\bar a)]]$. Then $q$ and $q'$ belong to the open set $\hat{F}^{-1}[V]$. So there are $\sigma_1,\sigma_2 \in \aut(\C)$ such that $\sigma_1 \eta \in \hat{F}^{-1}[V] \cap \pi_k^{-1}[[\varphi(\bar x,\bar a)]]$ and  $\sigma_2 \eta \in \hat{F}^{-1}[V] \cap \pi_k^{-1}[[\neg\varphi(\bar x,\bar a)]]$.

Take any $\bar \sigma_1 \in \aut(\C')$ extending $\sigma_1$. 
We have $\bar \sigma_1 \bar p \supset \sigma_1p=\sigma_1\pi_k(\eta) = \pi_k(\sigma_1 \eta) \ni \varphi(\bar x, \bar a)$, and so $\varphi (\bar x, \bar \sigma_1^{-1}(\bar a)) \in \bar p$, hence $\hat{F}(\sigma_1)= \bar \sigma_1/\autf_{KP}(\C')\in U_\varphi^{-1}$. On the other hand, we clearly have $\hat{F}(\sigma_1\eta) \in V$. Since $\hat{F}$ is a semigroup homomorphism and $\hat{F}(\eta)$ is the neutral element, we conclude that 
$$\hat{F}(\sigma_1)=\hat{F}(\sigma_1 \eta) \in V \cap U_\varphi^{-1}.$$
Similarly (and using the fact that $\bar p$ is $\autf_{KP}(\C')$-invariant),
$$\hat{F}(\sigma_2)=\hat{F}(\sigma_2 \eta) \in V \cap U_{\neg \varphi}^{-1}=V \cap (U_{\varphi}^{-1})^c.$$
As $V$ was an arbitrary open neighborhood of $g$, we get that $g \in \partial U_\varphi^{-1}$.
\end{proof}

Now, we have all the tools to prove Theorem  \ref{theorem theorem from ChSi} (which is a counterpart of Theorem 5.7 of \cite{ChSi}).



\begin{proof}[Proof of Theorem \ref{theorem theorem from ChSi}]



Recall that $F=\hat{F}|_{u\M}\colon u\M\to\gal_{KP}(T)$ is an epimorphism.
We need to show that $\ker(F) =\{u\}$. So take any $\eta \in \ker(F)$. It is enough to show that $r\eta = r$ for some $r \in {\mathcal M}$, because then $\eta= (ur)^{-1} ur \eta = (ur)^{-1}ur = u$. But this is equivalent to finding $r \in {\mathcal M}$ such that $\pi_k(r\eta)=\pi_k(r)$ for all $k<\lambda$.

Let ${\mathcal F}$ be the filter of comeager subsets of
$\gal_{KP}(T)$. Since $\gal_{KP}(T)$ is compact Hausdorff, it is a
Baire space, so $\mathcal F$ is a proper filter. Let ${\mathcal F}'$ be an ultrafilter extending ${\mathcal F}$. For any $g \in \gal_{KP}(T)$ let $r_g \in {\mathcal M}$ be such that $\hat{F}(r_g)=g$. Put
$$r:= \lim_{{\mathcal F}'} r_g.$$
Then $r \in {\mathcal M}$, and we will show that $\pi_k(r\eta)=\pi_k(r)$ for all $k$, which completes the proof. 

Suppose for a contradiction that $\pi_k(r\eta) \ne \pi_k(r)$ for some $k$. Then $\varphi(\bar x) \in  \pi_k(r\eta)$ and $\neg \varphi(\bar x) \in  \pi_k(r)$ for some formula $\varphi(\bar x)$ with parameters.  Thus,
$$P:= \left\{ g \in \gal_{KP}(T): r_g\eta \in \pi_k^{-1}[[\varphi(\bar x)]]\;\, \mbox{and}\;\, r_g \in \pi_k^{-1}[[\neg\varphi(\bar x)]]\right\} \in {\mathcal F}'.$$

Since $\eta \in {\mathcal M}$ and ${\mathcal M}$ is of bounded size, so is the $\aut(\C)$-orbit of $\pi_k(\eta)$, and we can apply Lemma  \ref{lemma: 5.2 from ChSi}. To be consistent with the notation from this lemma, put $C:={\mathcal M}=\cl(\aut(\C)\eta)$.  
By Lemma \ref{lemma: 5.2 from ChSi}, $E_{\varphi}^c \in {\mathcal F'}$, so we can find $g \in P \cap E_{\varphi}^c$.
Put $S=\hat{F}^{-1}[E_{\varphi}^c]$. Then $S$ is open in $EL$ and $r_g \in S$. Moreover, $r_g\eta$ belongs to the open set $\pi_k^{-1}[[\varphi(\bar x)]]$ and  $r_g$ belongs to the open set $\pi_k^{-1}[[\neg \varphi(\bar x)]]$.

Since $r_g$ is approximated by automorphisms of $\C$, $r_g=r_gu$, and the semigroup operation in $EL$ is left continuous, we get that there is $\sigma \in \aut(\C)$ such that 
$$\sigma \eta \in \pi_k^{-1}[[\varphi(\bar x)]], \; \sigma u \in \pi_k^{-1}[[\neg \varphi(\bar x)]], \;\sigma \in S.$$

We conclude that
$$\hat{F}(\sigma \eta)=\hat{F}(\sigma)F(\eta)=\hat{F}(\sigma) \in E_{\varphi}^c\;\, \mbox{and}\;\, \hat{F}(\sigma u)=\hat{F}(\sigma)F(u)=\hat{F}(\sigma) \in E_{\varphi}^c.$$
But 
$$\sigma \eta \in C \cap \pi_k^{-1}[[\varphi(\bar x)]]\;\, \mbox{and}\;\, \sigma u \in C \cap \pi_k^{-1}[[\neg\varphi(\bar x)]]$$
implies that $\hat{F}(\sigma) =\hat{F}(\sigma \eta)=\hat{F}(\sigma u) \in E_{\varphi}$, which is a contradiction.
\end{proof}

We have seen in Corollary \ref{corollary: G-compactness} that boundedness of ${\mathcal M}$ implies that the theory $T$ is G-compact. Note that this also follows immediately from Theorem \ref{theorem theorem from ChSi} (i.e. the epimorphism $h\colon \gal_L(T) \to \gal_{KP}(T)$ from Section \ref{section 1} is an isomorphism).
As was mentioned in the introduction, although by Section \ref{section
  3} we know that the Ellis group $u{\mathcal M}$ is always bounded
and absolute, taking any non G-compact theory, we get that $F$ is not
injective (i.e. the counterpart of the Ellis group conjecture is false in general, even in NIP theories). Corollary \ref{corollary: f iso implies G-compactness} shows even more, namely that if $T$ is non G-compact, then the epimorphism $f \colon u{\mathcal M} \to \gal_L(T)$ is not an isomorphism. It is thus still  interesting to look closer at the Ellis group $u{\mathcal M}$ -- an essentially new model-theoretic invariant of an arbitrary theory.\\

Now, we take the opportunity and describe a natural counterpart of the epimorphisms $f$ and $F$ for a short tuple $\bar \alpha$ in place of ${\bar c}$. But first note that in general there is no chance to find an epimorphism from the Ellis group of $S_{\bar \alpha}(\C)$ to $\gal_{KP}(T)$ for the same reason as (v) $\nrightarrow$ (i) in Proposition \ref{proposition: only one direction for bar alpha}: Take a 2-sorted structure with sorts $S_1$ and $S_2$ such that there is no structure on $S_1$, there are no interactions between $S_1$ and $S_2$, and the structure on $S_2$ is such that the corresponding $\gal_{KP}$ is non-trivial. Then, taking any $\alpha \in S_1$, we get that the Ellis group of the flow $S_{\alpha}(\C)$ is trivial, whereas $\gal_{KP}(T)$ is non-trivial, so there is no epimorphism. In order to resolve this problem, we need to use a certain localized version of the notion of Galois group proposed in \cite{DoKiLe}. 

Let $p_1=\tp(\bar \alpha/\emptyset)$. Following the notation from \cite{DoKiLe}, we define $\gal_L^{\fix,1}(p_1)$ as the quotient of the group of all elementary permutations of $p_1(\C)$ by the subgroup $\autf_L^{\fix,1}(p_1(\C))$ of all elementary permutations of $p_1(\C)$ fixing setwise the $E_L$-class of each realization of $p_1$ (but not necessarily of each tuple of realizations of $p_1$, and that is why we write superscript 1). $\gal_L^{\fix,1}(p_1)$ can be, of course, identified with the quotient of $\aut(\C)$ by the group $\autf_{L,p_1}^{\fix,1}(\C)$ of all automorphisms  fixing setwise the $E_L$-class of each realization of $p_1$. The group $\gal_{KP}^{\fix,1}(p_1)$ is defined analogously. It is easy that these groups do not depend on the choice of the monster model. We have the following diagram of natural epimorphisms.
\begin{figure}[H]
		\centering
		\begin{tikzcd}
			\gal_L(T) \arrow[r,"h"]\arrow[d]&\gal_{KP}(T)\arrow[d] \\
			\gal_L^{\fix,1}(p_1)\arrow[r,"h_1"] & \gal_{KP}^{\fix,1}(p_1)
		\end{tikzcd}
	\end{figure}

In particular, on $\gal_{L}^{\fix,1}(p_1)$ and $\gal_{KP}^{\fix,1}(p_1)$ we can take the quotient topologies coming from the vertical maps, and it is easy to check that $\gal_{L}^{\fix,1}(p_1)$ is a quasi-compact topological group, $\gal_{KP}^{\fix,1}(p_1)$ is a compact topological group, $h_1$ is a topological quotient mapping (i.e. a subset of $\gal_{KP}^{\fix,1}(p_1)$ is closed if and only if its preimage is closed),
but there is no reason why $\ker(h_1)$ should be the closure of the identity in $\gal_L^{\fix,1}(p_1)$ (it certainly contains it).
On can easily check that both groups $\gal_L^{\fix,1}(p_1)$ and $\gal_{KP}^{\fix,1}(p_1)$ do not depend on the choice of $\C$ as topological groups.

Now, we will define a counterpart of the map $\hat{f}$ recalled in Section \ref{section 1}. We enumerate $S_{\bar \alpha}(\C)$ as $\langle \tp({\bar \alpha}_k/\C) : k<\lambda\rangle$ for some cardinal $\lambda$ and some tuples ${\bar \alpha}_k\equiv \bar \alpha$, where ${\bar \alpha}_0=\bar \alpha$. Let $q_k=\tp(\bar \alpha_k/\C)$ for $k<\lambda$.
Let $\C'\succ \C$ be a monster model with respect to $\C$. 
As in \cite[Proposition 2.3]{KrPiRz}, using compactness and the density of $\aut(\C)$ in $EL(S_{\bar \alpha}(\C))$, we easily get

\begin{rem}\label{remark: Proposition 2.3 from KrPiRz}
For every $\eta \in EL(S_{\bar \alpha}(\C))$ there is $\sigma' \in \aut(\C')$ such that for all $k<\lambda$, $\eta(\tp(\bar \alpha_k/\C))= \tp(\sigma'(\bar \alpha_k)/\C)$. More generally, given any sequence $\langle \bar \beta_k: k \in I \rangle$, where $I$ is any set of size bounded with respect to $\C'$ and $\{ \tp(\bar \beta_k/\C) : k \in I \} \subseteq S_{\bar \alpha}(\C)$,  for every $\eta \in EL(S_{\bar \alpha}(\C))$ there is $\sigma' \in \aut(\C')$ such that for all $k\in I$, $\eta(\tp(\bar \beta_k/\C))= \tp(\sigma'(\bar \beta_k)/\C)$.  
\end{rem} 

By this remark and the definition of $\gal_L^{\fix,1}(p_1)$ (which here will be computed using $\C'$), one easily gets that $\hat{f_1}\colon EL(S_{\bar \alpha}(\C)) \to \gal_L^{\fix,1}(p_1)$ given by
$$\hat{f}_1(\eta)=\sigma'|_{p_1(\C')} /\autf_L^{\fix,1}(p_1(\C')),$$
where $\sigma' \in \aut(\C')$ is such that $\sigma'(\bar \alpha_k) \models \eta(q_k)$ for all $k<\lambda$, is a well-defined function. (To see that the value $\hat{f}_1(\eta)$ does not depend on the choice of $\sigma'$, one should use the obvious fact that among $\bar \alpha_k$, $k<\lambda$, there are representatives of all $E_L$-classes on $p_1(\C')$.) Note that the difference in comparison with $\hat{f}$ is that here we have to use all types from $S_{\bar \alpha}(\C)$ and not just one type $\tp(\bar \alpha/\C)$. One can check that the definition of $\hat{f}_1$ does not depend on the choice of the tuples $\bar \alpha_k \models q_k$ for $k<\lambda$.

We also leave as an easy exercise to check that if $\bar \alpha$ contains a (small) model or $\bar \alpha$ is replaced by $\bar c$, then $\gal_L^{\fix,1}(p_1)$ is isomorphic to $\gal_L(T)$, and for $\bar \alpha$ replaced by $\bar c$, the above $\hat{f}_1$ can be identified with $\hat{f}$. The same remark applies to $\gal_{KP}^{\fix,1}(p_1)$ and the epimorphism $\hat{F}_1\colon EL(S_{\bar \alpha}(\C)) \to \gal_{KP}^{\fix,1}(p_1) $ (defined before Question \ref{question: Newelski's conjecture for bar alpha} below).

\begin{prop}\label{proposition: f1 is an epimorphism}
$\hat{f}_1$ is a continuous semigroup epimorphism.
\end{prop}

\begin{proof}
The fact that $\hat{f}_1$ is onto is completely standard. It follows from the fact that for every $\sigma' \in \aut(\C')$ the coset $\sigma' \autf_L(\C')$ contains some $\sigma'' \in \aut(\C')$ which extends an automorphism $\sigma \in \aut(\C)$. 

Now, let us check continuity. Let $C \subseteq \gal_L^{\fix,1}(p_1)$ be closed. By the definition of the topology, $\{ \tp(\sigma'((\bar \alpha_k)_{k<\lambda})/\C) : \sigma'|_{p_1(\C')}/\autf_L^{\fix,1}(p_1(\C')) \in C\}$ is closed. Hence, $D:= \{ \langle \tp(\sigma'(\bar \alpha_k)/\C) : k < \lambda \rangle : \sigma'|_{p_1(\C')}/\autf_L^{\fix,1}(p_1(\C')) \in C\}$ is closed in $S_{\bar \alpha}(\C)^{S_{\bar \alpha}(\C)}$. On the other hand, $\hat{f}_1^{-1}[C] = EL(S_{\bar \alpha}(\C)) \cap D$. So $\hat{f}_1^{-1}[C]$ is closed. 

It remains to check that $\hat{f}_1$ is a homomorphism, which is slightly more delicate in comparison with the proof of the same statement for $\hat{f}$.
Take $\eta_1,\eta_2 \in EL(S_{\bar \alpha}(\C))$.  For each $j<\lambda$ take a unique $k_j<\lambda$ such that $\eta_2(q_j)=q_{k_j}$. 
By Remark \ref{remark: Proposition 2.3 from KrPiRz}, there is $\sigma_2' \in \aut(\C')$ such that 
$$\sigma_2'(\bar \alpha_j) \models \eta_2(q_j)=q_{k_j}$$ 
for all $j<\lambda$. 
Also by Remark \ref{remark: Proposition 2.3 from KrPiRz}, but applied to the sequence $(\bar \beta_{k,j})_{k<\lambda ,j\leq\lambda}$ (whose entries satisfy the types $q_k$, $k<\lambda$) defined by 
$$
\bar \beta_{k,j}:=\left\{
\begin{array}{ll}
\sigma_2'(\bar \alpha_j) & \mbox{when}\;\, j<\lambda\;\, \mbox{and}\;\, k=k_j\\ 
\bar \alpha_k & \mbox{otherwise}, 
\end{array}
\right.
$$
we can find $\sigma_1' \in \aut(\C')$ such that 
$$\sigma_1'(\bar \beta_{k,j}) \models \eta_1(q_k)$$
for all $k<\lambda$ and $j\leq \lambda$.

Then $(\sigma_1' \sigma_2')(\bar \alpha_j) =\sigma_1'(\beta_{k_j,j}) \models \eta_1(q_{k_j})=\eta_1(\eta_2(q_j)) =(\eta_1\eta_2)(q_j)$
for all $j<\lambda$. Hence,
$$\hat{f}_1(\eta_1\eta_2) = (\sigma_1' \sigma_2')|_{p_1(\C')}/\autf_L^{\fix,1}(p_1(\C')).$$
From the definition of $\hat{f}_1$, we also get
$$\hat{f}_1(\eta_2) = \sigma_2'|_{p_1(\C')}/ \autf_L^{\fix,1}(p_1(\C')).$$
By the last two formulas, in order to finish the proof that $\hat{f}_1$ is a homomorphism, it remains to check that
$$\hat{f}_1(\eta_1) = \sigma_1'|_{p_1(\C')}/ \autf_L^{\fix,1}(p_1(\C')),$$
%
but this follows from the fact that $\sigma_1'(\bar \alpha_k) = \sigma_1'(\bar \beta_{k,\lambda}) \models \eta_1(q_k)$ for all $k<\lambda$.
\end{proof}

Let ${\mathcal M}_1$ be a minimal left ideal in $EL(S_{\bar \alpha}(\C))$ and $u_1 \in {\mathcal M}_1$ an idempotent. Let $f_1=\hat{f}_1|_{u_1{\mathcal M}_1} \colon u_1{\mathcal M}_1 \to \gal_L^{\fix,1}(p_1)$. Since $u_1{\mathcal M}_1=u_1EL(S_{\bar \alpha}(\C)) u_1$, Proposition \ref{proposition: f1 is an epimorphism} implies that $f_1$ is a group epimorphism. Let 
$$\hat{F}_1 = h_1 \circ \hat{f}_1 \colon EL(S_{\bar \alpha}(\C)) \to  \gal_{KP}^{\fix,1}(p_1)\;\, \mbox{and}\;\, F_1=h_1 \circ f_1 \colon u{\mathcal M}_1 \to \gal_{KP}^{\fix,1}(p_1).$$ 
A natural counterpart of the Ellis group conjecture in this context says that $F_1$  is an isomorphism. It is false in general (by taking any example where $E_L$ differs from $E_{KP}$ on $p_1(\C)$). 
However, one could ask whether it is true assuming that ${\mathcal M}_1$ is of bounded size. (Note that then $E_L$ coincides with $E_{KP}$ on $p_1(\C)$ by Proposition \ref{proposition: only one direction for bar alpha} and Corollary 2.10 of \cite{HrPi}).

\begin{ques}\label{question: Newelski's conjecture for bar alpha}
Is it true that if ${\mathcal M}_1$ is of bounded size, then $F_1$ is an isomorphism?
\end{ques}

Actually, most of the above proof of Theorem \ref{theorem theorem from ChSi} can be easily adjusted to the context of Question \ref{question: Newelski's conjecture for bar alpha}. The only problem that appears is that (as we will see in Example \ref{example: counter-example to the variant of the theorem from ChSi}(6)) it is not the case that each type $p \in S_{\bar \alpha}(\C)$ (in a NIP theory) which does not fork over $\emptyset$ is invariant under $\autf_{KP,p}^{\fix,1}(\C)$ (the group of automorphisms fixing setwise the $E_{KP}$-class of each realization of $p:=\tp(\bar\alpha/\emptyset)$). This affects the proof of Lemma \ref{lemma: 5.1 from ChSi} and the final part of the proof of Lemma \ref{lemma: 5.2 from ChSi} (namely, we do not have $U_{\neg \varphi}=U_{\varphi}^c$). We finish with an example showing that the answer to Question \ref{question: Newelski's conjecture for bar alpha} is negative. We get even more, namely that the map $f_1 \colon u_1{\mathcal M}_1 \to \gal_L^{\fix,1}(p_1)$ is not an isomorphism although ${\mathcal M}_1$ is finite. In this example, we will compute a minimal left ideal of $EL(S_{\bar \alpha}(\C))$ and the Ellis group. The notation in this example is not fully compatible with the one used so far (e.g. $f_1$ will denote something else, and the role of $p_1$ will be played by $p$.)

\begin{ex}\label{example: counter-example to the variant of the theorem from ChSi}
Let $\mathbb{Q}_{l}$ and $\mathbb{Q}_{r}$ be two disjoint copies of the rationals.
Consider the 2-sorted structure $M$ with the sorts $S_1 := \mathbb{Q} \times  \mathbb{Q}$ and $S_2 :=  \mathbb{Q}_{l} \cup  \mathbb{Q}_{r}$, the equivalence relation $E$ on $S_2$ with two classes $\mathbb{Q}_{l}$ and $\mathbb{Q}_{r}$, 
the order $\leq_{S_2}$ on $S_2$ which is the standard order on each of the sets $\mathbb{Q}_{l}$ and $\mathbb{Q}_{r}$ (and no element of  $\mathbb{Q}_{l}$ is comparable with an element of $\mathbb{Q}_{r}$), and the binary relation $R$ on the product $S_1 \times S_2$ defined by
$$
R(a,b) \iff 
\left\{ \begin{array}{lll} 
a \in (-\infty,b] \times \mathbb{Q} & \mbox{for} & b \in \mathbb{Q}_{l},\\
a \in  \mathbb{Q} \times (-\infty,b] & \mbox{for} & b \in \mathbb{Q}_{r}.
\end{array} \right.
$$
Let $T=\Th(M)$ and $\C \succ M$ be a monster model. $T$ is
interpretable in the theory of $(\mathbb{Q},\leq)$, hence it is
$\aleph_0$-categorical and has NIP. Also, $M$ is saturated.  

Of course, $\mathbb{Q}_l$ as a set is definable over $\mathbb{Q}_l$ treated as an imaginary element (i.e. the ``left'' class of the $\emptyset$-definable equivalence relation $E$).

Consider $\leq_1$ and $\leq_2$ on $S_1$ which are given by the standard orders in the first and in the second coordinate, respectively. Then 
$$
\begin{array}{ll}
a_1 \leq_1 a_2 \iff (\exists y \in \mathbb{Q}_{l}) (R(a_1,y) \wedge \neg R(a_2,y)) \lor (\forall y \in \mathbb{Q}_{l}) (R(a_1,y) \Leftrightarrow R(a_2,y)),\\
a_1 \leq_2 a_2 \iff (\exists y \in \mathbb{Q}_{r}) (R(a_1,y) \wedge \neg R(a_2,y)) \lor (\forall y \in \mathbb{Q}_{r}) (R(a_1,y) \Leftrightarrow R(a_2,y)).
\end{array}
$$
So  $\leq_1$ and $\leq_2$ are both definable over the imaginary element $\mathbb{Q}_{l}$, and they are both dense, linear preorders. Let $\sim_i$ be given by $x_1 \leq_i x_2 \wedge x_2 \leq_i x_1$ and $\leq_i'$ be the induced linear order on $S_1/\!\!\sim_i$, for $i=1,2$.

Let $f_1 \colon S_1 \to \mathbb{Q}_l$ and $f_2 \colon S_1 \to \mathbb{Q}_r$ be given by $f_1(x) = \min \{ y \in \mathbb{Q}_l : R(x,y)\}$ and $f_2(x) = \min \{ y \in \mathbb{Q}_r : R(x,y)\}$. 
These functions are definable over the imaginary element $\mathbb{Q}_{l}$. Moreover, $x_1 \leq_i x_2 \iff f_i(x_1) \leq_i f_i(x_2)$ for $i=1,2$. Therefore, $f_1$ and $f_2$ induce isomorphisms $f_1' \colon (S_1/\!\!\sim_1,\leq_1') \to (\mathbb{Q}_l,\leq_{S_2})$ and  $f_2' \colon (S_1/\!\!\sim_2,\leq_2') \to (\mathbb{Q}_r,\leq_{S_2})$. Note that $R$ is definable using $\leq_{S_2}$, $f_1$ and $f_2$: 
$$R(x,y) \iff (y \in \mathbb{Q}_l \wedge f_1(x) \leq_{S_2} y) \lor (y \in \mathbb{Q}_r \wedge f_2(x) \leq_{S_2} y).$$
 
Let $\alpha = (0,0) \in S_1$.
The following statements are true.
\begin{enumerate}
\item In each of the sorts $S_1$ and $S_2$, there is a unique complete type over $\emptyset$ (in the theory $T$); thus $S_{S_1}(\C) = S_{\alpha}(\C)$. In fact, $\aut(M)$ acts transitively on $S_1$ and on $S_2$, and even the group of automorphisms fixing setwise each of the two classes of $E$ acts transitively on $S_1$. From now on, let $p=\tp(\alpha/\emptyset)$ be the unique type in $S_{S_1}(\emptyset)$.
\item Each automorphism of $\C$ which switches the two classes of $E$ also switches the orders $\leq_1$ and $\leq_2$. The automorphisms of $\C$ which preserve classes of $E$ also preserve $\leq_1$ and $\leq_2$.
\item There is no non-trivial, bounded, invariant equivalence relation on $S_1$. So $E_L|_{S_1}=E_{KP}|_{S_1}=E_{Sh}|_{S_1} = \; \equiv\!\!\!|_{S_1}$ is the total relation on $S_1$. Thus, $\autf_{L}^{\fix,1}(p(\C)) = \autf_{KP}^{\fix,1}(p(\C))=\aut(p(\C))$ and $\gal_L^{\fix,1}(p)=\gal_{KP}^{\fix,1}(p)$ is the trivial group.
\item The quantifier-free $\{\leq_1,\leq_2\}$-type over $\emptyset$ of any finite tuple in $S_1$ generates a complete type over the name of ${\mathbb Q}_l$ in the language of $T$.
\item For every $\epsilon_1,\epsilon_2\in\{0,1\}$ the set of formulas
  $Z_{\epsilon_1,\epsilon_2}:=\{(x\leq_1 a)^{\epsilon_1}\land(x\leq_2
  a)^{\epsilon_2}:a\in S_1\}$ generates a type
  $p_{\epsilon_1\epsilon_2}\in S_{S_1}(\C)$ that does not fork over
  $\emptyset$ (here $\varphi^0=\varphi$
  and $\varphi^1=\neg\varphi$). These are the only non-forking types
  in $S_{S_1}(\C)$. $p_{00}$ and $p_{11}$ are invariant under
  $\aut(\C)$, and the types $p_{01}$ and $p_{10}$ form a single orbit. 
\item  The types $p_{01},p_{10}$ are $\autf_{KP}(\C)$-invariant (as non-forking types in a NIP theory), but they are not $\autf_{KP,p}^{\fix,1}(\C)$-invariant.
\item There is $\eta_0 \in  EL(S_\alpha(\C))$ such that $\im(\eta_0)=\{p_{00},p_{11},p_{01},p_{10}\}$.
\item Each element $\eta \in EL(S_\alpha(\C))$ acts trivially on $p_{00},p_{11},p_{01},p_{10}$ or acts trivially on $p_{00}$ and $p_{11}$ and switches $p_{01}$ and $p_{10}$, and there exists $\eta$ which switches $p_{01}$ and $p_{10}$. Therefore, ${\mathcal M}_1:=EL(S_\alpha(\C)) \eta_0$ is a minimal left ideal in $EL(S_\alpha(\C))$, has two elements, and coincides with the Ellis group. Thus, the Ellis group of $S_\alpha(\C)$ has two elements, whereas $\gal_{KP}^{\fix,1}(p)$ is trivial (by (3)), so the answer to Question \ref{question: Newelski's conjecture for bar alpha} is negative.
\end{enumerate}
\end{ex}
\begin{proof}
(1) Whenever $g \colon S_2 \to S_2$ is an $\{E,\leq_{S_2}\}$-automorphism such that $g_l:=g|_{\mathbb{Q}_l} \colon \mathbb{Q}_l \to  \mathbb{Q}_l$ and $g_r:=g|_{\mathbb{Q}_r} \colon \mathbb{Q}_r \to  \mathbb{Q}_r$, then for $f \colon S_1 \to S_1$ given by $f(x_1,x_2)=(g_l(x_1),g_r(x_2))$ we see that $f \cup g$ is an automorphism of $M$. Similarly, whenever $g \colon S_2 \to S_2$ is an $\{E,\leq_{S_2}\}$-automorphism such that $g_l:=g|_{\mathbb{Q}_l} \colon \mathbb{Q}_l \to  \mathbb{Q}_r$ and $g_r:=g|_{\mathbb{Q}_r} \colon \mathbb{Q}_r \to  \mathbb{Q}_l$, then for $f \colon S_1 \to S_1$ given by $f(x_1,x_2)=(g_r(x_2),g_l(x_1))$ we see that $f \cup g$ is an automorphism of $M$. Having this, transitivity of $S_1$ and $S_2$ follows easily from the homogeneity of $(\mathbb{Q},\leq)$ and the existence of the obvious $\{E,\leq_{S_2}\}$-automorphism of $S_2$ switching the two $E$-classes.

(2) is obvious by the descriptions of $\leq_1$ and $\leq_2$ in terms of $\mathbb{Q}_l$, $\mathbb{Q}_r$ and $R$.
%

(3) First notice that if $a_0=(x_0,y_0),a_1=(x_1,y_1)\in S_1^M$, then
  there are $a_n=(x_n,y_n)\in S_1^M,n\geq 2$ such that the sequence
  $(a_n)$ is indiscernible in $M$. Indeed, it is enough to choose
  $x_n,y_n$ so that the sequences $(x_n),(y_n)$ are either constant
  or strictly monotonous. By saturation of $M$, we see that $E_L$ is
  total on $S_1$.

(4) It is easy to check it for tuples from  $S_1^M$ (see the proof of (1)), which is enough.

(5) 
By (4), for every $\epsilon_1,\epsilon_2\in\{0,1\}$ and $a_1,\dots,a_n \in S_1$ the formula 
$$\bigwedge_i (x\leq_1 a_i)^{\epsilon_1}\land (x\leq_2 a_i)^{\epsilon_2}\land \neg x\sim_1 a_i\land\neg x\sim_2 a_i$$ 
generates a complete type over $a_1,\dots,a_n$ and the name of $\mathbb{Q}_l$.
So the sets
  $Z_{\epsilon_1,\epsilon_2}$ generate complete types over $\C$. By
  (2), $p_{00}$ and $p_{11}$ are invariant and $p_{01}$ and $p_{10}$
  are either preserved or switched by any automorphism, so they do not
  fork over $\emptyset$.   Now, suppose $q \in S_{S_1}(\C) \setminus \{p_{00},p_{11},p_{01},p_{10}\}$. Then $q$ contains the formula $x \leq_{1} a_1  \wedge a_2 \leq_1 x$ or the formula $x \leq_{2} a_1  \wedge a_2 \leq_2 x$ for some $a_1,a_2 \in S_2$ (computed in $\C$). But each of these formulas is easily seen to divide over $\emptyset$.

(6) By (2), any automorphism which switches the $E$-classes also switches $p_{01}$ and $p_{10}$, but by (3), $\autf_{KP,p}^{\fix,1}(\C)=\aut(\C)$.

(7)
By (4), for any $a_1, \dots,a_n, l \in S_1$ (computed in $\C$),  we can choose $\sigma_{\bar a,l} \in \aut(\C)$ which preserve the orders $\leq_1$ and $\leq_2$ and such that $\sigma_{\bar a, l}(a_i) \leq_1 l$ and $\sigma_{\bar a, l}(a_i) \leq_2 l$ for all $i=1,\dots,n$. Then $(\sigma_{\bar a, l})$ is a net with the directed order on the indexes given by: $(a_1, \dots,a_n, l) \leq (a_1', \dots,a_n', l')$ if and only if $\bar a$ is a subsequence of $\bar a'$ and $l'\leq_1 l$ and $l' \leq_2 l$. Now, the limit of a convergent subnet will be the desired $\eta_0$.

(8) The first sentence follows from (5). This implies that the image of each element in ${\mathcal M}_1$ equals $\{p_{00},p_{11},p_{01},p_{10}\}$, ${\mathcal M}_1$  is minimal  and has two elements. Hence, any idempotent $u_1 \in {\mathcal M}_1$ fixes $p_{00},p_{11},p_{01},p_{10}$, so taking $\eta \in EL(S_{\alpha}(\C))$ such that $\eta(p_{01})=p_{10}$, we get that $u_1\eta u_1 (p_{01})=p_{10}$, and hence  $u_1{\mathcal M}_1$ has two elements and coincides with ${\mathcal M}_1$.
\end{proof}

\section{Appendix: an example}

In this appendix, we give an example showing that in general (vi) does not imply (v) in Proposition \ref{proposition: only one direction for bar alpha}: we construct an NIP structure $\C$ and a type $p$ over $\emptyset$ such that $p$ does not fork over $\emptyset$, however $EL(S_p(\C))$ has an unbounded minimal ideal. In fact the theory of $\C$ will be $\omega$-categorical.

\subsection{Ordered ultrametric spaces}

An {\em ordered ultrametric space} is a totally ordered set $(M,\leq)$ equipped with a distance function $d:M^{2} \to D$, where $(D,\unlhd,0)$ is a totally ordered set with minimal element $0$ such that:
\begin{itemize}
	\item $d(x,y)=0 \iff x=y$;
	\item $d(x,y)=d(y,x)$;
	\item $x<y<z \Longrightarrow d(x,z)=\max_{\unlhd} \{d(x,y),d(y,z)\}$.
\end{itemize}

As an example of such a space, take $(K,v,\leq)$ an ordered valued field with convex valuation ring. Consider the distance $d(x,y)=\val(x-y)$ taking value in $(\Gamma,\geq,\infty)$: the value group equipped with the reverse order and $\infty$ playing the role of $0$. This is an ordered ultrametric space.

Model theoretically, we represent ordered ultrametric spaces as two sorted structures $(M,D)$ in the language $L=\{0,\leq,\unlhd,d\}$, where $0$ is a constant symbol for the minimal element of $D$, $\leq$ is the order on $M$, $\unlhd$ is the order on $D$ and $d:M^2\to D$ is the distance function.
Let $\mathcal C$ be the class of finite ordered ultrametric spaces in the language $L$ (so both sorts are finite).
\begin{prop}
	The class $\mathcal C$ has the joint embedding and amalgamation properties.
\end{prop}

\begin{proof}
	Since all structures contain the structure with $M$ empty and $D=\{0\}$, we only need to show amalgamation. Let $A,B,C\in \mathcal C$ with embedding $f:A \to B$ and $g: A \to C$ and identify the images of $f$ and $g$ with $A$. We seek a structure $E$ which completes the square. Let $M(A)$ and $D(A)$ be the two sorts of $A$ and same for $B$ and $C$. Without loss, $D(B)$ and $D(C)$ both have at least two elements. First amalgamate $D(B)$ and $D(C)$ freely into $D(E)$ (so no element of $D(B)\setminus D(A)$ coincides with an element of $D(C)\setminus D(A)$ and extend the order in an arbitrary way). To amalgamate the $M$ sorts, assume first that $M(A)$ is empty. Define then $M(E)$ as the disjoint union of $M(B)$ and $M(C)$ ordered so that all elements of $M(B)$ are before all elements of $M(C)$. Now if $a\in M(A)\subseteq M(C)$ and $b\in M(B)\subseteq M(C)$, set $d(a,b) = m$, where $m=\max_{\unlhd} D(E)$. This makes $E=(M(E),D(E))$ into an ordered ultrametric space.

Assume now that $M(A)$ is not empty. Let $x<y$ be two consecutive points in $A$ and set $d=d(x,y)$. Let $B'$ be the points of $M(B)$ strictly between $x$ and $y$. Define further $B_0=\{z\in B':d(x,z)<d\}$ and $B_1=B'\setminus B_0$. Note that $x<B_0<B_1<y$ according to the order on $M(B)$ and for any $z\in B_0$, $z'\in B_1\cup\{y\}$, we have $d(z,z')=d$. Define the same way $C',C_0$ and $C_1$, so that $x<C_0<C_1<y$.
	
	We explain how to amalgamate $B_0$ and $C_0$. Let $b\in B_0$ and $c\in C_0$. If $d(x,b)\leq d(x,c)$, set $b<c$ and $d(b,c)=d(x,c)$. If $d(x,c)<d(x,b)$, then set $c<b$ and $d(c,b)=d(x,b)$. We amalgamate $B_1$ and $C_1$ similarly by considering $d(\cdot,y)$: if $d(b,y)\leq d(c,y)$, set $c<b$ and $d(c,b)=d(c,y)$ and symmetrically if $d(c,y)<d(b,y)$. Finally, if $b\in B_0\cup C_0$ and $c\in B_1\cup C_1$, then set $b<c$ and $d(b,c)=d$.
Now, if $b\in M(B)$ and $c\in M(C)$ have a point $x\in M(A)$ strictly between them, say $b<x<c$, set $b<c$ and $d(b,c)=\max\{d(b,x),d(x,c)\}$. One can check that this does not depend on the choice of $x$. Finally, if $b,c$ and both greater than the greatest point $x$ of $M(A)$, then one amalgamates them with the same rules as for $B_0$ and $C_0$ above (and symmetrically if $b,c$ and smaller than the minimum of $M(A)$). It is now straightforward to check that this does define an ordered ultrametric space $C$ as required.
\end{proof}

\subsection{Trees and betweenness relation}

By a meet-tree, we mean a partially ordered set $(T,\leq)$ such that, for every element $a\in T$, the set $\{x\in T: x\leq a\}$ is linearly ordered by $\leq$ and every two-element subset $\{a,b\}$ of $T$, has a greatest lower bound $a\wedge b$. This theory has a model companion: the theory $T_{d}$ of dense meet-trees which is $\omega$-categorical and admits elimination of quantifiers in the language $\{\leq, \wedge\}$.

Given a meet-tree $(T,\leq,\wedge)$, we define the tree-betweenness relation $B(x,y,z)$ which holds for $(a,b,c)\in T^3$ if and only if $b$ is in the path linking $a$ to $c$, that is: \[ B(a,b,c) \iff ((a\wedge c) \leq b \leq a)\vee ((a\wedge c)\leq b\leq c). \]

By an interval of $(T,B)$, we mean a set of the form $[a,b]:=\{x\in T : B(a,x,b)\}$ for some $a,b\in T$. Such an interval is equipped with a natural definable linear ordering where $c$ is less than $d$ if $B(a,c,d)$ (if we represent the interval as $[b,a]$ instead of $[a,b]$, we obtain the opposite ordering, with $b$ as the first point). 
Given three distinct points $a,b,c\in T$, the intersection $[a,b]\cap [b,c] \cap [a,c]$ has a unique point;
call this point the {\em meet} of $(a,b,c)$ and denote it by $\bigwedge (a,b,c)$. If some two of $a,b,c$ are equal, then define $\bigwedge(a,b,c)$ as being equal to their common value. Fix some $a\in T$. Given $b,c\in T$, write $b\mathrel{E_a}c$ if $a \notin [b,c]$. Then $E_a$ is an equivalence relation on $T\setminus \{a\}$. The equivalence classes of $E_a$ are called \emph{cones at $a$}.

For the purposes of this example, a (tree) betweenness structure is the reduct to $(B,\bigwedge)$ of a meet-tree $(T,\leq, \wedge)$. If $(T,\leq ,\wedge)\models T_d$, then the reduct $(T,B,\bigwedge)$ to the betweenness relation admits elimination of quantifiers in the language $(B, \bigwedge)$ and its automorphism group acts transitively on pairs of distinct elements.

We note some basic properties of trees and betweenness relations (the reader is encouraged to make drawings to follow the statements).

\begin{itemize}
	\item If $(T,B,\bigwedge)$ is a tree betweenness structure and $a\in T$ is any point, then one can define a meet-tree structure with $a$ as minimal element by setting $b\leq c \iff B(a,b,c)$. Taking the betweenness relation associated to this tree yields back $B$.
	\item Let $(T,B,\bigwedge)$ be a tree betweenness structure and let $T_0\subset T$ be a finite set. Then $\{a\in T:  a=\bigwedge(c,d,e)$ for some $c,d,e\in T_0\}$ is a substructure of $T$. In particular, the substructure generated by $T_0$ has size at most $|T_0|^3$.
	\item Let $T_0 \subset T$ be a finite \emph{substructure} and $a\in T$ a point. The possibilities for $\qftp(a/T_0)$ are as follows:
	\begin{enumerate}
		\item $T_0\cup \{a\}$ is a substructure. This splits into three cases:
		\begin{enumerate}
			\item $a\in T_0$;
			\item $a$ lies in an interval of $T_0$: there are $b,c\in T_0$ such that $a\in [b,c]$. The knowledge of a minimal such interval completely determines $\qftp(a/T_0)$.
			\item $a$ does not lie in an interval of $T_0$, then there is a unique $b\in T_0$ such that $[a,b]$ contains no other point of $T_0$ and for every $c\in T_0$, $b\in [a,c]$ ($b$ is the point in $T_0$ \emph{closest} to $a$). The knowledge of $b$ determines $\qftp(a/T_0)$.
		\end{enumerate}
		\item $T_0 \cup \{a\}$ is not a substructure. Then there are $b,c\in T_0$ such that $a_* := \bigwedge(a,b,c)$ is a new element. Then $a_* \in [b,c]$, so its type belongs to the case (1b) above. Also $\qftp(a/T_0a_*)$ belongs to case (1c) above, with $a_*$ being the closest element to $a$ in the tree $T_0\cup \{a_*\}$. In particular $a_* \in [b,a]$ for every $b\in T_0$ and $T_0 \cup \{a,a_*\}$ is a substructure.

		In this case, $\qftp(a/T_0)$ is determined by $\qftp(a_*/T_0)$.
	\end{enumerate}
	Note that it follows from this analysis that there are at most $|T_0|+|T_0|^2+|T_0|+|T_0|^2= O(|T_0|^2)$ quantifier-free one types over $T_0$.
\end{itemize}

\subsection{The example}

We work in a language $L$ with two sorts $ M$ and $ D$. In addition to those sorts, we have in our language a ternary relation $B(x,y,z)$ and a ternary function symbol $\bigwedge(x,y,z)$, all on the main sort $ M$, a function $d: M^2\to D$, a binary relation $\unlhd$ on $ D$ and a constant $0$ of sort $ D$.
Let $\mathcal C_0$ be the class of finite $L$-structures $(M,D;B,\bigwedge,d,\unlhd,0)$, where:
\begin{itemize}
	\item $B$ is a tree-betweenness relation on $M$ and $\bigwedge$ is the meet in that structure;
	\item $\unlhd$ is a linear order on $D$ with minimal element $0$;
	\item $d: M^2\to  D$ is such that for any two elements $x,y\in  M$, the interval $[x,y]$ equipped with $d$ is an ordered ultrametric space as described above (this condition is invariant under reversing the ordering on $[x,y]$).
\end{itemize}
Note that the third bullet implies that $d$ is an ultrametric distance on $ M$: for any $x,y,z\in M$, $d(x,y)\leq \max \{d(y,z),d(x,z)\}$.

\begin{lem}
		The class $\mathcal C_0$ is a Fra\" iss\'e class.

\end{lem}
\begin{proof}
	 Tree-betweenness relations are closed under substructures, as follows for instance from the first bullet above. Hence $\mathcal C_0$ is a universal class. 
	
As previously, we can skip joint embedding and only check amalgamation. Assume $E,F,G\in \mathcal C_0$ are substructures with $E\subseteq F$ and $E\subseteq G$. Without loss, $M(F)$ and $M(G)$ are not empty. We first amalgamate the sorts $D$ freely as in the case of ordered ultrametric spaces. We then amalgamate the main sorts. If $M(E)$ is empty, then arbitrarily fix point $a\in M(F)$ and $a'\in M(G)$ and identify them, adding a point to $M(E)$ that maps to both. So we can assume that $M(E)$ is non-empty. Next, it is enough to consider the case where $M(F)$ is generated over $M(E)$ by a single element $a$ and similarly $M(G)$ is generated over $M(E)$ by a single element $b$ (then induct on the size of $M(F)\setminus M(E)$ and $M(G)\setminus M(E)$).

We will use the list of (quantifier-free) types in betweenness relations presented above. If either $a$ or $b$ is of class (2), define $a_*$ and $b_*$ as done there. Otherwise, set $a_*=a$ (resp. $b_*=b$). Assume that the types of both $a_*$ and $b_*$ over $M(E)$ are of class (1b) in that list and both land in the same minimal interval $[x,y]$ of $M(E)$. Then that interval, with distance $d$ is an ordered ultrametric space, hence we can amalgamate those points as explained above. Assume next that both types of $a_*$ and $b_*$ are of class (1c) and have the same closest point $x$ in $M(E)$. Amalgamate $a_*$ and $b_*$ over $M(E)$ so that $\neg(a_* \mathrel{E_{x}}b_*)$ holds and set $d(a_*,b_*)=\max\{d(a_*,x),d(b_*,x)\}$. In all other cases, there is a unique way to extend the tree structure to $M(E)\cup \{a_*,b_*\}$. Having done this, there is $x\in M(E)$ in the interval $[a_*,b_*]$ and we can (and must) set $d(a_*,b_*)=\max_{\unlhd}\{d(a_*,x),d(b_*,x)\}$.

Let $E'$ be the resulting structure. Then $M(F)$ and $M(G)$ are generated over $E'$ by $a$ and $b$ respectively and the types of $a$ and $b$ over $M(E')$ are of class (1). We can therefore repeat the procedure to amalgamate them. This finishes the proof.

\end{proof}

Let $\mathcal U$ be the Fra\" iss\'e limit of $\mathcal C_0$ and $\C$ a monster model of it. It is an $\omega$-categorical structure which admits elimination of quantifiers in the language $L$. 

\begin{lem}
	The structure $\mathcal U$ is NIP.
\end{lem}
\begin{proof}
Recall that a theory is NIP if and only if for every formula $\phi(x,\bar y)$, there is $k$ such that the number of $\phi$-types over any finite set $A$ is bounded by $|A|^k$. We will in fact check that there are polynomially many 1-types over finite sets. So let $A\subset \C$ be a finite set and $c\in \C$. As the subtree generated by $M(A)$ has size polynomial in $A$, we can assume that $M(A)$ is a subtree.  Then we can close $A$ under the distance function and hence assume that $A$ is a substructure.

If $c$ is of sort $D$, then by quantifier elimination, there are at most $2|D(A)|+1$ possibilities for $\tp(c/A)$.

Next assume that $c$ is in the main sort $M$. The subtree generated by $c$ over $M(A)$ has size either $|M(A)|+1$ or $|M(A)|+2$. If the type of $c$ over $M(A)$ is of class (2) above, then $\tp(c/M(A))$ is determined by $\tp(c_*/M(A))$ and $\tp(c/M(A)c_*)$. Each of those types is of class (1) and we can therefore reduce to this case and assume that $M(A)\cup \{c\}$ is already a subtree. Let $D'$ be the set of distances of pairs of elements in $M(A)\cup \{c\}$. Then $D'\setminus D$ contains at most one element. By the previous paragraph, there are polynomially many possibilities for the type of that element over $D$. Adding that element to $A$, we can assume that for all $a,b\in M(A)\cup \{c\}$, $d(a,b)\in D(A)$.

We know from the analysis of types above that there are only polynomially many 1-types over $A$ in the reduct to the tree. As for the distance: take $a\in M(A)$ for which $d(a,c)$ is minimal. Then the data of $a$ and $d(a,c)$ determines all distances $d(b,c)$, $b\in M(A)$, using the ultrametric identity. Indeed, if $d(a,b)\triangleright d(a,c)$, then $d(b,c)=d(a,b)$ and otherwise $d(b,c)=d(a,c)$.

This discussion shows that there are polynomially many possibilities for $\tp(c/A)$ and we conclude that $\mathcal U$ is NIP.
\end{proof}

Note that by quantifier elimination (and $\omega$-categoricity), the automorphism group of $\mathcal U$ acts transitively on $D(\mathcal U)\setminus \{0\}$ and on pairs of distinct points of $M(\mathcal U)$. It follows that for each $a\in M(\mathcal U)$, the stabilizer of $a$ acts transitively on the cones at $a$. Let $\mathcal B$ be the imaginary sort of $d$-balls of the form $\{x\in M: d(x,a)<r\}$ for some $a\in M$ and $r\in D\setminus \{0\}$. Let $\in$ be the definable binary relation in $M\times \mathcal B$ representing membership of an element in a ball and let $r:\mathcal B \to D$ give the radius of a ball (definable as $r(\mathfrak b)=\max_{\unlhd}\{d(a,b):a,b\in \mathfrak b\}$). There is a unique type $p(x)$ over $\emptyset$ of an element of $\mathcal B$ (since the automorphism group of $\mathcal U$ acts transitively on pairs $(a,r)\in M(\mathcal U)\times D(\mathcal U)$).

By quantifier elimination, there is a unique type $r_0(t)$ in $S(\C)$ of sort $D$ containing $t\triangleright d$ for all $d\in D(\C)$. Pick any $a\in M(\C)$ and let $p_0(x)$ be the type over $\C$ of a ball $\mathfrak b$ containing $a$ with radius $r(\mathfrak b)\models r_0$. This does not depend on the choice of $a$ since such a ball contains all points in $M(\C)$. The type $p_0$ is thus invariant under $\aut(\C)$.

Let $\mathcal Q$ be the set of global extensions $q(x)$ of $p(x)$ which satisfy $r(x)\triangleleft d$ for some $d\in D(\C)\setminus \{0\}$. Pick a type $q\in \mathcal Q$ and $d\in D(\C)\setminus \{0\}$ such that $q\vdash r(x)\triangleleft d$. Let $\mathfrak b\models q$. Assume first that the ball coded by $\mathfrak b$ contains a point $a\in M(\C)$. One can find unboundedly many points in $\C$ at distance at least $d$ from each other. For each $\sigma \in \aut(\C)$, the ball coded by $\sigma(q)$ can contain at most one of those points. Since the automorphism group of $\C$ acts transitively on the main sort, $q$ has unbounded $\aut(\C)$-orbit. If now $\mathfrak b$ does not contain a point in $M(\C)$, pick some $a\in M(\C)$. Then all elements in the ball coded by $\mathfrak b$ lie in the same cone at $a$ and therefore $\mathfrak b$ determines a cone at $a$. If this cone contains an $M(\C)$-point, then $q$ has unbounded $\aut(\C)$-orbit, as there are unboundedly many cones at $a$ and $\aut(\C/a)$ acts transitively on them. Assume that this is not the case and take $a'\in M(\C)$, $a'\neq a$. Then for any point $d$ in the ball $\mathfrak b$ (in some larger monster model), $a$ lies in the interval $[a',d]$. It follows that the ball $\mathfrak b$ is included in the cone at $a'$ defined by $a$ and we are reduce to the previous case by changing $a$ to $a'$. Hence, in all cases, $q$ forks over $\emptyset$. Let $d_0\in D(\C)\setminus\{0\}$. The formula $r(x)\triangleleft d_0$ is weakly invariant as there is $q\in \mathcal Q$ with $q(x)\vdash r(x)\triangleleft d$ for all $d\in D(\C)\setminus \{0\}$, but that formula forks over $\emptyset$. By Corollary \ref{corollary: bounded = weakly invariant do not fork}, the minimal ideal of $EL(S_p(\C))$ is not bounded. However, $p$ does not fork over $\emptyset$ since $p_0$ is an invariant extension of it.

\section*{Acknowledgments}
	
We would like to thank to the referee for very careful reading and all the comments and suggestions which helped us to improve presentation.

\end{document}